\documentclass[oneside,english,a4wide]{amsart}
\usepackage[latin9]{inputenc}
\usepackage[a4paper]{geometry}
\usepackage{mathabx}
\usepackage{esvect}
\usepackage{babel}
\usepackage{mathtools}
\usepackage{amstext}
\usepackage{amsthm}
\usepackage{amssymb}
\usepackage{enumerate}
\usepackage{esint}
\usepackage{graphicx}
\usepackage{bbm}
\usepackage[all]{xy}
\usepackage{mathrsfs}
\usepackage[unicode=true,
bookmarks=true,bookmarksnumbered=true,bookmarksopen=true,bookmarksopenlevel=2,
breaklinks=false,pdfborder={0 0 1},backref=false,colorlinks=false]
{hyperref}
\hypersetup{colorlinks,linkcolor=blue,anchorcolor=blue,citecolor=blue}

\usepackage[svgnames]{xcolor} 

\usepackage{graphicx, tikz-cd} 
\usepackage{cite} 

\usepackage[most]{tcolorbox}
\newtcolorbox{maintheorembox}{
  enhanced,
  breakable,
  colback=white,
  colframe=black,
  boxrule=0.8pt,
  arc=0pt,
  left=6pt,
  right=6pt,
  top=6pt,
  bottom=6pt
}

\makeatletter
\numberwithin{equation}{section}
\numberwithin{figure}{section}
\theoremstyle{plain}
\newtheorem{thm}{\protect\theoremname}[section]
\theoremstyle{plain}

\newtheorem{proposition}[thm]{Proposition}

\theoremstyle{definition}

\theoremstyle{plain}

\newtheorem{lemma}[thm]{Lemma}
\theoremstyle{remark}

\theoremstyle{definition}
\newtheorem{definition}[thm]{Definition}
\newtheorem*{claim*}{Claim}

\theoremstyle{remark}
\newtheorem{remark}[thm]{Remark}
\theoremstyle{definition}
\newtheorem*{defn*}{\protect\definitionname}

\newtheorem*{maintheorem}{Main Theorem}

\usepackage{babel}

\makeatother

\providecommand{\definitionname}{Definition}
\providecommand{\lemmaname}{Lemma}
\providecommand{\propositionname}{Proposition}
\providecommand{\remarkname}{Remark}
\providecommand{\theoremname}{Theorem}

\newcommand{\Rmnum}[1]{\uppercase\expandafter{\romannumeral #1}}
\makeatother

\newcommand{\real}{{\mathbb R}}
\newcommand{\nat}{{\mathbb N}}

\newcommand{\com}{{\mathbb C}}
\newcommand{\ce}{{\mathbb E}}

\newcommand{\T}{{\mathcal T}}
\newcommand{\A}{{\mathcal A}}

\newcommand{\F}{{\mathcal F}}

\newcommand{\N}{{\mathcal N}}

\newcommand{\R}{{\mathcal R}}

\newcommand{\Q}{{\mathcal Q}}

\newcommand{\e}{\varepsilon}
\newcommand{\f}{\varphi}

\newcommand{\wt}{\widetilde}
\newcommand{\wh}{\widehat}

\newcommand{\be}{\begin{eqnarray*}}
	\newcommand{\ee}{\end{eqnarray*}}
\newcommand{\beq}{\begin{equation}}
	\newcommand{\eeq}{\end{equation}}
\newcommand{\beqn}{\begin{equation*}}
	\newcommand{\eeqn}{\end{equation*}}
\newcommand{\bs}{\begin{split}}
	\newcommand{\es}{\end{split}}
\newcommand{\W}{{\mathcal W}}
\renewcommand{\H}{{\mathcal H}}

\newcommand{\K}{{\mathcal K}}

\newcommand\norm[1]{ \left\| #1 \right\| }

\newcommand\sk[1]{ \left( #1 \right) }

\newcommand\lk[1]{ \left\{ #1 \right\} }
\newcommand{\lara}[1]{\left\langle #1 \right\rangle}

\renewcommand\d{{\rm d}}

\begin{document}
	\pagestyle{myheadings}
	\markboth{Noncommutative PI}{}	
	
	
	\title[Dimension-free Riesz transform estimates on biased hypercubes]{Dimension-free Riesz transform estimates on biased hypercubes via non-tracial baby Fock models}
	
	
	\author[Hun Hee Lee]{Hun Hee Lee}
	\address{Department of Mathematical Sciences and the Research Institute of Mathematics, Seoul National University, Gwanak-ro 1, Gwanak-gu, Seoul 08826, Republic of Korea}
	\email{hunheelee@snu.ac.kr}

	\author[Zhendong  Xu]{Zhendong Xu}
	\address{
		Department of Mathematical Sciences and the Research Institute of Mathematics, Seoul National University, Gwanak-ro 1, Gwanak-gu, Seoul 08826, Republic of Korea}
	\email{zhendong\_xu\_97@snu.ac.kr}
	
	\author[Sang-Gyun Youn]{Sang-Gyun Youn}
	\address{Department of Mathematics Education and NextQuantum Innovation Research Center, Seoul National University, Gwanak-ro 1, Gwanak-gu, Seoul 08826, Republic of Korea}
	\email{s.youn@snu.ac.kr}

	\author[Hao Zhang]{Hao Zhang}
	\address{
		Instituto de Ciencias Matem\'{a}ticas, Consejo Superior de Investigaciones Cient\'{i}ficas, C/ Nicol\'{a}s Cabrera 13-15, 28049, Madrid, Spain}
	\email{hao.zhang@icmat.es}
	
	\date{}

	\begin{abstract}

We study dimension-free Riesz transform estimates on hypercubes equipped with non-uniform product measures. For $2\leq p<\infty$, we prove a dimension-free Meyer inequality for the associated carr\'e du champ, with constants independent of both the dimension and the product measure. In contrast to the uniform case, the carr\'e-du-champ square function does not in general coincide with the unweighted Riesz square function. Under an additional boundedness assumption on the weight parameters, we further obtain dimension-free estimates for the latter.

Our proof is based on a transference from the biased hypercube to a non-tracial baby Fock model. We introduce new annihilation operators and the associated Riesz transforms on the biased hypercube, and derive explicit formulas relating them to the Riesz transforms on the non-tracial baby Fock model. The main analytic ingredient is a dimension-free $L_p$-estimate for the resulting noncommutative Riesz transform, proved by a rotation argument in the spirit of Pisier and Lust-Piquard. The same analytic mechanism also yields dimension-free Riesz transform estimates on $q$-Araki--Woods algebras.
    
\end{abstract}
	
	\maketitle
	
	\tableofcontents{}
	
\section{Introduction}\label{Introduction}

Let $T_t=e^{-tA}$, $t\geq 0$, be a symmetric Markov semigroup on a measure space $(\Sigma,\psi)$, with a positive infinitesimal generator $A$. The carr\'e du champ associated with $(T_t)_{t\geq 0}$ is defined by
\[\Gamma(f,g)=\frac12\left(\overline{f}A(g)+\overline{A(f)}g-A(\overline{f}g)\right).\]
A fundamental question in this setting is {\it Meyer's problem}: for $1<p<\infty$, under what conditions do we have
\[\|A^{1/2}f\|_{L_p(\Sigma,\psi)}\approx_p\|\Gamma(f,f)^{1/2}\|_{L_p(\Sigma,\psi)},\]
with constants independent of the dimension of the underlying model?

A basic example is the classical heat semigroup on $\mathbb{R}^n$, whose infinitesimal generator is
\[N=\sum_{j=1}^n\partial_j^*\partial_j=-\Delta.\]
In this case, the {\it carr\'e du champ} with respect to $N$ is given by $\displaystyle \Gamma(f,f)=\sum_{j=1}^n|\partial_jf|^2$. Consequently, the square function is given by
\[\Gamma(N^{-1/2}f,N^{-1/2}f)^{1/2}=\left(\sum_{j=1}^n|\partial_jN^{-1/2}f|^2\right)^{1/2},\]
where $\partial_jN^{-1/2}$ is the $j$-th classical Riesz transform. The results of Stein \cite{Ste1970} and Gundy--Varopoulos \cite{GV1979} yield the dimension-free estimate
\[\left\|\Gamma(N^{-1/2}f,N^{-1/2}f)^{1/2}\right\|_{L_p(\mathbb{R}^n)}\approx_p\|f\|_{L_p(\mathbb{R}^n)},\]
where the implicit constants depend only on $p$ and not on $n$.

Meyer subsequently established analogous estimates for the Ornstein--Uhlenbeck semigroup on Gaussian spaces \cite{Meyer1975,Meyer1984}. Pisier \cite{P1986} gave an analytic proof based on Gaussian rotations and the boundedness of the one-dimensional Hilbert transform. We refer to \cite{Bak1985,Bak1987,L2008,JM2010} for further developments concerning Meyer inequalities and Riesz transforms associated with Markov semigroups.

Dimension-free Riesz transform estimates have also been studied in several discrete and noncommutative settings. Lust-Piquard \cite{LP1998} proved such estimates on the discrete hypercube equipped with the uniform measure.
An important feature of Lust-Piquard's argument is that the commutative hypercube estimates are obtained through a noncommutative spin model, using a transference method inspired by Pisier's rotation argument. Further results for products of abelian groups and mixed $q$-Gaussian models were obtained in \cite{LP1999,LP2004}. Recently, the first non noncommutative approach via Bellman function to discrete Riesz transforms on hypercubes was discovered in
\cite{DIPV2026}. For noncommutative aspects, Junge, Mei and Parcet \cite{JMP2018} developed a general framework for Riesz transforms on group von Neumann algebras associated with $1$-cocycles. We also refer to \cite{JM2010,GPPGR2024,GPX2022,LX2025, DPS2025} for related results in noncommutative harmonic analysis.

The connection between discrete Riesz transforms and noncommutative models is also relevant beyond harmonic analysis. In particular, Naor \cite{NA2016} used Lust-Piquard's estimates on the hypercube in the study of metric $X_p$ inequalities and Banach-space embeddings. This provides part of the motivation for developing Riesz transform estimates on discrete
probability spaces with less symmetry than the classical uniform cube.

\medskip

\noindent
\textbf{Commutative model.}
Our main results concern the biased hypercube $\Omega_n=\{-1,1\}^n$ equipped with a product probability measure $\nu=\nu_1\times\cdots\times\nu_n$, where
\[\nu_j(\{-1\})=\frac{\mu_j^2}{\mu_j^{-2}+\mu_j^2},
\qquad \nu_j(\{1\}) =\frac{\mu_j^{-2}}{\mu_j^{-2}+\mu_j^2}, \qquad \mu_j>0.\]

Most of the classical theory of discrete Riesz transforms concerns hypercubes equipped with the uniform measure. Functional inequalities on such spaces have important applications in the local and nonlinear theory of Banach spaces, including the Ribe program, as well as in theoretical computer science; see, for example, \cite{Ball2012,E2019,IRV2020,NA2012,NA2016,NA2018}. There has also been increasing interest in biased and non-homogeneous product measures, including hypercontractive inequalities \cite{O2003,W2007}, vector-valued concentration inequalities \cite{GM2026}, and the work of Eskenazis on functional and metric inequalities for non-uniform hypercubes \cite{E2023,E2024}.

Let $[n]=\{1,2,\ldots,n\}$. For $j\in [n]$, define the normalized coordinate function $\zeta_j:\Omega_n\to\mathbb{R}$ by
\[\zeta_j(\omega)=\begin{cases}
\mu_j^2, & \omega_j=1,\\
-\mu_j^{-2}, & \omega_j=-1.
\end{cases}\]
Set $\zeta_{\varnothing}=1$ and $\zeta_A=\prod_{j\in A}\zeta_j$ for any non-empty subset $A\subseteq[n]$. Then $\{\zeta_A\}_{A\subseteq[n]}$ is an orthonormal basis of $L_2(\Omega_n,\nu)$.

With respect to this basis, we introduce new annihilation operators by
\[D_j(\zeta_A)=\begin{cases}
\zeta_{A\setminus\{j\}}, & j\in A,\\
0, & j\notin A.
\end{cases}\]
and define the $j$-th Riesz transform by \[ R_j =\frac{\mu_j^{-2}+\mu_j^2}{2}D_jN^{-1/2}. \]
When $\mu_j=1$ for every $j$, this reduces to the usual Riesz transform on the uniform hypercube.

The number operator is $\displaystyle N=\sum_{j=1}^nD_j^*D_j$, and the Ornstein--Uhlenbeck semigroup associated with the product measure $\nu$ is $P_t=e^{-tN}$. The same semigroup $P_t$ was studied by Eskenazis in \cite{E2023,E2024}, but we provide a first-order realization of its coordinate generators through the annihilation--creation factorization using $D_j$. This realization makes it possible to relate the resulting Riesz transforms $R_j$ explicitly to the Riesz transforms $R_{\pm j}^{\e}$ on a non-tracial baby Fock model.

The carr\'e du champ associated with $N$ is given by
\[\Gamma(f,g) = \sum_{j=1}^n \frac{1+\zeta_j^2}{2} \cdot \overline{D_jf}\,D_jg.\]
See Proposition \ref{hz:propA1} in Appendix \ref{appendix-A} for detailed calculations. Consequently,
\[\Gamma(N^{-1/2}f,N^{-1/2}f)^{1/2}=\left(\sum_{j=1}^n\frac{2(1+\zeta_j^2)}{(\mu_j^{-2}+\mu_j^2)^2} |R_jf|^2 \right)^{1/2}.\]
Unlike the uniform case, the carr\'e-du-champ square function $\Gamma(N^{-1/2}f,N^{-1/2}f)^{1/2}$ does not coincide with the unweighted Riesz square function $\displaystyle \left(\sum_{j=1}^n|R_jf|^2 \right)^{1/2}$. Thus, in the biased setting, the usual Riesz-transform formulation of Meyer's problem naturally leads to the following two questions.

\begin{enumerate}
    \item Does the intrinsic Meyer estimate
    \[\left\|\Gamma(N^{-1/2}f,N^{-1/2}f)^{1/2}\right\|_{L_p(\Omega_n,\nu)}\approx_p\|f\|_{L_p(\Omega_n,\nu)}\]
    hold for mean-zero $f$, with constants independent of the dimension and of the product measure?
    \item Under what conditions do we have the dimension-free estimate
    \[\left\|\left(\sum_{j=1}^n|R_jf|^2\right)^{1/2}\right\|_{L_p(\Omega_n,\nu)}\approx_p\left\|\Gamma(N^{-1/2}f,N^{-1/2}f)^{1/2}\right\|_{L_p(\Omega_n,\nu)}\]
    for the unweighted Riesz square function?
\end{enumerate}

Our strategy has two main components: a \emph{transference} from the biased hypercube to a non-tracial baby Fock model, followed by a \emph{dimension-free $L_p$-estimate} for the corresponding noncommutative Riesz transform $R_{\pm j}^{\e}$ on $\Gamma_{n,\e}$.

\medskip

\noindent
\textbf{Transference to the noncommutative model.} In Section \ref{sec-transference}, we develop a transference method of Lust-Piquard in a non-tracial setting for the Riesz transforms. We construct a state-preserving injective $*$-homomorphism
\[\Phi:L_{\infty}(\Omega_n,\nu)\longrightarrow \Gamma_{n,\varepsilon}\subseteq \Gamma_{n,\varepsilon}^{(2)},\]
where $\Gamma_{n,\varepsilon}^{(2)}$ is a doubled baby Fock model.

More importantly, we derive explicit formulas relating the biased-hypercube Riesz transforms $R_j$ to the baby Fock Riesz transforms $R_j^{\varepsilon}$ and $R_{-j}^{\varepsilon}$ inside $\Gamma_{n,\varepsilon}^{(2)}$. Combined with noncommutative Khintchine inequalities, these formulas allow us to derive the biased-hypercube estimates from a dimension-free $L_p$-estimate on the non-tracial baby Fock model.

Thus, after the transference step, the main analytic task is no longer commutative: it is to establish a dimension-free $L_p$-estimate for the corresponding noncommutative Riesz transform. This is the second main component of the proof.

\medskip

\noindent
\textbf{Dimension-free $L_p$-estimate on the noncommutative model.}
Let $(\Gamma_{n,\varepsilon},\tau_{n,\varepsilon})$ be the non-tracial baby Fock model. In Section \ref{sec:baby Fock}, passing to the doubled noncommutative probability space $(\Gamma_{n,\varepsilon}^{(2)},\tau_{n,\varepsilon}^{(2)})$, we define a Riesz transform $\mathcal{R}_n: \Gamma_{n,\varepsilon}^0 \rightarrow \Gamma_{n,\varepsilon}^{(2)}$, and denote its $L_p$-extension by $\mathcal{R}_n^{(p)}$. Our main noncommutative estimate is
\[\|\mathcal{R}_n^{(p)}(X)\|_{L_p(\Gamma_{n,\varepsilon}^{(2)},\tau_{n,\varepsilon}^{(2)})}
\approx_p \|X\|_{L_p(\Gamma_{n,\varepsilon},\tau_{n,\varepsilon})},\]
for every $1<p<\infty$ and every mean-zero $X\in L_p(\Gamma_{n,\varepsilon})$, with constants independent of the dimension; see Theorem \ref{Riesz-Baby-Fock}.

This dimension-free $L_p$-estimate is the analytic core of the argument. Its proof follows the rotation philosophy of Pisier and Lust-Piquard. The doubled space carries a canonical one-parameter group $(\widetilde{\alpha}_{\theta})_{\theta\in\mathbb{R}}$ of state-preserving $\ast$-automorphisms, which plays the role of the rotation in their arguments. We present a kernel representation of $\mathcal{R}_n$, which yields the decomposition
$\mathcal{R}_n=\widetilde{\Pi}^{(n)}\circ T_n$ in the notation used in the proof. The $L_p$-estimate for $T_n$ follows from the boundedness of the vector-valued Hilbert transform, while the $L_p$-boundedness of $\widetilde{\Pi}^{(n)}$ follows from Stein's complex interpolation theorem.

\medskip

\noindent
\textbf{Returning to the biased hypercube.}
In Section \ref{sec-main-theorem}, we transfer the preceding $L_p$-estimate back to the commutative model. The main tool for recovering the commutative dimension-free $L^p$-estimate is the {\it noncommutative Khintchine inequality}. Combined with the explicit transference formulas and the dimension-free estimate for $\mathcal{R}_n^{(p)}$, it yields our main commutative results.

\begin{maintheorembox}
\begin{maintheorem}
For any $2\leq p<\infty$ and mean-zero $f\in L_p(\Omega_n,\nu)$,
\[\left\|\Gamma(N^{-1/2}f,N^{-1/2}f)^{1/2}\right\|_{L_p(\Omega_n,\nu)}\approx_p \|f\|_{L_p(\Omega_n,\nu)},\]
with constants independent of $n$ and of the product measure $\nu$. Moreover, if
\begin{equation}\label{hz:optimal}
M=\sup_{j\geq1}\max\{\mu_j,\mu_j^{-1}\}<\infty,
\end{equation}
then we obtain
\[\left\|\left(\sum_{j=1}^n|R_jf|^2\right)^{1/2}\right\|_{L_p(\Omega_n,\nu)}\approx_{p,M}\|f\|_{L_p(\Omega_n,\nu)}.\]
\end{maintheorem}
\end{maintheorembox}

The first assertion gives an affirmative answer to Question 1, and hence to Meyer's problem for the biased Ornstein--Uhlenbeck semigroup considered above, in the range $2\leq p<\infty$. The second assertion answers Question 2 under the boundedness condition on the weight parameters $\mu_j$ and $\mu_j^{-1}$. A counterexample given at the end of Section \ref{sec-main-theorem} shows that the unweighted square-function estimate can fail in the absence of such control on the weight parameters. This indicates that condition \eqref{hz:optimal} is close to optimal.

\medskip

\noindent
\textbf{Further noncommutative extension.}
The rotation-based analytic mechanism developed for the baby Fock model also extends to the setting of $q$-Araki--Woods algebras. In Section \ref{sec-q-Araki-Woods}, we introduce a Riesz transform on a $q$-Araki--Woods algebra through a gradient on the corresponding doubled $q$-Fock space. The gradient arises from the infinitesimal action of a rotation, and the resulting Riesz transform admits a kernel representation analogous to that in the baby Fock setting. Similarly, applying the $L_p$-boundedness of the Hilbert transform and Stein's complex interpolation theorem, we obtain a dimension-free $L_p$-estimate for this Riesz transform for $1<p<\infty$; see Theorem \ref{thm-main2}. This result is independent of the biased hypercube problem, but shows that the same analytic mechanism underlying the baby Fock estimate extends to another non-tracial setting.

\medskip

\noindent

The framework developed here suggests several further questions, including Banach-space-valued inequalities, Pisier-type and Talagrand-type inequalities on biased hypercubes, and possible applications to metric and Banach-space embedding problems. Some of these directions will be considered in future work \cite{XZZ2026}.

\section{Preliminaries}

\subsection{Biased hypercube}\label{sec:prelim-weighted}

We recall some basics on a hypercube equipped with a non-homogeneous product measure. Let $\Omega_n=\{-1,1\}^n$ and let $\nu=\nu_1\times\cdots\times\nu_n$ be the product probability measure determined by 
\[\nu_j(\{-1\})=\frac{\mu_j^2}{\mu_j^{-2}+\mu_j^2}, \qquad\nu_j(\{1\})=\frac{\mu_j^{-2}}{\mu_j^{-2}+\mu_j^2}, \qquad \mu_j>0.\]

For any $j\in [n]$, define $z_j:\Omega_n\rightarrow \{\pm 1\}$ by $z_j(\omega)=\omega_j$ for all $\omega=(\omega_1,\ldots,\omega_n)\in\Omega_n$, and set $z_\emptyset=1$, $z_A=\prod_{j\in A}z_j$ for all $A\subseteq[n]$. Then $\{z_A\}_{A\subseteq[n]}$ is a linear basis of $L_\infty(\Omega_n,\nu)$.

With respect to the non-uniform product probability measure $\nu$, the family $\left\{z_A: A\subseteq [n]\right\}$ is not an orthonormal basis, so we consider the following normalized coordinate function $\zeta_j:\Omega_n\to\mathbb{R}$ by
\[\zeta_j(\omega)=\begin{cases}
\mu_j^2, & \omega_j=1,\\
-\mu_j^{-2}, & \omega_j=-1
\end{cases}\]
instead. We define $\zeta_{\varnothing}=1$ and $\zeta_A=\prod_{j\in A}\zeta_j$ for any non-empty subset $A\subseteq[n]$. Then $\{\zeta_A\}_{A\subseteq[n]}$ is an orthonormal basis of $L_2(\Omega_n,\nu)$.

For $j\in[n]$, let $E_j$ denote the marginal conditional expectation with respect to the $j$-th coordinate, that is,
\[E_jf(\omega)=\int_{\{-1,1\}}f(\omega_1,\ldots,\omega_{j-1},y,\omega_{j+1},\ldots,\omega_n)\,d\nu_j(y).\]
In particular, for any $j\in [n]$, we have
\begin{equation}\label{eq204}
E_j(\zeta_A)=\begin{cases} 0,& j\in A,\\ \zeta_A,& j\notin A. \end{cases}
\end{equation}

The positive generator of the corresponding discrete heat semigroup is $N=\sum_{j=1}^n(I-E_j)$, and the associated semigroup is $P_t=e^{-tN}$ with $t\geq0$. This semigroup for general non-homogeneous product measures was studied by Eskenazis in \cite{E2023,E2024}. In particular, when $\mu_j=1$ for all $j$, the construction reduces to the classical uniform hypercube.

\subsection{Baby Fock model}\label{sec:prelim-baby}

We next recall the twisted baby Fock model introduced by Nou \cite{NA2006}, following the formulation and structural properties developed in \cite{LR2011}. Let $[\pm n]=\{\pm1,\pm2,\ldots,\pm n\}$, and let $\e:[\pm n]\times [\pm n]\longrightarrow\{-1,1\}$ be a choice of signs satisfying 
\[\e(i,j)=\e(j,i)=\e(|i|,|j|),\qquad \e(i,i)=-1, \qquad i,j\in [\pm n].\]
Let $\{x_j\}_{j\in [\pm n]}$ be a family of self-adjoint matrices in $M_{2^{2n}}(\com)$ satisfying
\begin{equation}\label{hz:index}
x_ix_j-\e(i,j)x_jx_i=2\delta_{i,j}I, \qquad i,j\in [\pm n].
\end{equation}
The unital $\ast$-algebra $\A(n,\e)$ generated by $\{x_j\}_{j\in [\pm n]}$ is the corresponding spin algebra.

A concrete matrix realization can be obtained as follows. Let
\[Q=\begin{pmatrix}
0&1\\
1&0
\end{pmatrix},
\qquad
U=\begin{pmatrix}
1&0\\
0&-1
\end{pmatrix},\]
and, for $j\in [n]$, set
\begin{align*}
Q_j &= I^{\otimes(j-1)} \otimes Q \otimes I^{\otimes(2n-j)}, \qquad U_j = I^{\otimes(j-1)} \otimes U \otimes I^{\otimes(2n-j)}, \\Q_{-j}& =I^{\otimes(j+n-1)} \otimes Q \otimes I^{\otimes(n-j)}, \qquad U_{-j} = I^{\otimes(j+n-1)} \otimes U \otimes I^{\otimes(n-j)}.
\end{align*}

Then one may take
\begin{align*}
x_j&=U_1^{(1-\e(j,1))/2}U_2^{(1-\e(j,2))/2}\cdots U_{j-1}^{(1-\e(j,j-1))/2}Q_j\\
x_{-j}&=U_1^{(1-\e(-j,1))/2}\cdots U_n^{(1-\e(-j,n))/2}U_{-1}^{(1-\e(-j,-1))/2} \cdots U_{-j+1}^{(1-\e(-j,-j+1))/2}Q_{-j}.
\end{align*}

For a subset $S=\{s_1<s_2<\cdots<s_k\}\subseteq [\pm n]$, write
\[x_S=x_{s_1}x_{s_2}\cdots x_{s_k},\qquad x_\emptyset={\bf 1}. \]
Then $\{x_S:S\subseteq [\pm n]\}$ is a linear basis of $\A(n,\e)$. Define a tracial state $\varphi_{n,\e}$ on $\A(n,\e)$ by
\[\varphi_{n,\e}(x_S)=\delta_{S,\emptyset}, \qquad S\subseteq [\pm n].\]
It induces the inner product
\[\lara{x,y}=\varphi_{n,\e}(x^*y),\qquad x,y\in\A(n,\e).\]
Let $H=L_2(\A(n,\e),\varphi_{n,\e})$. Then $\{x_S:S\subseteq [\pm n]\}$ is an orthonormal basis of $H$. 

For $j\in [\pm n]$, define the left annihilation and creation operators on $H$ by
\[\beta_j(x_S)=
\begin{cases}
x_jx_S, & j\in S,\\
0, & j\notin S,
\end{cases}
\qquad
\beta_j^*(x_S)
=\begin{cases}
0, & j\in S,\\
x_jx_S, & j\notin S.
\end{cases}\]
The operators $\beta_j$ and $\beta_j^*$ satisfy
\begin{align*}
  &\beta_i\beta_j-\e(i,j)\beta_j\beta_i = 0, \\
  &\beta_i^*\beta_j^*-\e(i,j)\beta_j^*\beta_i^*=0,\\
  &\beta_i\beta_j^*-\e(i,j)\beta_j^*\beta_i =\delta_{i,j}I
\end{align*}
for all $i,j\in[\pm n]$.

Given positive parameters $\{\mu_j\}_{j=1}^n$, define the generalized Gaussian operators
\[\gamma_j=\mu_j\beta_{-j}+\mu_j^{-1}\beta_j^*,\qquad 1\leq j\leq n.\]
They satisfy
\begin{equation}\label{CR-gamma}
\begin{cases}
\gamma_i\gamma_j-\e(i,j)\gamma_j\gamma_i=0,
& 1\leq i\neq j\leq n,\\
\gamma_i^*\gamma_j-\e(i,j)\gamma_j\gamma_i^*=0,
& 1\leq i\neq j\leq n,\\
\gamma_i^2=(\gamma_i^*)^2=0,
& 1\leq i\leq n,\\
\gamma_i^*\gamma_i+\gamma_i\gamma_i^*
=(\mu_i^2+\mu_i^{-2}){\rm Id},
& 1\leq i\leq n.
\end{cases}
\end{equation}

Let $\Gamma_{n,\varepsilon}\subseteq B(H)$ be the von Neumann algebra generated by $\gamma_1,\ldots,\gamma_{n}$. The vacuum vector ${\bf 1}\in H$ is cyclic and separating for $\Gamma_{n,\varepsilon}$, and the associated vacuum state on $\Gamma_{n,\e}$ is
\[ \tau_{n,\varepsilon}(X)=\lara{{\bf 1},X{\bf 1}},\qquad X\in\Gamma_{n,\varepsilon}.\]

For $1\leq p<\infty$, we write $L_p(\Gamma_{n,\varepsilon})=L_p(\Gamma_{n,\varepsilon},\tau_{n,\varepsilon})$ for the associated Haagerup $L_p$-space. We refer to \cite{Ha1979,Terp1981,Ko1984,JX2003} for the general theory of Haagerup $L_p$-spaces. Let $D_{n,\e}$ denote the density operator associated with $\tau_{n,\varepsilon}$. Since $\Gamma_{n,\varepsilon}$ is finite-dimensional, $D_{n,\e}$ belongs to $\Gamma_{n,\varepsilon}$.

For $k\in [n]$, we identify $\Gamma_{k,\varepsilon}$ with the von Neumann subalgebra of
$\Gamma_{n,\varepsilon}$ generated by $\gamma_1,\ldots,\gamma_k$. The state $\tau_{k,\e}$ is the restriction of $\tau_{n,\varepsilon}$ to this subalgebra. More generally, for a finite subset $S\subseteq[n]$, we denote by $\Gamma_{S,\varepsilon}$ the von Neumann algebra generated by $\{\gamma_j:j\in S\}$ inside $\Gamma_{n,\varepsilon}$, and by $\tau_{S,\e}$ the restriction of $\tau_{n,\e}$ to $\Gamma_{S,\varepsilon}$. 

We record two structural facts that will be used repeatedly below. They follow from the description of the twisted baby Fock model in \cite{NA2006,LR2011}.

\begin{proposition}\label{prop:basis-gamma}
Let $i\in [n]$ and set $\eta_i= \gamma_i^*\gamma_i-\mu_i^{-2}{\rm Id}$. Then $\{{\rm Id},\gamma_i,\gamma_i^*,\eta_i\}$ forms an orthogonal linear basis of $\Gamma_{\{i\},\varepsilon}$ in $L_2(\Gamma_{\{i\},\e},\tau_{n,\e})$. Moreover, each $X\in\Gamma_{n,\varepsilon}$ admits a unique decomposition
\begin{equation}\label{eq-coor-x}
X= a+b\gamma_i+c\gamma_i^*+d\eta_i, \qquad a,b,c,d \in \Gamma_{[n]\setminus\{i\},\e}.
\end{equation}
\end{proposition}

\begin{proposition}\label{prop-indep}
\begin{enumerate}[(i)]
\item For $j\in [n]$, the modular automorphism group associated with $\tau_{n,\varepsilon}$ satisfies
\[\sigma_t^{\tau_{n,\varepsilon}}(\gamma_j)=\mu_j^{4it}\gamma_j, \qquad t\in\real. \]

\item Let $S,T\subseteq [n]$ be finite subsets with $S\subseteq T$, and let
$D_{S,\e}$ and $D_{T,\e}$ be the density operators associated with the corresponding restrictions of $\tau_{n,\varepsilon}$. Then the natural inclusion induces an isometric embedding
\[L_p(\Gamma_{S,\varepsilon}) \hookrightarrow L_p(\Gamma_{T,\varepsilon})\hookrightarrow L_p(\Gamma_{n,\e})\]
given on the canonical symmetric embeddings by
\[D_{S,\e}^{1/2p}XD_{S,\e}^{1/2p} \longmapsto D_{T,\e}^{1/2p}XD_{T,\e}^{1/2p} \longmapsto D_{n,\e}^{1/2p}XD_{n,\e}^{1/2p}. \]

\item If $S,T\subseteq [n]$ are finite and $S\cap T=\emptyset$, then $\Gamma_{S,\varepsilon}$ and $\Gamma_{T,\varepsilon}$ are independent with respect to $\tau_{n,\e}$ in the sense that
\[\tau_{n,\e}(cab)=\tau_{n,\e}(a)\tau_{n,\e}(cb),\qquad a\in\Gamma_{S,\varepsilon},\quad
c,b\in\Gamma_{T,\varepsilon}.\]

\item For a finite subset $S\subseteq [n]$ and $j\in S$, $\gamma_j^*\gamma_j$ commutes with
$\Gamma_{S\setminus\{j\}}$ and with $D_{S,\e}$.
\end{enumerate}
\end{proposition}

The canonical GNS map $\phi_{n,\e}: \Gamma_{n,\varepsilon} \longrightarrow H$ is given by $\phi_{n,\e}(X)=X{\bf 1}$ for all $X\in \Gamma_{n,\e}$. It extends to an onto isometry from $L_2(\Gamma_{n,\varepsilon},\tau_{n,\varepsilon})$ to $H$.

We also recall the $\e$-Ornstein--Uhlenbeck semigroup considered in \cite{LR2011}. Let
\[N^\e=\sum_{i\in [\pm n]}\beta_i^*\beta_i\in B(H).\]
Then $N^\e(x_S)=|S|x_S$ for all $S\subseteq [\pm n]$. The associated semigroup on $\Gamma_{n,\varepsilon}$ is
\[P_t^\e(X)=\phi_{n,\e}^{-1}e^{-tN^\e}\phi_{n,\e}(X), \qquad t\geq0,\]
and its infinitesimal generator is 
\begin{equation}\label{eq221}
\wt{N}^\e=\phi_{n,\e}^{-1}N^\e\phi_{n,\e}:\Gamma_{n,\e}\longrightarrow \Gamma_{n,\e}.    
\end{equation}

\subsection{$q$-Araki--Woods algebras}\label{sec:prelim-qaw} 

We finally recall the $q$-Fock space and the $q$-Araki--Woods construction. The $q$-Araki--Woods algebras were introduced by Hiai \cite{H2003}; see also \cite{LR2011}. Throughout this subsection, we assume that $-1<q<1$.

Let $(\H,\lara{\cdot,\cdot})$ be a separable Hilbert space. Let $\Omega$ denote the vacuum vector and set
\[\H^{\otimes 0}=\com\Omega.\]
For $n\geq1$, define the $q$-symmetrization operator $P_n:\H^{\otimes n}\to\H^{\otimes n}$ by
\[P_n(f_1\otimes\cdots\otimes f_n)=\sum_{\sigma\in S_n}q^{i(\sigma)} f_{\sigma(1)}\otimes\cdots\otimes f_{\sigma(n)},\]
where $i(\sigma)=\#\{(r,s): 1\leq r<s\leq n,\  \sigma(r)>\sigma(s)\}$. For $\xi\in\H^{\otimes m}$ and $\eta\in\H^{\otimes n}$, set
\[\lara{\xi,\eta}_q=\delta_{m,n}\lara{\xi,P_n\eta}.\]
The operator $P_n$ is strictly positive, and hence this defines an inner product on $\H^{\otimes n}$. We denote the corresponding Hilbert space by $\H^{\otimes_q n}$ and define the $q$-Fock space by
\[\F_q(\H)=\com\Omega\oplus\bigoplus_{n=1}^\infty \H^{\otimes_q n}.\]
We denote by $\F_q^{\rm finite}(\H)$ the subspace spanned by finite tensors.

We first recall some standard facts about second quantization. Let $T:\H\to\K$ be a contraction. Its second quantization is defined on finite tensors by
\[\Gamma(T)(f_1\otimes\cdots\otimes f_n)=T(f_1)\otimes\cdots\otimes T(f_n).\]
Then $\Gamma(T):\mathcal{F}_q(\mathcal{H})\rightarrow \mathcal{F}_q(\mathcal{K})$ extends to a contraction.

For a bounded operator $T\in B(\H)$, its differential second quantization is defined on finite tensors by
\[\d\Gamma(T)(f_1\otimes\cdots\otimes f_n)=\sum_{j=1}^n f_1\otimes\cdots\otimes T(f_j) \otimes\cdots\otimes f_n.\]
In contrast to $\Gamma(T)$, the operator $\d\Gamma(T)$ is in general unbounded on the full Fock space, even when $T$ is a contraction. The number operator on $\F_q(\H)$ is given by $N=\d\Gamma(I_\H)$, and satisfies
\[N(f_1\otimes\cdots\otimes f_k)=k\,f_1\otimes\cdots\otimes f_k, \qquad f_1,\ldots,f_k\in\H.\]
The corresponding $q$-Ornstein--Uhlenbeck semigroup is
\[e^{-tN}=\Gamma(e^{-t}I_\H),\qquad t\geq0.\]

We shall use the following standard properties of second quantization; see, for instance, \cite{A2017}.

\begin{proposition}\label{prop200}
\begin{enumerate}[(a)]
\item If $T:\H\to\K$ and $T':\K\to\W$ are contractions, then
\begin{equation}\label{SQ-eq:1}
\Gamma(T^*)=\Gamma(T)^*, \qquad \Gamma(T'T)=\Gamma(T')\Gamma(T).
\end{equation}

\item If $T_n,T:\H\to\K$ are contractions and $T_n\to T$ strongly, then
\begin{equation}\label{SQ:eq-2}
\Gamma(T_n)\longrightarrow\Gamma(T) \qquad \text{strongly}.
\end{equation}

\item If $\K=\H$ and $T\in B(\H)$ is self-adjoint, then
\begin{equation}\label{SQ:eq-3}
\Gamma(e^{itT})=e^{it\d\Gamma(T)}, \qquad t\in\real.
\end{equation}
If, in addition, $T$ is positive, then
\begin{equation}\label{SQ:eq-4}
\Gamma(e^{-zT})=e^{-z\d\Gamma(T)}, \qquad z\in\com, \quad \Re z\geq0.
\end{equation}
\end{enumerate}
\end{proposition}

We next recall the one-particle Hilbert space used in the $q$-Araki--Woods construction. Let $(U_t)_{t\in\real}$ be a strongly continuous one-parameter group of orthogonal transformations on $\real^d$. On the complexification $\com^d$, let $A$ be the positive self-adjoint operator satisfying $U_t=A^{it}$ for all $t\in\real$. Define
\begin{equation}\label{hz:inner}
\lara{x,y}_U=\lara{ 2A(1+A)^{-1}x,y}_{\com^d}.
\end{equation}
The corresponding deformed Hilbert space will be denoted by $\real_A^d$. Whenever there is no ambiguity, we write $\lara{\cdot,\cdot}$ in place of $\lara{\cdot,\cdot}_U$. 

For $l\in\nat$, we use the notation
\[\real_A^{ld}= \real_A^d\otimes\com^l.\]
Equivalently, $\real_A^{ld}$ is the deformed Hilbert space associated with the orthogonal group $U_t\otimes{\rm Id}_{\com^l}$, whose positive generator is $A_l=A\otimes{\rm Id}_{\com^l}$.

For $f\in\real^d$, the left $q$-creation operator $a_+(f)\in B(\F_q(\real_A^d))$ is defined by
\begin{align*}
a_+(f)\Omega&=f,\\
a_+(f)(f_1\otimes\cdots\otimes f_n)&=f\otimes f_1\otimes\cdots\otimes f_n,
\end{align*}
and the left $q$-annihilation operator $a_-(f)=a_+(f)^*\in B(\F_q(\real_A^d))$ is given by
\begin{align*}
a_-(f)\Omega&=0,\\
a_-(f)(f_1\otimes\cdots\otimes f_n)&=\sum_{j=1}^n\lara{f,f_j}q^{j-1}f_1\otimes\cdots\otimes\wh{f_j}\otimes\cdots\otimes f_n,
\end{align*}
where $\wh{f_j}$ means that the $j$-th tensor factor is omitted.

For $f\in\real^d$, define the {\it Segal field operator}
\[s(f)=a_+(f)+a_-(f) \in B(\F_q(\real_A^d)).\]
The associated $q$-Araki--Woods algebra is
\[\Gamma_q(\real_A^d)=\lk{s(f):f\in\real^d}'' \subseteq B(\F_q(\real_A^d)).\]
It is equipped with the vacuum state $\f(X)=\lara{\Omega,X\Omega}$ on $\Gamma_q(\real_A^d)$, which is normal and faithful. The modular automorphism group associated with $\f$ is given by
\[\sigma_t^\f(s(f))=s(U_{-t}f),\qquad t\in\real, \quad f\in\real^d;\]
see \cite{H2003,S1997} for more details.

We next recall second quantization at the von Neumann algebra level in the form that will be used below. For each $l\geq1$, let
\[\phi_{ld}: \Gamma_q(\real_A^{ld}) \longrightarrow \F_q(\real_A^{ld}), \qquad \phi_{ld}(X)=X\Omega,\]
denote the canonical GNS map. Then $\phi_{ld}$ extends to a unitary from $L_2(\Gamma_q(\real_A^{ld}))$ to $\F_q(\real_A^{ld})$.

Let $\T:\real^l\to\real^h$ be a contraction, and denote its complexification from $\com^l$ to $\com^h$ by the same symbol. Then $I\otimes\T:\real_A^{ld} \longrightarrow \real_A^{hd}$ is a contraction and satisfies
\[(I\otimes\T)(U_t\otimes {\rm Id}_{l})=(U_t\otimes {\rm Id}_{h})(I\otimes\T), \qquad t\in\real.\]
Hence, by \cite[Proposition 1.1]{H2003}, there exists a unique normal unital completely positive contraction
\[\Gamma_{A_l,A_h}(I\otimes\T):\Gamma_q(\real_A^{ld})\longrightarrow\Gamma_q(\real_A^{hd})\]
such that
\[\Gamma_{A_l,A_h}(I\otimes\T)(X)\Omega=\Gamma(I\otimes\T)(X\Omega),\qquad X\in\Gamma_q(\real_A^{ld}).\]
Equivalently, on $\phi_{ld}^{-1}\big(\F_q^{\rm finite}(\real_A^{ld})\big)$,
\begin{equation}\label{eq230}
\Gamma_{A_l,A_h}(I\otimes\T)=\phi_{hd}^{-1}\circ \Gamma(I\otimes\T)\circ \phi_{ld}.    
\end{equation}
The map $\Gamma_{A_l,A_h}(I\otimes\T)$ preserves the vacuum states and intertwines the corresponding modular automorphism groups. When $h=l$, we simply write
\[\Gamma_{A_l}(I\otimes\T)=\Gamma_{A_l,A_l}(I\otimes\T).\]
The above \eqref{eq230} and Proposition \ref{prop200} (a) imply that 
\begin{align}\label{eq231}
  \Gamma_{A_h,A_l}(I\otimes\T')\circ   \Gamma_{A_l,A_h}(I\otimes\T) =\phi_{ld}^{-1}\Gamma(I\otimes\T')\Gamma(I\otimes\T)\phi_{ld}=\Gamma_{A_l}(I\otimes \T'\T).    
\end{align}
We also record the canonical inclusions that will be used below. Let $1\leq l\leq m$, and let $j:\real^l\longrightarrow\real^m$ be the canonical embedding. Then
\[h_{ld,md}^{-1}=\Gamma_{A_l,A_m}(I\otimes j):\Gamma_q(\real_A^{ld})\longrightarrow\Gamma_q(\real_A^{md})\]
is a state-preserving embedding. On finite tensors, it is characterized by
\begin{equation}\label{eq322}
   \left ( h_{ld,md}^{-1}\circ \phi_{ld}^{-1}\right )\sk{(f_1\otimes g_1)\otimes\cdots\otimes(f_k\otimes g_k)}=\phi_{md}^{-1}\sk{(f_1\otimes j(g_1))\otimes\cdots\otimes(f_k\otimes j(g_k))}.
\end{equation}

We shall also use the standard extension procedure for such maps on Haagerup $L_p$-spaces. For $1\leq p<\infty$, let $L_p(\Gamma_q(\real_A^d))= L_p(\Gamma_q(\real_A^d),\f)$ denote the associated Haagerup $L_p$-space, and let $D_\f$ be the density operator associated with $\f$. Then $D_\f^{1/2p}\Gamma_q(\real_A^d)D_\f^{1/2p}$ is norm dense in $L_p(\Gamma_q(\real_A^d))$. Let $(\N,\psi)$ be a von Neumann algebra equipped with a normal faithful state, and suppose that $\Psi: \Gamma_q(\real_A^d) \longrightarrow \N$ is a normal state-preserving completely positive contraction satisfying
\[\Psi\circ\sigma_t^\f=\sigma_t^\psi\circ\Psi,\qquad t\in\real.\]

Let $D_\f$ and $D_\psi$ denote the density operators associated with $\f$ and $\psi$, respectively. Then the canonical symmetric $L_p$-extension of $\Psi$ is given on the canonical dense subspace by
\[\Psi^{(p)}\left(D_\f^{1/2p}XD_\f^{1/2p}\right)=D_\psi^{1/2p}\Psi(X)D_\psi^{1/2p},\]
and extends to a contraction on the corresponding Haagerup $L_p$-spaces. We refer to \cite{Ha1979,Terp1981,Ko1984,HJX2010,LR2011} for the general framework.

For a simple tensor $f=f_1\otimes\cdots\otimes f_n$ and an index set $I=\{i_1<\cdots<i_l\}$, write
\begin{align*}
  a_+^I(f) &=a_+(f_{i_1})\cdots a_+(f_{i_l}),\\
  a_-^I(f) &=a_-(f_{i_1})\cdots a_-(f_{i_l})
\end{align*}
with the convention $a_\pm^\emptyset(f)=I$. We recall the Wick formula from \cite[Lemma 3.1]{H2003}.

\begin{lemma}\label{lem:qaw-wick}
For arbitrary $f_1,\ldots,f_n\in\real^d$, one has
\begin{equation}\label{lem-Wick}
\phi_d^{-1}\sk{f_1\otimes\cdots\otimes f_n}
=\sum_{\substack{I=\{i_1<\cdots<i_l\},\\
I^c=\{j_1<\cdots<j_m\},\\
I\cup I^c=\{1,\ldots,n\}}}q^{i(I)}a_+^I(f)a_-^{I^c}(f),
\end{equation}
where $i(I)=\#\{(r,s): 1\leq r\leq l,\  1\leq s\leq m,\  i_r>j_s\}$.

\end{lemma}

In particular, $\phi_d^{-1}\sk{\F_q^{\rm finite}(\real_A^d)}$ is $w^*$-dense in $\Gamma_q(\real_A^d)$, and $D_\f^{1/2p}\phi_d^{-1}\sk{\F_q^{\rm finite}(\real_A^d)}D_\f^{1/2p}$ is norm dense in $L_p(\Gamma_q(\real_A^d))$ for $1\leq p<\infty$.

\section{Transference from biased hypercubes to baby Fock models}\label{sec-transference}

\subsection{The biased hypercube as a commutative baby Fock subalgebra} \label{subsec:weighted-baby-embedding}

Let us explain how to embed the biased hypercube $L_{\infty}(\Omega_n,\nu)$ into a commutative subalgebra of a non-tracial baby Fock model $(\Gamma_{n,\e},\tau_{n,\e})$ by a state-preserving $\ast$-homomorphism. 

The first step is to construct an injective $*$-homomorphism $\pi:L_{\infty}(\Omega_n,\nu)\rightarrow M_2^{\otimes n}$. For each $j\in[n]$, we define $\pi_j:L_\infty(\{-1,1\},\nu_j)\longrightarrow M_2(\com)$ given by
\[\pi_j(f) = \begin{pmatrix}
f(1) & 0\\
0 & f(-1)
\end{pmatrix}. \]
Then $\pi=\pi_1\otimes\cdots\otimes\pi_n: L_\infty(\Omega_n,\nu)\longrightarrow M_2(\com)^{\otimes n}$ is an injective $\ast$-homomorphism. Furthermore, if we consider a state $\psi_j$ on $M_2(\com)$ given by
\[\psi_j\left(\begin{bmatrix}
x_{11} & x_{12}\\
x_{21} & x_{22}
\end{bmatrix}\right)
=\frac{\mu_j^{-2}}{\mu_j^{-2}+\mu_j^2}x_{11}+\frac{\mu_j^2}{\mu_j^{-2}+\mu_j^2}x_{22},\]
then it is straightforward to verify $\displaystyle (\psi_j\circ \pi_j)(f)=\int_{\left\{-1,1\right\}}f d\nu_j$ for all $f\in L_\infty(\{-1,1\},\nu_j)$. Hence, 
\[\pi: L_\infty(\Omega_n,\nu)\rightarrow (M_2(\com)^{\otimes n},\psi) \]
is a state-preserving injective $\ast$-homomorphism with respect to the tensor product state $\psi =\psi_1\otimes\cdots\otimes\psi_n$.

The second step is to apply the matrix realization of non-tracial baby Fock models obtained by Lee and Ricard \cite{LR2011}. Indeed, an identification $(M_2^{\otimes n},\psi)\cong (\Gamma_{n,\e},\tau_{n,\e})$  is obtained by iterating \cite[Proposition 3.3]{LR2011} as follows.

\begin{proposition}\label{prop-matrix-realization}
There exists a state-preserving $\ast$-isomorphism $\Theta_{n,\varepsilon}:(\Gamma_{n,\varepsilon},\tau_{n,\varepsilon}) \longrightarrow (M_2^{\otimes n},\psi)$ such that
\[\Theta_{n,\varepsilon}(\eta_j)=I_2^{\otimes(j-1)}\otimes
\begin{bmatrix}
\mu_j^2 & 0\\
0 & -\mu_j^{-2}
\end{bmatrix} \otimes I_2^{\otimes(n-j)}\]
for all $j\in[n]$, where $\eta_j$ is as in Proposition \ref{prop:basis-gamma}.
\end{proposition}

We can therefore regard $L_{\infty}(\Omega_n,\nu)$ as a commutative $*$-subalgebra of $\Gamma_{n,\varepsilon}$. More precisely, let us define
\[\Phi
=\Theta_{n,\varepsilon}^{-1}\circ\pi:L_\infty(\Omega_n,\nu)\longrightarrow\Gamma_{n,\varepsilon}.\]
Then $\Phi$ is a state-preserving injective $\ast$-homomorphism. More explicitly, for a simple tensor $f=f_1\otimes\cdots\otimes f_n\in L_{\infty}(\Omega_n,\nu)$ with $f_j\in L_\infty(\{-1,1\},\nu_j)$, we have
\[\Phi(f)=\prod_{j=1}^n \left(\frac{f_j(1)}{\mu_j^2+\mu_j^{-2}}\gamma_j^*\gamma_j+\frac{f_j(-1)}{\mu_j^2+\mu_j^{-2}}\gamma_j\gamma_j^*\right).\]
In particular, for the normalized coordinate functions $\zeta_j$, Proposition \ref{prop-matrix-realization} immediately gives
\begin{equation}\label{eq203}
\Phi(\zeta_j)=\eta_j,\qquad j\in[n].    
\end{equation}

By Proposition \ref{prop-indep}(iv), the operators $\eta_1,\ldots,\eta_n$ commute with one another and with $D_{n,\varepsilon}$. Hence, setting $\eta_\emptyset={\rm Id}$ and $\eta_A=\prod_{j\in A}\eta_j$, we obtain $\Phi(\zeta_A)=\eta_A$ for all $A\subseteq[n]$. Consequently,
\[\Phi\bigl(L_\infty(\Omega_n,\nu)\bigr)= W^*(\eta_1,\ldots,\eta_n) \subseteq \Gamma_{n,\varepsilon}.\]
Thus $\Phi$ identifies the biased hypercube with the commutative von Neumann subalgebra in a
state-preserving way. For every $1\leq p<\infty$, $\Phi$ also induces an isometric embedding $\Phi^{(p)}:L_p(\Omega_n,\nu) \longrightarrow L_p(\Gamma_{n,\varepsilon},\tau_{n,\varepsilon})$ given by $\Phi^{(p)}(f)=D_{n,\varepsilon}^{1/2p}\Phi(f)D_{n,\varepsilon}^{1/2p}$.

The preceding realization concerns only the underlying probability spaces. We next construct the doubled baby Fock model that will be used to handle the baby Fock Riesz transforms.

\bigskip

\subsection{The doubled baby Fock model}\label{subsec:doubled-baby-fock}

We now construct the doubled baby Fock model that will be used in the sequel. Recall from Subsection \ref{sec:prelim-baby} that the $n$-mode spin algebra $\A(n,\e)$ is a $\ast$-subalgebra of $M_2^{\otimes 2n}$ generated by $\{x_j\}_{j\in[\pm n]}$. Our aim is to make explicit how the larger space $M_2^{\otimes 2n}$ can be viewed as a doubled $2n$-mode spin model of $\A(n,\e)$.

For any $\theta\in\real$, let $R_\theta=\begin{bmatrix}
1&0\\
0&e^{i\theta}
\end{bmatrix}$, and define $\alpha_\theta:M_2^{\otimes 2n}\longrightarrow M_2^{\otimes 2n}$ by
\[\alpha_\theta(x)=(R_\theta^{\otimes 2n})^* x R_\theta^{\otimes 2n}, \qquad x\in M_2^{\otimes 2n}.\]
Then $(\alpha_\theta)_{\theta\in\real}$ is a one-parameter group of trace-preserving $\ast$-automorphisms of $M_2^{\otimes 2n}$.

Let $y_j=\alpha_{\pi/2}(x_j)$ for any $j\in[\pm n]$. Then
\begin{align}
x_i y_j-\e(i,j)y_jx_i&=0,
\label{eq320}\\
y_i y_j-\e(i,j)y_jy_i&=2\delta_{i,j}I
\label{eq321}
\end{align}
for all $i,j\in[\pm n]$. Moreover, $M_2^{\otimes 2n}$ is generated by $\{x_j,y_j:j\in[\pm n]\}$, and
\[\alpha_\theta(x_j)=x_j\cos\theta+y_j\sin\theta, \qquad j\in[\pm n].\]

Together with the original commutation relations for the operators $x_j$, the relations \eqref{eq320} and \eqref{eq321} allow us to regard the family $\{x_j,y_j:j\in[\pm n]\}$ as a $2n$-mode spin system. Indeed, extend the choice of signs $\e$ to $\e^{(2)}:[\pm2n]\times[\pm2n] \longrightarrow \{-1,1\}$ by requiring
\[\e(i,j)=\e^{(2)}(i,j)=\e^{(2)}(i,n+j)=\e^{(2)}(n+i,j)=\e^{(2)}(n+i,n+j)\]
for all $i,j\in[n]$, and
\[\e^{(2)}(k,l)=\e^{(2)}(|k|,|l|), \qquad k,l\in[\pm2n].\]
In addition, the matrix algebra $M_2^{\otimes 2n}$ can be identified with the spin algebra $\A(2n,\e^{(2)})$ under the correspondence $\left\{\begin{array}{lll}
x_{\pm j}\in M_2^{\otimes 2n}\longleftrightarrow x_{\pm j}\in \A(2n,\e^{(2)})\\
y_j\in M_2^{\otimes 2n}\longleftrightarrow  x_{n+j}\in \A(2n,\e^{(2)})\\
y_{-j}\in M_2^{\otimes 2n}\longleftrightarrow x_{-(n+j)}\in \A(2n,\e^{(2)})
\end{array} \right .$ for all $j\in [n]$.

Let $\varphi^{(2)}_{n,\e}$ denote the normalized tracial state on $M_{2^{2n}}(\com)$. Its restriction to the spin algebra generated by $\{x_j:j\in[\pm n]\}$ coincides with $\varphi_{n,\e}$. For $S,T\subseteq[\pm n]$, write $x_S=x_{s_1}\cdots x_{s_k}$ and $y_T=y_{t_1}\cdots y_{t_l}$, whenever $S=\{s_1<\cdots<s_k\}$ and $T=\{t_1<\cdots<t_l\}$, with $x_\emptyset=y_\emptyset={\bf 1}$. Then $\{x_Sy_T:S,T\subseteq[\pm n]\}$ is an orthonormal basis of
\[H^{(2)}=L_2(M_{2^{2n}}(\com),\varphi^{(2)}_{n,\e}).\]

We use the same notation $\beta_j,\beta_j^*$ for the annihilation and creation operators associated with the first coordinate family on $H^{(2)}$. Thus, for $j\in[\pm n]$,
\[\beta_j(x_Sy_T)=\begin{cases}
x_jx_Sy_T,&j\in S,\\
0,&j\notin S,
\end{cases}
\qquad
\beta_j^*(x_Sy_T)=\begin{cases}
0,&j\in S,\\
x_jx_Sy_T,&j\notin S.
\end{cases}\]
For $j\in[n]$, define the annihilation and creation operators associated
with the second coordinate family by
\[\beta_{n+j}(x_Sy_T)=\begin{cases}
y_jx_Sy_T,&j\in T,\\
0,&j\notin T,
\end{cases}
\qquad
\beta_{n+j}^*(x_Sy_T)=\begin{cases}
0,&j\in T,\\
y_jx_Sy_T,&j\notin T,
\end{cases}\]
and
\[\beta_{-(n+j)}(x_Sy_T)=\begin{cases}
y_{-j}x_Sy_T,&-j\in T,\\
0,&-j\notin T,
\end{cases}
\qquad
\beta_{-(n+j)}^*(x_Sy_T)=\begin{cases}
0,&-j\in T,\\
y_{-j}x_Sy_T,&-j\notin T.
\end{cases}\]

We set $\mu_{n+j}=\mu_j$, and define
\begin{align*}
  &\gamma_j=\mu_j\beta_{-j}+\mu_j^{-1}\beta_j^*\\
  &\gamma_{n+j}=\mu_j\beta_{-n-j}+\mu_j^{-1}\beta_{n+j}^*
\end{align*}
for all $j\in[n]$. We then define the doubled baby Fock algebra by
\[\Gamma_{n,\e}^{(2)}=\{ \gamma_1,\ldots,\gamma_{2n} \}'' \subseteq B(H^{(2)}),\]
and equip it with the vacuum state $\tau_{n,\e}^{(2)}$ on $\Gamma_{n,\e}^{(2)}$ given by $\tau_{n,\e}^{(2)}(X)=\lara{{\bf 1},X{\bf 1}}$. Let $D_{n,\e}^{(2)}$ denote the density operator associated with $\tau_{n,\e}^{(2)}$, and let $\phi_{n,\e}^{(2)}:\Gamma_{n,\e}^{(2)} \longrightarrow H^{(2)}$ be the canonical GNS map given by $\phi_{n,\e}^{(2)}(X)=X{\bf 1}$. Identifying $H$ with the subspace of $H^{(2)}$ spanned by $\{x_S:S\subseteq[\pm n]\}$, we have $\phi_{n,\e}^{(2)}(X)=\phi_{n,\e}(X)$ for all $X\in\Gamma_{n,\e}$.

Under the above correspondence, the doubled space $(\Gamma_{n,\e}^{(2)},\tau_{n,\e}^{(2)})$ is naturally identified with  the non-tracial baby Fock model $(\Gamma_{2n,\e^{(2)}},\tau_{2n,\e^{(2)}})$ associated with $\e^{(2)}$ and the weight parameters $\mu_1,\ldots,\mu_n,\mu_1,\ldots,\mu_n$.

The von Neumann subalgebra of $\Gamma_{n,\e}^{(2)}$ generated by $\gamma_1,\ldots,\gamma_n$ is naturally identified with $\Gamma_{n,\e}$, and $\tau_{n,\e}^{(2)}\Bigr|_{\Gamma_{n,\e}}= \tau_{n,\e}$.  We shall therefore regard $(\Gamma_{n,\e},\tau_{n,\e})\subseteq(\Gamma_{n,\e}^{(2)},\tau_{n,\e}^{(2)})$ as an inclusion of noncommutative probability spaces. 

By Proposition \ref{prop-indep}(ii), applied to the doubled model, for every $1\leq p<\infty$ this inclusion induces an isometric embedding
\[L_p(\Gamma_{n,\e},\tau_{n,\e}) \hookrightarrow L_p(\Gamma_{n,\e}^{(2)},\tau_{n,\e}^{(2)})\]
given on the canonical symmetric embeddings by
\[D_{n,\e}^{1/2p}XD_{n,\e}^{1/2p}\longmapsto (D_{n,\e}^{(2)})^{1/2p}X(D_{n,\e}^{(2)})^{1/2p}.\]

The second coordinate family $\gamma_{n+1},\ldots,\gamma_{2n}$ will be used in the next subsection to assemble the baby Fock Riesz transforms into an operator on the doubled model.

\bigskip

\subsection{Riesz transforms on the biased hypercube and their baby Fock realization}
\label{subsec:riesz-transference}

We now introduce the Riesz transforms on the biased hypercube and relate them explicitly to the corresponding Riesz transforms on the baby Fock model. Recall that $\{\zeta_A:A\subseteq[n]\}$ is an orthonormal basis of $L_2(\Omega_n,\nu)$ and that, by the multiplicativity of $\Phi$ and \eqref{eq203}, the state-preserving injective unital $\ast$-homomorphism $\Phi:L_\infty(\Omega_n,\nu)\hookrightarrow\Gamma_{n,\e}$ satisfies $\Phi(\zeta_A)=\eta_A$ for all $A\subseteq[n]$.

\begin{definition}
For each $j\in[n]$, we define the annihilation operator $D_j\in B(L_2(\Omega_n,\nu))$ by
\begin{equation}\label{eq205}
D_j(\zeta_A)=\begin{cases}
\zeta_{A\setminus\{j\}},&j\in A,\\
0,&j\notin A.
\end{cases}
\end{equation}
\end{definition}

\begin{remark}
The adjoint $D_j^*$ is explicitly given by
\begin{equation}\label{eq206}
D_j^*(\zeta_A)=\begin{cases}
0,&j\in A,\\
\zeta_j\zeta_A,&j\notin A.
\end{cases}
\end{equation}
Furthermore, by \eqref{eq204}, \eqref{eq205}, and \eqref{eq206}, we have $D_j^*D_j=I-E_j$, where $E_j$ is the marginal conditional expectation. Thus, $\displaystyle \sum_{j\in [n]}D_j^*D_j$ coincides with the number operator $N=\displaystyle \sum_{j\in [n]}(I-E_j)$ discussed in Subsection \ref{sec:prelim-weighted}. Recall that $N$ is the infinitesimal generator of the Orenstein-Uhlenbeck semigroup $P_t = e^{-tN}$ studied in \cite{E2023,E2024}.
\end{remark}

For $1\leq p\leq\infty$, we denote by $L_p^0(\Omega_n,\nu)=\left\{ f\in L_p(\Omega_n,\nu):\int_{\Omega_n}f\,d\nu=0\right\}$ the mean-zero subspace. Then $N^{-1/2}$ is naturally defined on $L_{\infty}^0(\Omega_n,\nu)$.

\begin{definition}
For $j\in[n]$, we define the $j$-th Riesz transform on $(\Omega_n,\nu)$ by
\begin{equation}\label{eq207}
R_j=\frac{\mu_j^{-2}+\mu_j^2}{2}D_jN^{-1/2}:L_\infty^0(\Omega_n,\nu)\longrightarrow L_\infty(\Omega_n,\nu).
\end{equation}
\end{definition}

\begin{remark}
For any non-empty $A\subseteq[n]$,
\[R_j(\zeta_A)=\begin{cases}
\displaystyle \frac{\mu_j^{-2}+\mu_j^2}{2\sqrt{|A|}}
\zeta_{A\setminus\{j\}},
&j\in A,\\[8pt]
0,
&j\notin A.
\end{cases}.\]
\end{remark}

We next introduce the Riesz transforms $R_{\pm j}^{\e}$ on the baby Fock model $\Gamma_{n,\e}$. Let us denote by $\Gamma_{n,\e}^0=\left\{X\in\Gamma_{n,\e}: \tau_{n,\e}(X)=0\right\}$. Recall that we defined the number operator $N^{\e}\in B(H)$ and $\wt{N}^{\e}=\phi_{n,\e}^{-1}\circ N^{\e}\circ \phi_{n,\e}\in B(\Gamma_{n,\e})$ in Subsection \ref{sec:prelim-baby}.

\begin{definition}
For $j\in[\pm n]$, we define the $j$-th partial derivative on $\Gamma_{n,\e}$ by
\[\partial_j^\e=\phi_{n,\e}^{-1}\circ\beta_j\circ\phi_{n,\e}:\Gamma_{n,\e}\longrightarrow\Gamma_{n,\e},\]
and the $j$-th Riesz transform by
\begin{equation}\label{eq208}
R_j^\e=\partial_j^\e\circ (\wt{N}^\e)^{-1/2}=\phi_{n,\e}^{-1}\circ\beta_j(N^\e)^{-1/2}\circ\phi_{n,\e}:\Gamma_{n,\e}^0\longrightarrow\Gamma_{n,\e}.
\end{equation}
\end{definition}

\begin{remark}
The definition of the Riesz transforms $R_j^\e$ is natural since it is compatible with the carr\'e du champ associated with the $\e$-Ornstein--Uhlenbeck semigroup as in the classical Euclidean spaces $\real^n$. Let
\[\Gamma_P(X,Y)=\frac12 \left(X^*\wt N^\e(Y)+\wt N^\e(X)^*Y-\wt N^\e(X^*Y)\right),\qquad X,Y\in\Gamma_{n,\e}.\]
Then
\begin{equation}\label{eq:baby-Gamma}
\Gamma_P(X,Y)=\sum_{j\in[\pm n]}\partial_j^\e(X)^*\partial_j^\e(Y), \qquad X,Y\in\Gamma_{n,\e}.
\end{equation}
In particular, for any $X\in\Gamma_{n,\e}^0$,
\begin{equation}\label{eq:baby-Riesz-Gamma}
\Gamma_P\left((\wt N^\e)^{-1/2}X,(\wt N^\e)^{-1/2}X\right)=\sum_{j\in[\pm n]}|R_j^\e(X)|^2.
\end{equation}
Thus the Riesz transforms $R_j^\e$ admit the same coordinate carr\'e-du-champ interpretation as the classical Riesz transforms. The proof of \eqref{eq:baby-Gamma} is given in Appendix \ref{appendix-B}.

\end{remark}

Before comparing the two families of Riesz transforms directly, we first determine explicitly how the number operator $\wt{N}^\e$ and the partial derivatives $\partial_{\pm j}^\e$ act on the commutative $\ast$-subalgebra $\Phi(L_\infty(\Omega_n,\nu))$. These identities will provide the link between the Riesz transforms $R_j$ on the biased hypercube and $R_{\pm j}^\e$ on the baby Fock model.

\begin{proposition}\label{prop300}
For any $A\subseteq[n]$ and $j\in[n]$, we have
\begin{align}
\label{eq209.5} \wt{N}^\e(\eta_A)&=2|A|\eta_A,\\
\label{eq210} \partial_j^\e(\eta_A)&=
\begin{cases}
-\mu_j^{-1}\gamma_j^*\eta_{A\setminus\{j\}},
&j\in A,\\
0,
&j\notin A,
\end{cases}\\
\label{eq211}
\partial_{-j}^\e(\eta_A)&=
\begin{cases}
\mu_j\gamma_j\eta_{A\setminus\{j\}},
&j\in A,\\
0,
&j\notin A.
\end{cases}
\end{align}
Consequently, for any $f\in L_\infty(\Omega_n,\nu)$, we have
\begin{equation}\label{eq212}
\left\{\begin{array}{lll}
\wt{N}^\e(\Phi(f)) = 2 \Phi(Nf)\\
 \partial_j^\e(\Phi(f))=-\mu_j^{-1}\gamma_j^*\Phi(D_jf)\\
\partial_{-j}^\e(\Phi(f))=\mu_j\gamma_j\Phi(D_jf)
\end{array} \right . .
\end{equation}

\end{proposition}

\begin{proof}
It is enough to verify the identities on the basis elements $\eta_A$. First, $\phi_{n,\e}(\eta_A)$ is a homogeneous spin word of degree $2|A|$. Hence $N^\e\phi_{n,\e}(\eta_A)=2|A|\phi_{n,\e}(\eta_A)$, or equivalently,
\[\wt{N}^\e(\eta_A)=(\phi_{n,\e}^{-1}\circ N^\e\circ\phi_{n,\e})(\eta_A)=2|A|\eta_A.\]

We next compute the partial derivatives. Let $j\in A$ and set $B=A\setminus\{j\}$. Since $\eta_B$ involves neither the $j$-th nor the $(-j)$-th spin coordinate, we have
\[\phi_{n,\e}(\eta_A)=x_{-j}x_j\phi_{n,\e}(\eta_B).\]
Using $\e(j,-j)=-1$, we obtain
\[\beta_j\phi_{n,\e}(\eta_A)=\beta_j\bigl(x_{-j}x_j\phi_{n,\e}(\eta_B)\bigr)=-x_{-j}\phi_{n,\e}(\eta_B).\]
On the other hand, since $\gamma_j^*\phi_{n,\e}(\eta_B)=\mu_jx_{-j}\phi_{n,\e}(\eta_B)$, we obtain
\[\partial_j^\e(\eta_A)=-\mu_j^{-1}\gamma_j^*\eta_B,\]
which proves \eqref{eq210} when $j\in A$.

Similarly, $\beta_{-j}\phi_{n,\e}(\eta_A)=x_j\phi_{n,\e}(\eta_B)$, while $\gamma_j\phi_{n,\e}(\eta_B)=\mu_j^{-1}x_j\phi_{n,\e}(\eta_B)$. Thus, we obtain
\[\partial_{-j}^\e(\eta_A)=\mu_j\gamma_j\eta_B,\]
which proves \eqref{eq211} when $j\in A$. If $j\notin A$, both $\beta_j$ and $\beta_{-j}$ vanish on $\phi_{n,\e}(\eta_A)$. Thus, we establish \eqref{eq210} and \eqref{eq211}.

Finally, since $\Phi(\zeta_A)=\eta_A$, while $N(\zeta_A)=|A|\zeta_A$ and $D_j$ is given by \eqref{eq205}, the three identities in \eqref{eq212} follow by linearity.
\end{proof}

A direct consequence of Proposition \ref{prop300} is the following proposition, which gives the precise relation between the biased-hypercube Riesz transform $R_j$ and baby Fock Riesz transforms $R_{\pm j}^\e$.

\begin{proposition}\label{prop:Riesz-transference}
Let $f\in L_\infty^0(\Omega_n,\nu)$ and $j\in[n]$. Then
\begin{align}
\label{eq215} &R_j^\e(\Phi(f))=-\frac{\sqrt{2}\mu_j^{-1}}{\mu_j^2+\mu_j^{-2}}\gamma_j^*\Phi(R_jf),\\
\label{eq216} &R_{-j}^\e(\Phi(f))=\frac{\sqrt{2}\mu_j}{\mu_j^2+\mu_j^{-2}}
\gamma_j\Phi(R_jf).
\end{align}

\end{proposition}

\begin{proof}
By \eqref{eq212}, we obtain
\begin{equation}\label{eq214}
(\wt{N}^\e)^{-1/2}(\Phi(f))=\frac{1}{\sqrt{2}}\Phi(N^{-1/2}f)
\end{equation}
for all $f\in L_\infty^0(\Omega_n,\nu)$. Then, by \eqref{eq214} and \eqref{eq212},
\[R_j^\e(\Phi(f))=\partial_j^\e (\wt{N}^\e)^{-1/2} (\Phi(f))=-\frac{\mu_j^{-1}}{\sqrt{2}}
\gamma_j^*\Phi(D_jN^{-1/2}f)=-\frac{\sqrt{2}\mu_j^{-1}}{\mu_j^2+\mu_j^{-2}}\gamma_j^*\Phi(R_jf).\]
Similarly, we obtain
\[R_{-j}^\e(\Phi(f))=\partial_{-j}^\e (\wt{N}^\e)^{-1/2} (\Phi(f))=\frac{\mu_j}{\sqrt{2}}
\gamma_j\Phi(D_jN^{-1/2}f)=\frac{\sqrt{2}\mu_j}{\mu_j^2+\mu_j^{-2}}
\gamma_j\Phi(R_jf).\]

\end{proof}

\bigskip

\section{Dimension-free Riesz transform estimates on baby Fock models}
\label{sec:baby Fock}

In this section, we define a dilated Riesz transform $\mathcal{R}_n$ and establish the dimension-free $L_p$-estimate for $\mathcal R_n$. The proof follows the spirit of the rotation method of Pisier and Lust-Piquard. The main ingredients are a kernel representation of $\mathcal R_n$ with the decomposition $\mathcal{R}_n=\wt{\Pi}^{(n)}\circ T_n$, and $L^p$-estimates for each of $\wt{\Pi}^{(n)}$ and $T_n$ using the $L^p$-estimates for the vector-valued Hilbert transform of $\mathbb T$ and Stein's complex interpolation theorem.

For $1\leq p<\infty$, let $L_p^0(\Gamma_{n,\e})= \overline{ \left\{D_{n,\e}^{1/2p}XD_{n,\e}^{1/2p}:
X\in\Gamma_{n,\e}^0 \right\} }^{\| \cdot \|_{L_p(\Gamma_{n, \e})} }$. Then the main result of this section for its $L_p$-extension $\mathcal R_n^{(p)}:L_p^0(\Gamma_{n,\e})\longrightarrow L_p(\Gamma_{n,\e}^{(2)})$ is the following.

\begin{thm}\label{Riesz-Baby-Fock}
Let $1<p<\infty$. There exists a constant $C_p>0$, depending only on
$p$, such that
\[C_p^{-1}\|X\|_{L_p(\Gamma_{n,\e})}\leq \|\mathcal R_n^{(p)}(X)\|_{L_p(\Gamma_{n,\e}^{(2)})} \leq C_p \|X\|_{L_p(\Gamma_{n,\e})}\]
for every $X\in L_p^0(\Gamma_{n,\e})$. In particular, the constant $C_p$ is independent of $n$, the choice of signs $\e$, and the parameters $\mu_1,\ldots,\mu_n$.
\end{thm}

\subsection{A dilated Riesz transform and its Kernel representation}

We begin with a kernel representation on the spin Hilbert space. Recall from Subsection \ref{subsec:doubled-baby-fock} that $H$ is naturally identified with the closed subspace of $H^{(2)}$ spanned by $\{x_S:S\subseteq[\pm n]\}$, and that $\{x_Sy_T:S,T\subseteq[\pm n]\}$ is an orthonormal basis of $H^{(2)}$.

For $j\in[\pm n]$, we define
\[y_jH=\{y_j\xi:\xi\in H\}\subseteq H^{(2)}.\]
Then the subspaces $\{y_jH:j\in[\pm n]\}$ are mutually orthogonal. Let $\Pi_j:H^{(2)}\longrightarrow y_jH$ denote the corresponding orthogonal projection, and set $\displaystyle \Pi^{(n)}=\sum_{j\in[\pm n]}\Pi_j$. We also denote by $E_H:H^{(2)}\to H$ the orthogonal projection onto $H$. Recall that the rotation $\alpha_\theta:H^{(2)}\longrightarrow H^{(2)}$ was introduced in Subsection \ref{subsec:doubled-baby-fock} and satisfies
\[ \alpha_\theta(x_j) = x_j\cos\theta+y_j\sin\theta, \qquad j\in[\pm n].\]

Let us define a kernel function 
\[\kappa(\theta)=\frac{{\rm sgn}(\theta)}{\sqrt{-\log\cos^2\theta}}, \qquad -\frac{\pi}{2}<\theta<\frac{\pi}{2}.\]
The following lemma is essentially the kernel representation of Lust-Piquard \cite[Lemma 3.1]{LP1998}, written in the notation of our doubled spin system. 
\begin{lemma}\label{lem:L2-kernel}
If $\xi\in H$ satisfies $\varphi_{n,\e}(\xi)=0$, then
\[E_H\alpha_\theta(\xi)=(\cos\theta)^{N^\e}\xi,\qquad -\frac{\pi}{2}<\theta<\frac{\pi}{2}.\]
Moreover, for any $j\in[\pm n]$ and $\xi\in H$ satisfying $\varphi_{n,\e}(\xi)=0$,
\begin{equation}\label{kernel-rep-eq1}
\beta_j(N^\e)^{-1/2}\xi=\frac{1}{\sqrt{2\pi}}\,
y_j\Pi_j\,{\rm p.v.}\int_{-\pi/2}^{\pi/2}\alpha_\theta(\xi)\kappa(\theta)\,d\theta.
\end{equation}
\end{lemma}

The preceding formula immediately gives a kernel representation of the baby Fock Riesz transform $R_{\pm j}^{\e}$ defined in the previous section.

\begin{lemma}\label{kernel-rep}
Let $X\in\Gamma_{n,\e}^0$ and $j\in[n]$. Then
\begin{align}
\label{eq-kernel-rep-1}
&\mu_j\gamma_{n+j}R_j^\e(X)=\frac{1}{\sqrt{2\pi}}\,
(\phi_{n,\e}^{(2)})^{-1}\Pi_j\,{\rm p.v.}\int_{-\pi/2}^{\pi/2}\alpha_\theta(\phi_{n,\e}(X))\kappa(\theta)\,d\theta,\\
\label{eq-kernel-rep-2}&\mu_j^{-1}\gamma_{n+j}^*R_{-j}^\e(X)=\frac{1}{\sqrt{2\pi}}\,
(\phi_{n,\e}^{(2)})^{-1}\Pi_{-j}\,{\rm p.v.}\int_{-\pi/2}^{\pi/2}\alpha_\theta(\phi_{n,\e}(X))\kappa(\theta)\,d\theta.
\end{align}
\end{lemma}

\begin{proof}
By Lemma \ref{lem:L2-kernel}, for any $X\in \Gamma_{n,\e}^0$,
\begin{equation}\label{eq401}
\phi_{n,\e}(R_j^\e(X))=\frac{1}{\sqrt{2\pi}}\,y_j\Pi_j\,{\rm p.v.}\int_{-\pi/2}^{\pi/2}\alpha_\theta(\phi_{n,\e}(X))\kappa(\theta)\,d\theta.
\end{equation}
Since $R_j^\e(X)\in\Gamma_{n,\e}$ does not involve $\gamma_{n+1},\ldots,\gamma_{2n}$, we have
\begin{align*}
  &\beta_{-(n+j)}\phi_{n,\e}(R_j^\e(X))=0,\\
  &\beta_{n+j}^*\phi_{n,\e}(R_j^\e(X))=y_j\phi_{n,\e}(R_j^\e(X))
\end{align*}
and this implies $\displaystyle \mu_j\gamma_{n+j}\phi_{n,\e}(R_j^\e(X))=y_j\phi_{n,\e}(R_j^\e(X))$. Moreover, by \eqref{eq401}, we obtain
\begin{align*}
  \phi_{n,\e}^{(2)}(\mu_j\gamma_{n+j}R_j^\e(X))&=  \mu_j\gamma_{n+j}\phi_{n,\e}(R_j^\e(X))=y_j\phi_{n,\e}(R_j^\e(X))\\
  &=\frac{1}{\sqrt{2\pi}}\,\Pi_j\,{\rm p.v.}\int_{-\pi/2}^{\pi/2}\alpha_\theta(\phi_{n,\e}(X))\kappa(\theta)\,d\theta,
\end{align*}
which is equivalent to \eqref{eq-kernel-rep-1}. 

The proof of \eqref{eq-kernel-rep-2} is analogous. Indeed, since $R_{-j}^\e(X)\in\Gamma_{n,\e}$ does not involve $\gamma_{n+1},\ldots,\gamma_{2n}$, we have
\begin{align*}
&\beta_{n+j}\phi_{n,\e}(R_{-j}^\e(X))=0,\\
&\beta_{-(n+j)}^*\phi_{n,\e}(R_{-j}^\e(X))= y_{-j}\phi_{n,\e}(R_{-j}^\e(X)).
\end{align*}
Hence $\mu_j^{-1}\gamma_{n+j}^*\phi_{n,\e}(R_{-j}^\e(X))=y_{-j}\phi_{n,\e}(R_{-j}^\e(X))$ holds. Applying Lemma \ref{lem:L2-kernel} with $-j$ in place of $j$ and arguing as above, we obtain
\[\phi_{n,\e}^{(2)}\bigl(\mu_j^{-1}\gamma_{n+j}^*R_{-j}^\e(X)\bigr)=\frac{1}{\sqrt{2\pi}}\,
\Pi_{-j}\,{\rm p.v.}\int_{-\pi/2}^{\pi/2}\alpha_\theta(\phi_{n,\e}(X))\kappa(\theta)\,d\theta,\]
which is equivalent to \eqref{eq-kernel-rep-2}.
\end{proof}

We use the doubled baby Fock model introduced in Subsection \ref{subsec:doubled-baby-fock} to assemble the baby Fock Riesz transforms $R_{\pm j}^\e$ into a single dilated Riesz transform $\mathcal{R}_n$.

\begin{definition}
We define $\mathcal R_n:\Gamma_{n,\e}^0 \longrightarrow \Gamma_{n,\e}^{(2)}$ by 
\begin{equation}\label{eq217}
\mathcal R_n(X)=\sum_{j=1}^n\left( \mu_j\gamma_{n+j}R_j^\e(X)+\mu_j^{-1}\gamma_{n+j}^*R_{-j}^\e(X)\right),\qquad X\in\Gamma_{n,\e}^0.
\end{equation}
For $1\leq p<\infty$, we define $\mathcal R_n^{(p)}:L_p^0(\Gamma_{n,\e})\longrightarrow L_p(\Gamma_{n,\e}^{(2)})$ by
\begin{equation}\label{eq219}
\mathcal R_n^{(p)}\left(D_{n,\e}^{1/2p}XD_{n,\e}^{1/2p}\right)=(D_{n,\e}^{(2)})^{1/2p}\mathcal R_n(X)
(D_{n,\e}^{(2)})^{1/2p},\qquad X\in\Gamma_{n,\e}^0.
\end{equation}

\end{definition}

\begin{remark}\label{rmk401}
\begin{enumerate}
    \item Applying Proposition \ref{prop:Riesz-transference}, for any $f\in L_\infty^0(\Omega_n,\nu)$, we obtain
\begin{equation}\label{eq218}
\mathcal R_n(\Phi(f))=\sum_{j=1}^n\frac{\sqrt{2}}{\mu_j^2+\mu_j^{-2}}\left(-\gamma_{n+j}\gamma_j^*+\gamma_{n+j}^*\gamma_j\right)\Phi(R_jf).
\end{equation}
Thus the biased-hypercube Riesz transforms are encoded in the action of $\mathcal R_n$ on the commutative subalgebra $\Phi(L_\infty(\Omega_n,\nu))$.
\item If we define  $\nabla^\e\xi =\sum_{j\in[\pm n]}y_j\beta_j(\xi)$ on $H$, then
$\nabla^\e\xi=\left.\frac{d}{d\theta}\right|_{\theta=0}\alpha_\theta(\xi)$ holds. Moreover,
\[\phi_{n,\e}^{(2)}(\mathcal R_n(X))=\nabla^\e(N^\e)^{-1/2}\phi_{n,\e}(X)\]
holds for all $X\in\Gamma_{n,\e}^0$. Thus, on the GNS space, $\mathcal R_n$ is precisely the Riesz transform obtained by composing the infinitesimal rotation $\nabla^\e$ with $(N^\e)^{-1/2}$.
\end{enumerate}

\end{remark}

\subsection{$L_p$-estimates for the Riesz transform}

For a linear operator $S:H^{(2)}\to H^{(2)}$, we write $\wt S=(\phi_{n,\e}^{(2)})^{-1}S\phi_{n,\e}^{(2)}\in B(\Gamma_{n,\e}^{(2)})$. In this way, we define
\begin{align*}
&\wt{\alpha}_\theta
=(\phi_{n,\e}^{(2)})^{-1}\alpha_\theta\phi_{n,\e}^{(2)}
\in B(\Gamma_{n,\e}^{(2)}),\\
&\wt{\Pi}^{(n)}
=(\phi_{n,\e}^{(2)})^{-1}\Pi^{(n)}\phi_{n,\e}^{(2)}
\in B(\Gamma_{n,\e}^{(2)}).
\end{align*}
Let us define
\begin{equation}\label{eq404}
T_n(X)=\frac{1}{\sqrt{2\pi}}\, {\rm p.v.}\int_{-\pi/2}^{\pi/2} \wt{\alpha}_\theta(X)\kappa(\theta)\,d\theta, \qquad X\in\Gamma_{n,\e}^{(2)}.    
\end{equation}
Then by Lemma \ref{kernel-rep},
\begin{equation}\label{eq:R-factorization}
\mathcal R_n(X)=\frac{1}{\sqrt{2\pi}}\,\wt{\Pi}^{(n)}{\rm p.v.}\int_{-\pi/2}^{\pi/2}
\wt{\alpha}_\theta(X)\kappa(\theta)\,d\theta=\left(\wt{\Pi}^{(n)}\circ T_n\right)(X),
\qquad X\in\Gamma_{n,\e}^0.
\end{equation}

We next establish dimension-free $L_p$-estimates for $\wt{\Pi}^{(n)}$ and $T_n$. In particular, for the estimate of $T_n$, a key point is that the lifted map $\wt{\alpha}_\theta$ is a state-preserving $*$-automorphism of $\Gamma_{n,\e}^{(2)}$. We establish this fact first. For this purpose, we first examine how the spin rotation acts on the annihilation operators. This will allow us to identify the lifted map $\wt{\alpha}_\theta$ with the restriction of the unitary conjugation induced by $\alpha_\theta$ on $H^{(2)}$.

\begin{proposition}\label{prop:rotation-baby-fock}
For any $\theta\in\mathbb R$, the lifted map $\wt{\alpha}_\theta$ is a state-preserving $*$-automorphism of $\Gamma_{n,\e}^{(2)}$. Moreover, for any $j\in[n]$,
\begin{align}
\wt{\alpha}_\theta(\gamma_j)&=\gamma_j\cos\theta+\gamma_{n+j}\sin\theta,\label{eq:rotation-gamma-1}\\
\wt{\alpha}_\theta(\gamma_{n+j})&=-\gamma_j\sin\theta+\gamma_{n+j}\cos\theta.
\label{eq:rotation-gamma-2}
\end{align}
\end{proposition}

\begin{proof}
Recall that $\alpha_\theta$ is unitary on $H^{(2)}$ and fixes the vacuum vector ${\bf 1}$. We first show that the annihilation operators are covariant under this rotation.

Fix $j\in[ \pm n]$, and let $K_j$ be the closed subspace of $H^{(2)}$ spanned by all basis vectors $x_Sy_T$ such that $j\notin S\cup T$. Then
\[H^{(2)}=K_j\oplus x_jK_j\oplus y_jK_j\oplus x_jy_jK_j\]
orthogonally. For $\xi\in K_j$, the annihilation operators satisfy
\begin{align*}
    \left\{\begin{array}{lll}
    \beta_j(\xi)=0\\
\beta_j(x_j\xi)=\xi\\
\beta_j(y_j\xi)=0\\
\beta_j(x_jy_j\xi)=y_j\xi
    \end{array} \right . 
\qquad \text{and}\qquad
    \left\{\begin{array}{lll}
    \beta_{n+j}(\xi)=0\\
\beta_{n+j}(x_j\xi)=0\\
\beta_{n+j}(y_j\xi)=\xi\\
\beta_{n+j}(x_jy_j\xi)=-x_j\xi
\end{array} \right . .
\end{align*}
Moreover, $\alpha_\theta(K_j)=K_j$, and
\begin{align*}
\alpha_\theta(x_j\xi)
&=
(x_j\cos\theta+y_j\sin\theta)\alpha_\theta(\xi),\\
\alpha_\theta(y_j\xi)
&=
(-x_j\sin\theta+y_j\cos\theta)\alpha_\theta(\xi),\\
\alpha_\theta(x_jy_j\xi)
&=
x_jy_j\alpha_\theta(\xi).
\end{align*}
The last identity follows from $x_j^2=y_j^2={\rm Id}$ and $x_jy_j=-y_jx_j$.

Checking on the four orthogonal summands above, we obtain
\begin{align}
\alpha_\theta\beta_j\alpha_{-\theta}
&=\beta_j\cos\theta+\beta_{n+j}\sin\theta,
\label{eq:rotation-beta-1}\\
\alpha_\theta\beta_{n+j}\alpha_{-\theta}
&=-\beta_j\sin\theta+\beta_{n+j}\cos\theta.
\label{eq:rotation-beta-2}
\end{align}
Applying the same argument to the negative indices gives
\begin{align}
\alpha_\theta\beta_{-j}\alpha_{-\theta}
&=\beta_{-j}\cos\theta+\beta_{-(n+j)}\sin\theta,
\label{eq:rotation-beta-3}\\
\alpha_\theta\beta_{-(n+j)}\alpha_{-\theta}
&=-\beta_{-j}\sin\theta+\beta_{-(n+j)}\cos\theta.
\label{eq:rotation-beta-4}
\end{align}
Since $\alpha_\theta$ is unitary, taking adjoints in \eqref{eq:rotation-beta-1}--\eqref{eq:rotation-beta-4} gives the corresponding identities for the creation operators.

Recall that $\mu_{n+j}=\mu_j$, $\gamma_j=\mu_j\beta_{-j}+\mu_j^{-1}\beta_j^*$, and $\gamma_{n+j}=\mu_j\beta_{-(n+j)}+\mu_j^{-1}\beta_{n+j}^*$. It follows from \eqref{eq:rotation-beta-1}--\eqref{eq:rotation-beta-4} that
\begin{align*}
\alpha_\theta\gamma_j\alpha_{-\theta}
&=\gamma_j\cos\theta+\gamma_{n+j}\sin\theta,\\
\alpha_\theta\gamma_{n+j}\alpha_{-\theta}
&=-\gamma_j\sin\theta+\gamma_{n+j}\cos\theta.
\end{align*}
Hence the unitary conjugation $X\mapsto\alpha_\theta X\alpha_{-\theta}$ leaves $\Gamma_{n,\e}^{(2)}$ invariant.

We now compare this conjugation with the lifted map $\wt{\alpha}_\theta$. Since $\alpha_\theta({\bf 1})={\bf 1}$, for any $X\in\Gamma_{n,\e}^{(2)}$ we have
\begin{align*}
\phi_{n,\e}^{(2)}\bigl(\alpha_\theta X\alpha_{-\theta}\bigr)
=\alpha_\theta X\alpha_{-\theta}({\bf 1})=\alpha_\theta(X{\bf 1})=\alpha_\theta\bigl(\phi_{n,\e}^{(2)}(X)\bigr).
\end{align*}
By the definition of $\wt{\alpha}_\theta$, this implies $\wt{\alpha}_\theta(X)=\alpha_\theta X\alpha_{-\theta}$ for all $X\in\Gamma_{n,\e}^{(2)}$.
Therefore $\wt{\alpha}_\theta$ is a $*$-automorphism of $\Gamma_{n,\e}^{(2)}$, with inverse $\wt{\alpha}_{-\theta}$. The identities \eqref{eq:rotation-gamma-1} and \eqref{eq:rotation-gamma-2} follow from the computation above.

Finally, since $\tau_{n,\e}^{(2)}(X)=\langle X{\bf 1},{\bf 1}\rangle$, the unitarity of $\alpha_\theta$ and the identity $\alpha_\theta({\bf 1})={\bf 1}$ explain
\begin{align*}
\tau_{n,\e}^{(2)}(\wt{\alpha}_\theta(X))= \left\langle\alpha_\theta X\alpha_{-\theta}({\bf 1}),{\bf 1}\right\rangle = \left\langle \alpha_\theta(X{\bf 1}),{\bf 1}\right\rangle= \left\langle X{\bf 1},{\bf 1}\right\rangle=\tau_{n,\e}^{(2)}(X).
\end{align*}
Thus $\wt{\alpha}_\theta$ is state preserving.
\end{proof}

Since $\wt{\alpha}_\theta$ is a state-preserving $*$-automorphism, it commutes with the modular automorphism group associated with $\tau_{n,\e}^{(2)}$. Consequently, for every $1\leq p<\infty$, its canonical extension $\wt{\alpha}_\theta^{(p)}:L_p(\Gamma_{n,\e}^{(2)})\to L_p(\Gamma_{n,\e}^{(2)})$ is an isometry. Moreover, $(\wt{\alpha}_\theta^{(p)})_{\theta\in\mathbb R}$ forms a continuous one-parameter group of isometries on $L_p(\Gamma_{n,\e}^{(2)})$.

We first estimate the singular integral appearing in $T_n$.

\begin{lemma}\label{lem:Tn-bound}
Let $1<p<\infty$. There exists a constant $C_p>0$, depending only on $p$, such that
\[\|T_n^{(p)}(X)\|_{L_p(\Gamma_{n,\e}^{(2)})}\leq C_p\|X\|_{L_p(\Gamma_{n,\e}^{(2)})}\]
for all $X\in L_p(\Gamma_{n,\e}^{(2)})$. In particular, $C_p$ is independent of $n$, $\e$, and $\mu_1,\ldots,\mu_n$.
\end{lemma}

\begin{proof}

Extend $\kappa$ to a function on $\mathbb T=(-\pi,\pi]$ by setting it equal to zero outside $(-\pi/2,\pi/2)$, and consider
\[k(\theta)=\kappa(\theta)-\frac{1}{2\tan(\theta/2)}\in L_1(\mathbb T).\]
Indeed, since $\displaystyle \kappa(\theta)=\frac{1}{\theta} + O(\theta)$ and $\displaystyle \frac{1}{2\tan(\theta/2)}=\frac{1}{\theta}+O(\theta)$ as $\theta\rightarrow 0$, we have
\[k(\theta)=\kappa(\theta)-\frac{1}{2\tan(\theta/2)}=\left ( \kappa(\theta)-\frac{1}{\theta} \right )+\left (\frac{1}{\theta}-\frac{1}{2\tan(\theta/2)} \right )=O(\theta).\]

Since $\kappa(\theta)=1/(2\tan(\theta/2))+k(\theta)$, we decompose $T_n$ as $T_n=H_n+K_n$, where
\begin{align*}
H_n(X)&=\frac{1}{\sqrt{2\pi}}\,{\rm p.v.}\int_{-\pi}^{\pi}\frac{\wt{\alpha}_\theta(X)}{2\tan(\theta/2)}\,d\theta,\\
K_n(X)&=\frac{1}{\sqrt{2\pi}}\,\int_{-\pi}^{\pi}\wt{\alpha}_\theta(X)k(\theta)\,d\theta.
\end{align*}
Here we have extended $\kappa$ by zero outside $(-\pi/2,\pi/2)$, so that the above decomposition agrees with the original definition of $T_n$.

We first estimate $H_n$. Since $(\wt{\alpha}_\theta^{(p)})_{\theta\in\mathbb R}$ is a group of isometries on $L_p(\Gamma_{n,\e}^{(2)})$, the transference principle of Coifman and Weiss \cite{CW1977} yields
\[\|H_n^{(p)}(X)\|_{L_p(\Gamma_{n,\e}^{(2)})}\leq C \|\mathsf H\|_{L_p(\mathbb T;L_p(\Gamma_{n,\e}^{(2)}))\to L_p(\mathbb T;L_p(\Gamma_{n,\e}^{(2)}))}\, \|X\|_{L_p(\Gamma_{n,\e}^{(2)})},\]
where $\mathsf H: L_p(\mathbb T;L_p(\Gamma_{n,\e}^{(2)}))\to L_p(\mathbb T;L_p(\Gamma_{n,\e}^{(2)}))$ is the vector-valued Hilbert transform on $\mathbb T$. Since noncommutative $L_p$-spaces are UMD spaces with UMD constants depending only on $p$, the vector-valued Hilbert transform satisfies
\[\|\mathsf H\|_{L_p(\mathbb T;L_p(\Gamma_{n,\e}^{(2)})) \to L_p(\mathbb T;L_p(\Gamma_{n,\e}^{(2)}))} \leq C_p,\]
where $C_p$ is independent of $n$, $\e$, and the parameters $\mu_1,\ldots,\mu_n$.

For the remaining part, since $k\in L_1(\mathbb T)$ and $\wt{\alpha}_\theta^{(p)}$ is an isometry, Minkowski's inequality gives
\begin{align*}
\|K_n^{(p)}(X)\|_{L_p(\Gamma_{n,\e}^{(2)})}
&\leq \frac{1}{\sqrt{2\pi}}\int_{-\pi}^{\pi}|k(\theta)| \|\wt{\alpha}_\theta^{(p)}(X)\|_{L_p(\Gamma_{n,\e}^{(2)})}\,d\theta\\
&=\frac{\|k\|_{L_1(\mathbb T)}}{\sqrt{2\pi}}\|X\|_{L_p(\Gamma_{n,\e}^{(2)})}.
\end{align*}
Combining the two estimates, we obtain
\[\|T_n^{(p)}(X)\|_{L_p(\Gamma_{n,\e}^{(2)})} \leq \left(C_p+\frac{\|k\|_{L_1(\mathbb T)}}{\sqrt{2\pi}}\right)\|X\|_{L_p(\Gamma_{n,\e}^{(2)})}.\]
Since the kernel $k$ is independent of $n$, $\e$, and $\mu_1,\ldots,\mu_n$, by setting $C_p'=C_p+\frac{\|k\|_{L_1(\mathbb T)}}{\sqrt{2\pi}}$, we obtain
\[\|T_n^{(p)}(X)\|_{L_p(\Gamma_{n,\e}^{(2)})}\leq C_p' \|X\|_{L_p(\Gamma_{n,\e}^{(2)})},\]
where $C_p'$ depends only on $p$.
\end{proof}

We next estimate the projection $\wt{\Pi}^{(n)}$. Let us define the second-coordinate number operator on $H^{(2)}$ by
\[N_y^\e=\sum_{j=1}^n\left(\beta_{n+j}^*\beta_{n+j}+\beta_{-(n+j)}^*\beta_{-(n+j)}\right),\]
and denote by $\wt{N}_y^\e=(\phi_{n,\e}^{(2)})^{-1}N_y^\e\phi_{n,\e}^{(2)}$. Here, $N_y^\e(x_Sy_T)=|T|x_Sy_T$ for all $S,T\subseteq[\pm n]$. On the other hand, by the definition of $\Pi^{(n)}$, we have
\begin{equation}\label{eq402}
    \Pi^{(n)}(x_Sy_T)=\begin{cases}
x_Sy_T,& |T|=1,\\
0,& |T|\neq1.
\end{cases}
\end{equation}
Thus, $\Pi^{(n)}$ is precisely the spectral projection of $N_y^\e$ corresponding to the eigenvalue $1$, i.e. $\Pi^{(n)}={\bf 1}_{\left\{1\right\}}(N_y^\e)$.

\begin{lemma}\label{lem:Pi-bound}
Let $1<p<\infty$. The operator $\wt{\Pi}^{(n)}$ extends to a bounded operator on $L_p(\Gamma_{n,\e}^{(2)})$, and there exists a constant $C_p>0$, depending only on $p$, such that
\[\|(\wt{\Pi}^{(n)})^{(p)}\|_{L_p(\Gamma_{n,\e}^{(2)})\to L_p(\Gamma_{n,\e}^{(2)})}\leq C_p.\]
\end{lemma}

\begin{proof}
For $t\geq0$, let $S_t=e^{-t\wt{N}_y^\e}$. This is the Ornstein--Uhlenbeck semigroup acting non-trivially only on the system generated by $\{\gamma_{n+1},\ldots,\gamma_{2n}\}$. In particular, $(S_t)_{t\geq0}$ is a semigroup of normal, completely positive, state-preserving contractions on $\Gamma_{n,\e}^{(2)}$. Moreover, $(S_t)_{t\geq0}$ is symmetric on $L_2(\Gamma_{n,\e}^{(2)})$.

By Stein's complex interpolation theorem \cite{Ste1970,JLMX2006}, for any $1<p<\infty$ there exist $\theta_p>0$ and $B_p>0$, depending only on $p$, such that $(S_t^{(p)})_{t\geq0}$ admits an analytic extension $S_z^{(p)}$ to the sector
\[\Sigma_{\theta_p}=\{z\in\mathbb C:|\arg z|<\theta_p\},\]
with $\displaystyle \sup_{z\in\Sigma_{\theta_p}}\|S_z^{(p)}\|_{L_p(\Gamma_{n,\e}^{(2)})\to L_p(\Gamma_{n,\e}^{(2)})} \leq B_p$. Here, $S_z=e^{-z\wt{N}_y^\e}$ for $z\in\mathbb C$.

Choose $a_p>0$ sufficiently large so that $\arctan\frac{\pi}{a_p}<\theta_p$. Then $a_p+is\in\Sigma_{\theta_p}$ for all $-\pi\leq s\leq\pi$. If $\wt{N}_y^\e u=ku$ for some $k\in\mathbb N_0$, then $S_{a_p+is}u=e^{-(a_p+is)k}u$. Hence,
\begin{align*}
\left(\frac{e^{a_p}}{2\pi}\int_{-\pi}^{\pi} e^{is}S_{a_p+is}\,ds\right)u
&=\left(\frac{e^{a_p}}{2\pi}\int_{-\pi}^{\pi}e^{is}e^{-(a_p+is)k}\,ds \right)u \\
&=\left( \frac{e^{a_p(1-k)}}{2\pi} \int_{-\pi}^{\pi}e^{-is(k-1)}\,ds\right)u =
\begin{cases}
u,& k=1,\\
0,& k\neq1.
\end{cases}
\end{align*}
Thus, by \eqref{eq402}, we obtain $\displaystyle \wt{\Pi}^{(n)}=\frac{e^{a_p}}{2\pi}\int_{-\pi}^{\pi}e^{is}S_{a_p+is}\,ds$ on $\Gamma_{n,\e}^{(2)}$, and
\[(\wt{\Pi}^{(n)})^{(p)}=\frac{e^{a_p}}{2\pi}\int_{-\pi}^{\pi}e^{is}S_{a_p+is}^{(p)}\,ds\]
with the following estimate
\[\|(\wt{\Pi}^{(n)})^{(p)}\|_{L_p(\Gamma_{n,\e}^{(2)})\to L_p(\Gamma_{n,\e}^{(2)})} \leq e^{a_p} \sup_{|s|\leq\pi} \|S_{a_p+is}^{(p)}\|_{L_p(\Gamma_{n,\e}^{(2)})\to L_p(\Gamma_{n,\e}^{(2)})} \leq e^{a_p}B_p.\]
Since both $a_p$ and $B_p$ depend only on $p$, setting $C_p=e^{a_p}B_p$ completes the proof.
\end{proof}

Combining \eqref{eq:R-factorization} with Lemma \ref{lem:Tn-bound} and Lemma \ref{lem:Pi-bound}, and using the isometric inclusion of the first copy $\Gamma_{n,\e}$ into
$\Gamma_{n,\e}^{(2)}$, there exists a constant $C_p$ depending only on $p$ such that 
\begin{equation}\label{eq403}
\norm{\mathcal R_n^{(p)}}_{L_p^0(\Gamma_{n,\e})\rightarrow L_p(\Gamma_{n,\e}^{(2)})}\leq C_p.    
\end{equation}
To complete the proof of Theorem \ref{Riesz-Baby-Fock}, it remains to show that $\mathcal R_n$ is an isomorphism onto its range. We first record the $L_2$-identity underlying the reverse estimate.

\begin{lemma}\label{lem:R-L2-isometry}
For all $X,Y\in\Gamma_{n,\e}^0$,
\[\left\langle \mathcal R_n(X), \mathcal R_n(Y)\right\rangle_{L_2(\Gamma_{n,\e}^{(2)})}=\langle X,Y\rangle_{L_2(\Gamma_{n,\e})}.\]
In particular, $\mathcal R_n$ is isometric on $\Gamma_{n,\e}^0$ with respect to the corresponding $L_2$-norms.
\end{lemma}

\begin{proof}
Let $\xi=\phi_{n,\e}(X)$ and $\eta=\phi_{n,\e}(Y)$. By Remark \ref{rmk401} (2),
\[\phi_{n,\e}^{(2)}(\mathcal R_n(X))=\sum_{j\in[\pm n]} y_j\beta_j(N^\e)^{-1/2}\xi,\]
and similarly for $Y$. Since the subspaces
$\{y_jH:j\in[\pm n]\}$ are mutually orthogonal,
\begin{align*}
&\left\langle\mathcal R_n(X),\mathcal R_n(Y)\right\rangle_{L_2(\Gamma_{n,\e}^{(2)})}
=\sum_{j\in[\pm n]}\left\langle\beta_j(N^\e)^{-1/2}\xi,\beta_j(N^\e)^{-1/2}\eta\right\rangle_{H}\\
&=\left\langle\xi,(N^\e)^{-1/2}\left(\sum_{j\in[\pm n]}\beta_j^*\beta_j\right)(N^\e)^{-1/2}\eta\right\rangle_H=
\langle\xi,\eta\rangle_H=\langle X,Y\rangle_{L_2(\Gamma_{n,\e})}.
\end{align*}
Since $X$ and $Y$ have mean zero, $(N^\e)^{-1/2}$ is well defined on the corresponding vectors. This proves the claim.
\end{proof}

We now combine the preceding results to complete the proof of Theorem \ref{Riesz-Baby-Fock}.

\begin{proof}[Proof of Theorem \ref{Riesz-Baby-Fock}]

Let $J:\Gamma_{n,\e}\hookrightarrow\Gamma_{n,\e}^{(2)}$ denote the canonical inclusion, and let $J^{(p)}$ be its $L_p$-extension. The upper estimate has already been established in \eqref{eq403}. It remains to prove the reverse inequality.

We first work at the algebraic $L_2$-level, where the adjoints below are taken with respect to the corresponding $L_2$-inner products on $\Gamma_{n,\e}$ and $\Gamma_{n,\e}^{(2)}$. Let
$E:\Gamma_{n,\e}^{(2)}\to\Gamma_{n,\e}$ be the state-preserving conditional expectation onto $\Gamma_{n,\e}$. Then, for $X\in\Gamma_{n,\e}$ and ${\bf Y}\in\Gamma_{n,\e}^{(2)}$,
\[\langle J(X),{\bf Y}\rangle_{L_2(\Gamma_{n,\e}^{(2)})}=\langle J(X){\bf 1},{\bf Y}{\bf 1}\rangle_{H^{(2)}}=\langle X{\bf 1},E({\bf Y}){\bf 1}\rangle_H=\langle X,E({\bf Y})\rangle_{L_2(\Gamma_{n,\e})}.\]
Thus, $J^*=E$ at the algebraic $L_2$-level.

Moreover, $\left(\wt{\Pi}^{(n)}\right)^*=\wt{\Pi}^{(n)}$ with respect to the $L_2$-inner product, while $\kappa(-\theta)=-\kappa(\theta)$ and $\wt{\alpha}_\theta^*=\wt{\alpha}_{-\theta}$ imply $T_n^*=-T_n$ by \eqref{eq404}. Thus, by \eqref{eq:R-factorization},
\begin{equation}\label{eq:R-adjoint-factorization}
\mathcal R_n^*=-ET_n\wt{\Pi}^{(n)}.
\end{equation}

By Proposition \ref{prop-indep}, applied to the doubled system, the first copy $\Gamma_{n,\e}$ is invariant under the modular automorphism group $\sigma_t^{\tau_{n,\e}^{(2)}}$, and $\displaystyle \sigma_t^{\tau_{n,\e}^{(2)}}\Bigr|_{\Gamma_{n,\e}}=\sigma_t^{\tau_{n,\e}}$. Hence we have
$\sigma_t^{\tau_{n,\e}^{(2)}}\circ J=J\circ\sigma_t^{\tau_{n,\e}}$ 
and $E\circ\sigma_t^{\tau_{n,\e}^{(2)}}=\sigma_t^{\tau_{n,\e}}\circ E$. Moreover, $\wt{\alpha}_\theta$ commutes with $\sigma_t^{\tau_{n,\e}^{(2)}}$, and hence so does $T_n$. In addition, by Proposition \ref{prop-indep}, $\wt{N}_y^\e$ commutes with $\sigma_t^{\tau_{n,\e}^{(2)}}$. Since $\wt{\Pi}^{(n)}={\bf 1}_{\{1\}}(\wt{N}_y^\e)$, it follows that $\wt{\Pi}^{(n)}$ also commutes with $\sigma_t^{\tau_{n,\e}^{(2)}}$.

Thus all the maps involved in \eqref{eq:R-factorization} and \eqref{eq:R-adjoint-factorization} are compatible with the corresponding modular automorphism groups, and their canonical $L_p$-extensions are compatible with composition. In particular, $\mathcal R_n^*$ admits a bounded $L_p$-extension $(\mathcal R_n^*)^{(p)}$ satisfying
\[(\mathcal R_n^*)^{(p)}=-E^{(p)}T_n^{(p)}(\wt{\Pi}^{(n)})^{(p)}.\]
Since $E$ is a state-preserving conditional expectation, its $L_p$-extension $E^{(p)}$ is contractive. Hence, Lemma \ref{lem:Tn-bound} and Lemma \ref{lem:Pi-bound} yield
\[\|(\mathcal R_n^*)^{(p)}\|_{L_p(\Gamma_{n,\e}^{(2)})\to L_p(\Gamma_{n,\e})}\leq C_p,\]
where $C_p$ depends only on $p$.

On the other hand, Lemma \ref{lem:R-L2-isometry} shows that, for all $X,Y\in\Gamma_{n,\e}^0$,
\begin{align*}
\langle \mathcal R_n^*\mathcal R_n(X),Y\rangle_{L_2(\Gamma_{n,\e})}&= \langle \mathcal R_n(X),\mathcal R_n(Y)\rangle_{L_2(\Gamma_{n,\e}^{(2)})}\\
&=\langle X,Y\rangle_{L_2(\Gamma_{n,\e})}.
\end{align*}
Therefore, $\mathcal R_n^*\mathcal R_n=I$ on $\Gamma_{n,\e}^0$. By the compatibility of the $L_p$-extensions with composition, we obtain
\[(\mathcal R_n^*)^{(p)}\mathcal R_n^{(p)}( D_{n, \e}^{1/2p} X D_{n, \e}^{1/2p})= D_{n, \e}^{1/2p} X D_{n, \e}^{1/2p},\qquad X\in\Gamma_{n,\e}^0.\]

Since $D_{n, \e}^{1/2p} \Gamma_{n,\e}^0 D_{n, \e}^{1/2p}$ is dense in $L_p^0(\Gamma_{n,\e})$ and both operators are bounded, together with \eqref{eq403} this yields
\[C_p^{-1}\|X\|_{L_p(\Gamma_{n,\e})} \leq \|\mathcal R_n^{(p)}(X)\|_{L_p(\Gamma_{n,\e}^{(2)})} \leq C_p\|X\|_{L_p(\Gamma_{n,\e})}, \qquad X\in L_p^0(\Gamma_{n,\e}),\]
for some $C_p>0$ depending only on $p$. Thus $\mathcal R_n^{(p)}$ is an isomorphism from $L_p^0(\Gamma_{n,\e})$ onto its range. This completes the proof of Theorem \ref{Riesz-Baby-Fock}.
\end{proof}

\bigskip

\section{Dimension-free Riesz transform estimates on biased hypercubes}
\label{sec-main-theorem}

In this section, we return to the biased hypercube and combine the transference results from Section \ref{sec-transference} with the dimension-free estimate for $\mathcal R_n^{(p)}$ obtained in Theorem \ref{Riesz-Baby-Fock}. We first give an affirmative answer to Meyer's problem for the biased Ornstein--Uhlenbeck semigroup. We then study the unweighted Riesz square function under the additional condition \eqref{hz:optimal}.

\subsection{An affirmative answer to Meyer's problem}

We first introduce the doubled commutative embedding that will be used below. Let $\eta_{n+j}=\gamma_{n+j}^*\gamma_{n+j}-\mu_j^{-2}{\rm Id}$ for all $j\in[n]$. Applying Proposition \ref{prop-matrix-realization} to the doubled baby
Fock model, we obtain a state-preserving injective $*$-homomorphism 
\[\Phi_2:L_\infty(\Omega_n\times\Omega_n,\nu\times\nu) \longrightarrow \Gamma_{n,\e}^{(2)} \]
satisfying $\Phi_2(\zeta_j\otimes1)=\eta_j$ and $\Phi_2(1\otimes\zeta_j)=\eta_{n+j}$ for all $j\in[n]$. By Proposition \ref{prop-indep}(iv), applied to the doubled model, the operators $\eta_1,\ldots,\eta_n,\eta_{n+1},\ldots,\eta_{2n}$ commute with $D_{n,\e}^{(2)}$. Hence the image of $\Phi_2$ is contained in the centralizer of $\tau_{n,\e}^{(2)}$, and $\Phi_2$ induces an isometric embedding $\Phi_2^{(p)}:L_p(\Omega_n\times\Omega_n,\nu\times\nu) \longrightarrow L_p(\Gamma_{n,\e}^{(2)})$ for any $1\leq p<\infty$.

Recall that $z_j(\omega)=\omega_j$ and $z_j= \frac{2}{\mu_j^2+\mu_j^{-2}}\zeta_j+\frac{\mu_j^{-2}-\mu_j^2}{\mu_j^2+\mu_j^{-2}}$. Thus,
\begin{align*}
\Phi_2(z_j\otimes1)&=\frac{1}{\mu_j^2+\mu_j^{-2}}(\gamma_j^*\gamma_j-\gamma_j\gamma_j^*),\\
\Phi_2(1\otimes z_j)&=\frac{1}{\mu_j^2+\mu_j^{-2}}(\gamma_{n+j}^*\gamma_{n+j}
-\gamma_{n+j}\gamma_{n+j}^*).
\end{align*}
Consequently,
\begin{align}
\label{eq:Phi2-square}
\Phi_2(1\otimes1-z_j\otimes z_j)=\frac{2}{(\mu_j^2+\mu_j^{-2})^2}\Big( \gamma_j\gamma_j^*\gamma_{n+j}^*\gamma_{n+j}+\gamma_j^*\gamma_j\gamma_{n+j}\gamma_{n+j}^*\Big).
\end{align}

We next record the sign-invariance property needed to apply the noncommutative Khintchine inequality.

\begin{lemma}\label{lem:sign-invariance}
For every $\delta=(\delta_1,\ldots,\delta_n)\in\{-1,1\}^n$, there exists a state-preserving $*$-automorphism $\rho_\delta:\Gamma_{n,\e}^{(2)}\to\Gamma_{n,\e}^{(2)}$ such that
\begin{align*}
\left\{\begin{array}{lll}
\rho_\delta(X)=X,& X\in\Gamma_{n,\e}\\
\rho_\delta(\gamma_{n+j})=\delta_j\gamma_{n+j},&j\in [n]\\
\rho_\delta(\gamma_{n+j}^*)=\delta_j\gamma_{n+j}^*,& j\in[n].
\end{array} \right .    
\end{align*}
In particular, its canonical $L_p$-extension $\rho_\delta^{(p)}$ is an isometry on $L_p(\Gamma_{n,\e}^{(2)})$ for any $1\leq p<\infty$.
\end{lemma}

\begin{proof}
Define a unitary $U_\delta$ on $H^{(2)}$ by $U_\delta({\bf 1})={\bf 1}$ and
\[U_\delta(x_Sy_T)=\left(\prod_{t\in T}\delta_{|t|}\right)x_Sy_T, \qquad S,T\subseteq[\pm n].\]
Then we have
\begin{align*}
&U_\delta\beta_jU_\delta^*=\beta_j,\qquad j\in [\pm n],\\    
&U_\delta\beta_j^*U_\delta^*=\beta_j^*,\qquad j\in[\pm n],\\
&U_\delta\beta_{n+j}U_\delta^*=\delta_j\beta_{n+j}, \qquad j\in [n],\\ 
&U_\delta\beta_{n+j}^*U_\delta^*=\delta_j\beta_{n+j}^*, \qquad j\in [n],\\ 
&U_\delta\beta_{-(n+j)}U_\delta^*=\delta_j\beta_{-(n+j)}, \qquad j\in [n],\\ 
&U_\delta\beta_{-(n+j)}^*U_\delta^*=\delta_j\beta_{-(n+j)}^*, \qquad j\in [n] .
\end{align*}

It follows that the unitary conjugation $\rho_\delta(X)=U_\delta XU_\delta^*$ leaves $\Gamma_{n,\e}^{(2)}$ invariant, fixes $\Gamma_{n,\e}$ pointwise, and satisfies $\rho_\delta(\gamma_{n+j})=\delta_j\gamma_{n+j}$ for all $j\in[n]$. In addition, since $U_\delta({\bf 1})={\bf 1}$, we have
\[\tau_{n,\e}^{(2)}\left (\rho_\delta({\bf X})\right )=\tau_{n,\e}^{(2)}\left ( U_{\delta}{\bf X}U_{\delta}^* \right )=\tau_{n,\e}^{(2)}\left ( {\bf X}\right ).\]
Hence its canonical $L_p$-extension $\rho_{\delta}^{(p)}$ is an isometry.
\end{proof}

Let $\ce_{\omega'}$ denote the conditional expectation obtained by integrating with respect to the second variable on $(\Omega_n\times\Omega_n,\nu\times\nu)$. By Proposition \ref{hz:propA1}, for any mean-zero $f$,
\begin{align}
\label{eq:Gamma-Riesz}
\Gamma(N^{-1/2}f,N^{-1/2}f)&=\ce_{\omega'}\left[\sum_{j=1}^n(1\otimes1-z_j\otimes z_j)(|R_jf|^2\otimes1)\right] \notag \\
= & \sum_{j=1}^n \left(\int_{\Omega_n}(1-z_j(z)z_j(z'))\,d\nu(z')\right)\cdot |R_jf(z)|^2 \notag\\
= & \sum_{j=1}^n\frac{2(1+\zeta_j^2)}{(\mu_j^{-2}+\mu_j^2)^2}|R_jf|^2.
\end{align}
Thus, unlike the uniform case, the carr\'e du champ square function $\Gamma(N^{-1/2}f,N^{-1/2}f)^{1/2}$ is a weighted form of the classical unweighted Riesz square function $\left(\sum_{j=1}^n|R_jf|^2\right)^{1/2}$, where the weights depend on
both the coordinate $j$ and the underlying measure $\nu_j$.

Our approach to obtaining an affirmative answer to Meyer's problem consists of two steps. The first step is to overcome the non-tracial nature of the baby Fock model and apply the noncommutative Khintchine inequality to establish
\[\|f\|_{L_p(\Omega_n,\nu)}\approx_p\left\|\left(\sum_{j=1}^n(1\otimes1-z_j\otimes z_j)(|R_jf|^2\otimes1)\right)^{1/2}\right\|_{L_p(\Omega_n\times\Omega_n,\nu\times\nu)}.\]
The second step is to establish
\[\left\|\left(\sum_{j=1}^n (1\otimes1-z_j\otimes z_j)(|R_jf|^2\otimes1)\right)^{1/2}\right\|_{L_p(\Omega_n\times\Omega_n,\nu\times\nu)}\approx_p\left\|\Gamma(N^{-1/2}f,N^{-1/2}f)^{1/2}\right\|_{L_p(\Omega_n,\nu)}\]
to complete the proof. We begin with the first step, which is given by the following dimension-free estimate.

\begin{thm}\label{thm-Riesz-Walsh}
Let $2\leq p<\infty$. There exists a constant $C_p>0$, depending only on $p$, such that, for any mean-zero $f\in L_p(\Omega_n,\nu)$,
\begin{align*}
C_p^{-1}\|f\|_{L_p(\Omega_n,\nu)}&\leq \left\|\left(\sum_{j=1}^n(1\otimes1-z_j\otimes z_j)(|R_jf|^2\otimes1)\right)^{1/2}\right\|_{L_p(\Omega_n\times\Omega_n,\nu\times\nu)} \leq C_p\|f\|_{L_p(\Omega_n,\nu)}.
\end{align*}
In particular, $C_p$ is independent of $n$ and of the product measure $\nu$.
\end{thm}

\begin{proof}
Let $A_j=\frac{\sqrt{2}}{\mu_j^2+\mu_j^{-2}}\left(-\gamma_{n+j}\gamma_j^*+\gamma_{n+j}^*\gamma_j\right)\Phi(R_jf)$ for all $j\in[n]$, and let $D=D_{n,\e}^{(2)}$ for simplicity. By \eqref{eq218} and \eqref{eq219}, \begin{equation}\label{eq500} 
\mathcal R_n^{(p)}(\Phi^{(p)}(f))=D^{1/2p}\left(\sum_{j=1}^nA_j\right)D^{1/2p}.
\end{equation}

We first note that each $A_j$ belongs to the centralizer of $\tau_{n,\e}^{(2)}$. Indeed, $-\gamma_{n+j}\gamma_j^*+\gamma_{n+j}^*\gamma_j$ is fixed by $\sigma_t^{\tau_{n,\e}^{(2)}}$ since $\sigma_t^{\tau_{n,\e}^{(2)}}(\gamma_j)=\mu_j^{4it}\gamma_j$ and $\sigma_t^{\tau_{n,\e}^{(2)}}(\gamma_{n+j})=\mu_j^{4it}\gamma_{n+j}$ by applying Proposition \ref{prop-indep} to the doubled model, and $\Phi(R_jf)$ is contained in the centralizer.

By \eqref{eq205} and \eqref{eq207}, $\Phi(R_jf)$ is generated by $\eta_1,\cdots,\eta_{j-1},\eta_{j+1},\cdots,\eta_n$. Thus, $\Phi(R_jf)$ commutes with $\gamma_j,\gamma_j^*,\gamma_{n+j}$ and $\gamma_{n+j}^*$ by applying Proposition \ref{prop-indep}(iv) to the doubled model. Using the relations \eqref{CR-gamma}, we therefore obtain
\begin{align}
\label{eq:Aj-square}
A_j^*A_j=A_jA_j^*=\frac{2}{(\mu_j^2+\mu_j^{-2})^2}\Big(\gamma_j\gamma_j^*\gamma_{n+j}^*\gamma_{n+j}+\gamma_j^*\gamma_j\gamma_{n+j}\gamma_{n+j}^*\Big)|\Phi(R_jf)|^2.
\end{align}

Since $\Phi^{(p)}$ is isometric, Theorem \ref{Riesz-Baby-Fock} and \eqref{eq500} give
\begin{align*}
&\|f\|_{L_p(\Omega_n,\nu)}=\|\Phi^{(p)}(f)\|_{L_p(\Gamma_{n,\e})}\\
&\approx_p \|\mathcal R_n^{(p)}(\Phi^{(p)}(f))\|_{L_p(\Gamma_{n,\e}^{(2)})}=
\left\|D^{1/2p}\left(\sum_{j=1}^nA_j\right)D^{1/2p}\right\|_{L_p(\Gamma_{n,\e}^{(2)})}=
\left\|\sum_{j=1}^nA_jD^{1/p}\right\|_{L_p(\Gamma_{n,\e}^{(2)})},
\end{align*}
where the last equality follows from the fact that $A_jD^{1/p}=D^{1/p}A_j$ for all $j$.

Let $\ce_\delta$ denote the expectation over independent Rademacher signs $\delta_1,\ldots,\delta_n$. By Lemma
\ref{lem:sign-invariance},
\[\left\|\sum_{j=1}^nA_jD^{1/p} \right\|_{L_p(\Gamma_{n,\e}^{(2)})}= \left(\ce_\delta \left\|\sum_{j=1}^n\delta_jA_jD^{1/p}
\right\|_{L_p(\Gamma_{n,\e}^{(2)})}^p \right)^{1/p}.\]

Although the state $\tau_{n,\e}^{(2)}$ is non-tracial, each $A_j$ belongs to its centralizer. Hence $A_jD^{1/p}=D^{1/p}A_j$, and using $A_j^*A_j=A_jA_j^*$, we obtain
\begin{align*}
\sum_{j=1}^n (A_jD^{1/p})^*(A_jD^{1/p})=D^{1/p} \left(\sum_{j=1}^nA_j^*A_j\right)D^{1/p}=\sum_{j=1}^n(A_jD^{1/p})(A_jD^{1/p})^*.
\end{align*}
Therefore, applying the noncommutative Khintchine inequality \cite[Theorem 6.2]{HJX2010} in the Haagerup $L_p$-space, we obtain
\begin{align*}
\norm{f}_{L_p(\Omega_n,\nu)}&\approx_p \left\|\sum_{j=1}^nA_jD^{1/p}\right\|_{L_p(\Gamma_{n,\e}^{(2)})}\\
&\approx_p\norm{\left(D^{1/p}\left(\sum_{j=1}^nA_j^*A_j\right)D^{1/p}\right)^{1/2}}_{L_p(\Gamma_{n,\e}^{(2)})}=
\norm{D^{1/2p}\left(\sum_{j=1}^nA_j^*A_j\right)^{1/2}D^{1/2p}}_{L_p(\Gamma_{n,\e}^{(2)})}.
\end{align*}
By \eqref{eq:Phi2-square} and \eqref{eq:Aj-square}, $A_j^*A_j=\Phi_2\left((1\otimes1-z_j\otimes z_j)(|R_jf|^2\otimes1)\right)$. Since $\Phi_2$ is a $*$-homomorphism and its image is contained in the centralizer of $\tau_{n,\e}^{(2)}$,
\begin{align*}
D^{1/2p}\left(\sum_{j=1}^nA_j^*A_j\right)^{1/2}D^{1/2p}=
\Phi_2^{(p)}\left(\left(\sum_{j=1}^n(1\otimes1-z_j\otimes z_j)(|R_jf|^2\otimes1)\right)^{1/2}\right).
\end{align*}
Finally, since $\Phi_2^{(p)}$ is an isometry, we obtain
\[\|f\|_{L_p(\Omega_n,\nu)}\approx_p\norm{\left(\sum_{j=1}^n(1\otimes1-z_j\otimes z_j)(|R_jf|^2\otimes1)\right)^{1/2}}_{L_p(\Omega_n\times\Omega_n,\nu\times\nu)}.\]
\end{proof}

We now turn to the second step.

\begin{thm}\label{thm-Riesz-G}
Let $2\leq p<\infty$. There exists a constant $C_p>0$, depending only on $p$, such that
\[C_p^{-1}\|f\|_{L_p(\Omega_n,\nu)} \leq \|\Gamma(N^{-1/2}f,N^{-1/2}f)^{1/2}\|_{L_p(\Omega_n,\nu)} \leq C_p\|f\|_{L_p(\Omega_n,\nu)}\]
for all mean-zero $f\in L_p(\Omega_n,\nu)$. The constant $C_p$ is independent of $n$ and $\nu$.
\end{thm}

\begin{proof}
Since $p\geq2$, we have $p/2\geq1$, and the conditional expectation $\ce_{\omega'}$ is contractive on $L_{p/2}$. Hence, by \eqref{eq:Gamma-Riesz},
\begin{align*}
\left\|\Gamma(N^{-1/2}f,N^{-1/2}f)^{1/2}\right\|_{L_p(\Omega_n,\nu)}^2&=
\left\|\ce_{\omega'}\left[\sum_{j=1}^n(1\otimes1-z_j\otimes z_j)(|R_jf|^2\otimes1)\right]\right\|_{L_{p/2}(\Omega_n,\nu)}\\
&\leq \left\|\sum_{j=1}^n(1\otimes1-z_j\otimes z_j)(|R_jf|^2\otimes1)\right\|_{L_{p/2}(\Omega_n\times\Omega_n,\nu\times\nu)}\\
&=\left\|\left (\sum_{j=1}^n(1\otimes1-z_j\otimes z_j)(|R_jf|^2\otimes1) \right )^{1/2}\right\|_{L_{p}(\Omega_n\times\Omega_n,\nu\times\nu)}^2.
\end{align*}
Then, by Theorem \ref{thm-Riesz-Walsh}, we obtain
\begin{equation}\label{eq502}
\left\|\Gamma(N^{-1/2}f,N^{-1/2}f)^{1/2}\right\|_{L_p(\Omega_n,\nu)} \leq C_p\|f\|_{L_p(\Omega_n,\nu)}.    
\end{equation}

For the reverse estimate, we use the Meyer inequality of Bakry and Junge-Mei
\cite{Bak1985,JM2010},
\begin{equation}\label{eq501}
\|N^{1/2}g\|_{L_p(\Omega_n,\nu)} \leq C_p\|\Gamma(g,g)^{1/2}\|_{L_p(\Omega_n,\nu)}.    
\end{equation}
Since $f$ has mean zero, $g=N^{-1/2}f$ is well defined, so we obtain
\begin{equation}\label{eq503}
    \|f\|_{L_p(\Omega_n,\nu)} \leq C_p\left\|\Gamma(N^{-1/2}f,N^{-1/2}f)^{1/2}\right\|_{L_p(\Omega_n,\nu)}
\end{equation}
by \eqref{eq501}. Thus, combining \eqref{eq502} and \eqref{eq503} completes the proof.
\end{proof}

\begin{remark}
The lower estimate in the proof of Theorem \ref{thm-Riesz-G} is obtained by appealing to the result of Bakry and Junge-Mei \cite{Bak1985,JM2010}. It would be interesting to derive this estimate directly from Theorem
\ref{thm-Riesz-Walsh}.
\end{remark}

\subsection{The unweighted Riesz square function} In the previous subsection, Theorem \ref{thm-Riesz-G} established an affirmative answer to Meyer's problem without any further assumption on the parameters $\mu_j$. We now turn to the second question raised in the Introduction: Under what conditions do we have the dimension-free estimate
\[ \left\|\left(\sum_{j=1}^n|R_jf|^2\right)^{1/2}\right\|_{L_p(\Omega_n,\nu)} \approx_p \left\|\Gamma(N^{-1/2}f,N^{-1/2}f)^{1/2}\right\|_{L_p(\Omega_n,\nu)}\]
for the unweighted Riesz square function?

We first show that the second question has an affirmative answer under the boundedness condition \eqref{hz:optimal}.

\begin{proposition}\label{prop500}
Assume that $\displaystyle M=\sup_{j\in \mathbb N}\max\left\{\mu_j,\mu_j^{-1}\right\}<\infty$ holds. Then, for any $2\leq p<\infty$ and mean-zero $f\in L_p(\Omega_n,\nu)$, we have
\begin{align*}
\left(\frac{2}{1+M^4}\right)^{1/2}\left\|\left(\sum_{j=1}^n|R_jf|^2\right)^{1/2}\right\|_{L_p(\Omega_n,\nu)}
&\leq\left\|\left(\sum_{j=1}^n(1\otimes1-z_j\otimes z_j)(|R_jf|^2\otimes1)\right)^{1/2}\right\|_{L_p(\Omega_n\times\Omega_n,\nu\times\nu)}
\\
&\leq\sqrt{2}\left\|\left(\sum_{j=1}^n|R_jf|^2\right)^{1/2}\right\|_{L_p(\Omega_n,\nu)}.
\end{align*}
In particular, there exists a constant $C_{p,M}>0$, depending only
on $p$ and $M$, such that
\begin{equation}\label{hz:df}
C_{p,M}^{-1}\|f\|_{L_p(\Omega_n,\nu)}\leq\left\|\left(\sum_{j=1}^n|R_jf|^2\right)^{1/2}\right\|_{L_p(\Omega_n,\nu)}\leq C_{p,M}\|f\|_{L_p(\Omega_n,\nu)}
\end{equation}
for all mean-zero $f\in L_p(\Omega_n,\nu)$.

\end{proposition}

\begin{proof}
Let us write $a_j(\omega,\omega')=1-z_j(\omega)z_j(\omega')$ for $\omega,\omega'\in\Omega_n$. Then $a_j(\omega,\omega')\in\{0,2\}$. Let us fix $\omega\in\Omega_n$. Then
\[\mathbb P_{\omega'}\bigl[z_j(\omega')\neq z_j(\omega)\bigr]
=\begin{cases}
\displaystyle \frac{\mu_j^2}{\mu_j^{-2}+\mu_j^2}, &z_j(\omega)=1,\\[8pt]
\displaystyle \frac{\mu_j^{-2}}{\mu_j^{-2}+\mu_j^2}, &z_j(\omega)=-1.
\end{cases}\]
Since $\max\{\mu_j,\mu_j^{-1}\}\leq M$, we have
\[\mathbb P_{\omega'}\bigl[z_j(\omega')\neq z_j(\omega)\bigr]\geq\frac{1}{1+M^4},\]
and hence
\[\ce_{\omega'}a_j(\omega,\omega')=2\cdot \mathbb P_{\omega'}\bigl[z_j(\omega')\neq z_j(\omega)\bigr] \geq\frac{2}{1+M^4}.\]

Since $p\geq2$, the function $t\mapsto t^{p/2}$ is convex on $[0,\infty)$. Thus, by Jensen's inequality,
\begin{align*}
\int_{\Omega_n}\left(\sum_{j=1}^na_j(\omega,\omega')|R_jf(\omega)|^2\right)^{p/2}\,d\nu(\omega')&\geq\left(\sum_{j=1}^n\ce_{\omega'}a_j(\omega,\omega')|R_jf(\omega)|^2\right)^{p/2}\\
&\geq \left(\frac{2}{1+M^4}\right)^{p/2}\left(\sum_{j=1}^n|R_jf(\omega)|^2\right)^{p/2}.
\end{align*}
This implies
\begin{align*}
&\left\|\left(\sum_{j=1}^n(1\otimes1-z_j\otimes z_j)(|R_jf|^2\otimes1)\right)^{1/2}\right\|_{L_p(\Omega_n\times\Omega_n,\nu\times\nu)}\geq
\left(\frac{2}{1+M^4}\right)^{1/2} \left\|\left(\sum_{j=1}^n|R_jf|^2\right)^{1/2}\right\|_{L_p(\Omega_n,\nu)}.
\end{align*}

For the reverse estimate, since
$0\leq a_j(\omega,\omega')\leq2$,
\[\sum_{j=1}^n a_j(\omega,\omega')|R_jf(\omega)|^2 \leq 2\sum_{j=1}^n|R_jf(\omega)|^2.\]
Taking the $L_p$-norm on
$\Omega_n\times\Omega_n$ gives
\[\left\|\left(\sum_{j=1}^n(1\otimes1-z_j\otimes z_j)(|R_jf|^2\otimes1)\right)^{1/2}\right\|_{L_p(\Omega_n\times\Omega_n,\nu\times\nu)}\leq\sqrt{2}\left\|\left(\sum_{j=1}^n|R_jf|^2\right)^{1/2}\right\|_{L_p(\Omega_n,\nu)}. \]

Lastly, since
\[\left\|\left(\sum_{j=1}^n(1\otimes1-z_j\otimes z_j) (|R_jf|^2\otimes1)\right)^{1/2}\right\|_{L_p(\Omega_n\times\Omega_n,\nu\times\nu)}\approx_p \norm{f}_{L_p(\Omega_n,\nu)}\]
by Theorem \ref{thm-Riesz-Walsh}, there exists $C_p>0$ such that
\begin{align}
C_p^{-1}\|f\|_{L_p(\Omega_n,\nu)} \leq \left\| \left( \sum_{j=1}^n
(1\otimes1-z_j\otimes z_j)(|R_jf|^2\otimes1)\right)^{1/2}\right\|_{L_p(\Omega_n\times\Omega_n,\nu\times\nu)} \leq
C_p\|f\|_{L_p(\Omega_n,\nu)}.
\end{align}
Combining all the estimates above, we obtain 
\begin{align*}
 \frac{1}{C_p\sqrt{2}}\norm{f}_{L_p(\Omega_n,\nu)} \leq  \left\|\left(\sum_{j=1}^n|R_jf|^2\right)^{1/2}\right\|_{L_p(\Omega_n,\nu)}\leq C_p \left(\frac{1+M^4}{2}\right)^{1/2}\norm{f}_{L_p(\Omega_n,\nu)}.
\end{align*}
\end{proof}

By Theorem \ref{thm-Riesz-G}, we can establish an affirmative answer to the second question raised in the Introduction under the boundedness condition $ \displaystyle M= \sup_{j\in \mathbb N}\max\left\{\mu_j,\mu_j^{-1}\right\}<\infty$. 

We conclude this subsection by discussing the necessity of the boundedness condition on the parameters. Let $\Omega_\infty=\{-1,1\}^{\mathbb N}$ and equip it with the product
probability measure $\nu_\infty=\bigotimes_{j\geq1}\nu_j$, where each $\nu_j$ is the
probability measure introduced in Subsection \ref{sec:prelim-weighted}. A dimension-free estimate on the finite-dimensional biased hypercubes naturally yields the same estimate on $(\Omega_\infty,\nu_\infty)$ for functions depending on only finitely many coordinates, since each such function can be regarded as a function on some finite-dimensional hypercube. We show that such an estimate forces the following boundedness condition $\displaystyle M= \sup_{j\geq1}\max\{\mu_j,\mu_j^{-1}\}<\infty$.

\begin{proposition}\label{prop:necessity-weight}
Let $2\leq p<\infty$. Suppose that there exists a constant $C_p>0$ such
that
\[\left\|\left(\sum_{j\geq1}|R_jf|^2\right)^{1/2}\right\|_{L_p(\Omega_\infty,\nu_\infty)}\leq C_p\|f\|_{L_p(\Omega_\infty,\nu_\infty)}\]
for every mean-zero function $f$ depending on only finitely many
coordinates. Then 
\[M=\sup_{j\geq1}\max\{\mu_j,\mu_j^{-1}\}<\infty.\]
\end{proposition}

\begin{proof}
Fix $i\in\mathbb N$ and let $\displaystyle f_i=\frac{\zeta_i}{\|\zeta_i\|_{L_p(\Omega_\infty,\nu_\infty)}}$. Then $\|f_i\|_{L_p(\Omega_\infty,\nu_\infty)}=1$, and
\[R_jf_i=\delta_{j,i}\frac{\mu_i^{-2}+\mu_i^2}{2\|\zeta_i\|_{L_p(\Omega_\infty,\nu_\infty)}}=\delta_{j,i}\frac{\left ( \mu_i^{-2}+\mu_i^2\right )^{1+1/p}}{2(\mu_i^{-2p+2}+\mu_i^{2p-2})^{1/p}}.\]
Therefore,
\[\left\|\left(\sum_{j\geq1}|R_jf_i|^2\right)^{1/2}\right\|_{L_p(\Omega_\infty,\nu_\infty)}=\frac{(\mu_i^{-2}+\mu_i^2)^{1+1/p}}{2(\mu_i^{-2p+2}+\mu_i^{2p-2})^{1/p}}.\]

Suppose first that $\mu_i\geq1$. Since $p\geq2$, we have
\[\mu_i^{-2p+2}+\mu_i^{2p-2}\leq2\mu_i^{2p-2}.\]
Hence
\[\left\|\left(\sum_{j\geq1}|R_jf_i|^2\right)^{1/2}\right\|_{L_p(\Omega_\infty,\nu_\infty)}
\geq 2^{-1-1/p}\mu_i^{4/p}.\]
By the assumed estimate, we obtain $2^{-1-1/p}\mu_i^{4/p}\leq C_p$. Thus the parameters $\mu_i$ are uniformly bounded whenever $\mu_i\geq1$.

If $0<\mu_i\leq1$, the same argument, with $\mu_i^{-1}$ in place of $\mu_i$, yields $2^{-1-1/p}\mu_i^{-4/p}\leq C_p$. Hence, we obtain $\displaystyle \sup_{i\geq1}\max\{\mu_i,\mu_i^{-1}\}<\infty$.
\end{proof}

Thus, the boundedness condition \eqref{hz:optimal} is necessary for the dimension-free estimate \eqref{hz:df}.

\bigskip

\section{Dimension-free Riesz transform estimates on $q$-Araki--Woods algebras}
\label{sec-q-Araki-Woods}

In this section, we establish dimension-free Riesz transform estimates on $q$-Araki--Woods algebras. We use the notation and the basic facts on $q$-Fock spaces, second quantization, and $q$-Araki--Woods algebras introduced in Subsection \ref{sec:prelim-qaw}. The proof follows the same rotation and transference mechanism as in Section \ref{sec:baby Fock}.
We first realize the Riesz transform through a rotation on the doubled $q$-Fock space, and then obtain the $L_p$-estimate from its kernel representation.

\subsection{The Riesz transform and its kernel representation}

Let $w_1=(1,0)^T$ and $w_2=(0,1)^T$ be the canonical basis of $\com^2$, and let $j:\com\to\com^2$ be the isometric embedding given by $j(\lambda)=\lambda w_1$. Recall that $\real_A^{2d}=\real_A^d\otimes\com^2$. We denote by
\[j_{d,2d}=\Gamma(I\otimes j):\F_q(\real_A^d)\longrightarrow\F_q(\real_A^{2d})\]
the corresponding isometric embedding. In particular, for $f_1,\ldots,f_k\in\real_A^d$,
\[j_{d,2d}(f_1\otimes\cdots\otimes f_k)=(f_1\otimes w_1)\otimes\cdots\otimes(f_k\otimes w_1).\]
For the canonical embedding $h_{d,2d}^{-1}:\Gamma_q(\mathbb R^d_A)\rightarrow \Gamma_q(\mathbb R^{2d}_A)$ in Subsection \ref{sec:prelim-qaw}, we have
\begin{equation}\label{eq600}
\phi_{2d}^{-1}\circ j_{d,2d}\circ \phi_d= h_{d,2d}^{-1}.
\end{equation}

Let $P=\begin{bmatrix}
0&i\\
-i&0
\end{bmatrix}$ and $R_\theta=e^{i\theta P}=\begin{bmatrix}
\cos\theta&-\sin\theta\\
\sin\theta&\cos\theta
\end{bmatrix}$ for $\theta\in\real$. Then $(R_\theta)_{\theta\in\real}$ is the usual rotation group on $\real^2$. Let us define the gradient on $\F_q^{\rm finite}(\real_A^d)$ by
\begin{equation}\label{eq612}
\nabla =i\d\Gamma(I\otimes P)\circ j_{d,2d}:\mathcal{F}_q^{\rm finite}(\mathbb R^d_A)\rightarrow \mathcal{F}_q^{\rm finite}(\mathbb R^{2d}_A),    
\end{equation}
and recall that the number operator on $\F_q(\real_A^d)$ is $N=\d\Gamma(I)$. We define the corresponding Riesz transform by
\begin{equation}
\R=\nabla N^{-1/2}:\F_q^{\rm finite}(\real_A^d)\ominus\com\Omega\longrightarrow \F_q^{\rm finite}(\real_A^{2d}).    
\end{equation}

The following identity is the $q$-Fock analogue of the relation $\sum_{j\in[\pm n]}\beta_j^*\beta_j=N^\e$ used in the proof of Lemma \ref{lem:R-L2-isometry}.

\begin{lemma}\label{lem601}
On $\F_q^{\rm finite}(\real_A^d)$, we have $\nabla^*\nabla=N$. Consequently, $\R$ extends to an isometry from $\F_q(\real_A^d)\ominus\com\Omega$ into $\F_q(\real_A^{2d})$.
\end{lemma}

\begin{proof}
This is the same computation as in \cite[Lemma 1.2]{LP1999}. More precisely, if $u,v\in(\real_A^d)^{\otimes_q k}$, then we have $\langle\nabla u,\nabla v\rangle_q= k\langle u,v\rangle_q$. In addition, since $N=kI$ on $(\real_A^d)^{\otimes_q k}$, the desired identity follows.
\end{proof}

We next introduce the projection that appears in the kernel representation. Let $E=\begin{bmatrix}
0&0\\
0&1
\end{bmatrix}$ and define $N_2=\d\Gamma(I\otimes E)$ on $\F_q^{\rm finite}(\real_A^{2d})$. Thus $N_2$ counts the number of tensor components whose second factor is $w_2$. Let
\[\Pi={\bf 1}_{\{1\}}(N_2)\]
be the spectral projection corresponding to the eigenvalue $1$. Equivalently, if
\[v_j^{(k)}= w_1^{\otimes (j-1)}\otimes w_2\otimes w_1^{\otimes (k-j)}, \qquad 1\leq j\leq k,\]
then $\Pi$ is the orthogonal projection onto the closed subspace
\begin{equation}\label{eq:Pi-K}
\W=\bigoplus_{k=1}^\infty \overline{(\real_A^d)^{\otimes k}\otimes{\rm span}
\{v_1^{(k)},\ldots,v_k^{(k)}\}},
\end{equation}
where the closure is taken with respect to the $q$-inner product.

The following kernel representation is the $q$-Fock version of Lemma \ref{lem:L2-kernel}. The setting is slightly different but its proof is
essentially the same as that of \cite[Lemma 1.3]{LP1999}; we include it here for the reader's convenience.

\begin{lemma}\label{lem-ker-F}
Let the notation be as above.
\begin{enumerate}[(i)]
\item For any $\theta\in (-\pi/2,\pi/2)$, we have
\begin{equation}\label{KR:eq-1}
(\cos\theta)^N=j_{d,2d}^*\circ \Gamma(I\otimes R_\theta)\circ j_{d,2d}
=j_{d,2d}^*\circ e^{i\theta\d\Gamma(I\otimes P)}\circ j_{d,2d}.
\end{equation}

\item On $\F_q^{\rm finite}(\real_A^{2d})$, for any $a>0$, we have
\begin{equation}\label{KR:eq-2}
\Pi = \frac{e^a}{2\pi} \int_{-\pi}^{\pi}e^{is}e^{-(a+is)N_2}\,ds.
\end{equation}

\item On $\F_q^{\rm finite}(\real_A^d)\ominus\com\Omega$, we have
\[\nabla\circ (\cos\theta)^{N-1}\sin\theta=\Pi\circ \Gamma(I\otimes R_\theta)\circ j_{d,2d}\]
and, on $\F_q^{\rm finite}(\real_A^d)$, we have
\[N=j_{d,2d}^*\circ  \d\Gamma(I\otimes P)^2\circ j_{d,2d}.\]

\item On $\F_q^{\rm finite}(\real_A^d)\ominus\com\Omega$,
\begin{equation}\label{KR:eq-3}
\R =\frac{1}{\sqrt{2\pi}}\Pi\circ{\rm p.v.}\int_{-\pi/2}^{\pi/2}\frac{{\rm sgn}(\theta)}{\sqrt{-\log\cos^2\theta}}\Gamma(I\otimes R_\theta)\circ j_{d,2d}\,d\theta.
\end{equation}
\end{enumerate}
\end{lemma}

\begin{proof}
Let $u=f_1\otimes\cdots\otimes f_k$. Then
\begin{align}
\notag\left (\Gamma(I\otimes R_\theta)\circ j_{d,2d}\right )(u)
&= \bigotimes_{r=1}^k \left( f_r\otimes (\cos\theta\,w_1+\sin\theta\,w_2)\right)\\
\label{eq601}&=\sum_{S\subseteq[k]} (\cos\theta)^{k-|S|}(\sin\theta)^{|S|} \bigotimes_{r=1}^k \left(f_r\otimes w_{1+\mathbf 1_S(r)}\right).
\end{align}

First, let us prove (i). For any $S\subseteq[k]$, we have
\[j_{d,2d}^*\left(\bigotimes_{r=1}^k\left(f_r\otimes w_{1+\mathbf 1_S(r)}\right)\right)
=\delta_{S,\emptyset}\cdot \bigotimes_{r=1}^k f_r.\]
Therefore, applying $j_{d,2d}^*$ to \eqref{eq601}, we obtain
\[ \left (j_{d,2d}^*\circ \Gamma(I\otimes R_\theta)\circ j_{d,2d}\right )(u)
= (\cos\theta)^k u=(\cos\theta)^N u.\]
Since $R_\theta=e^{i\theta P}$, Proposition \ref{prop200} (c) gives
\[\Gamma(I\otimes R_\theta)=e^{i\theta\d\Gamma(I\otimes P)}.\]
This completes the proof of (i). 

For (ii), if $N_2\xi=m\xi$ for some $m\in\mathbb N_0$, then
\begin{align*}
\left(\frac{e^a}{2\pi}\int_{-\pi}^{\pi}e^{is}e^{-(a+is)N_2}\,ds \right)\xi = \frac{e^{a(1-m)}}{2\pi}
\int_{-\pi}^{\pi}e^{-is(m-1)}\,ds\,\xi =
\begin{cases}
\xi,&m=1,\\
0,&m\neq1.
\end{cases}
\end{align*}
Thus, we obtain $\Pi\displaystyle = \frac{e^a}{2\pi} \int_{-\pi}^{\pi}e^{is}e^{-(a+is)N_2}\,ds$.

For (iii). For any $S\subseteq[k]$ and $u=f_1\otimes\cdots\otimes f_k$, we have
\[\Pi\left(\bigotimes_{r=1}^k\left(f_r\otimes w_{1+\mathbf 1_S(r)}\right)\right)
=\mathbf 1_{\{|S|=1\}}\bigotimes_{r=1}^k\left(f_r\otimes w_{1+\mathbf 1_S(r)}\right).\]
Therefore, applying $\Pi$ to \eqref{eq601}, we obtain
\begin{align}\label{eq602}
\left (\Pi\circ \Gamma(I\otimes R_\theta)\circ j_{d,2d}\right )(u)=(\cos\theta)^{k-1}\sin\theta \sum_{r=1}^k (f_1\otimes w_1)\otimes\cdots\otimes (f_r\otimes w_2)\otimes\cdots\otimes (f_k\otimes w_1).
\end{align}

On the other hand, since $Pw_1=-iw_2$, the gradient $\nabla=i\d\Gamma(I\otimes P)\circ j_{d,2d}$ is computed as
\begin{align}
\notag \nabla u&=i\sum_{r=1}^k
(f_1\otimes w_1)\otimes\cdots\otimes
\bigl(f_r\otimes Pw_1\bigr)\otimes\cdots\otimes
(f_k\otimes w_1)\\
\label{eq603}&=\sum_{r=1}^k(f_1\otimes w_1)\otimes\cdots\otimes(f_r\otimes w_2)\otimes\cdots\otimes(f_k\otimes w_1).
\end{align}
Combining \eqref{eq602} and \eqref{eq603}, this proves the first part of (iii):
\[\left (\Pi\circ \Gamma(I\otimes R_\theta)\circ j_{d,2d}\right )(u) =(\cos\theta)^{k-1}\sin\theta\,\nabla u=\nabla(\cos\theta)^{N-1}\sin\theta\,u.\]

The second assertion of (iii) immediately follows from Lemma \ref{lem601} and \eqref{eq612}.

For (iv), if $u$ is homogeneous of degree $k\geq1$, then
\begin{align*}
&\frac{1}{\sqrt{2\pi}}\Pi\,{\rm p.v.}\int_{-\pi/2}^{\pi/2}\frac{{\rm sgn}(\theta)}{\sqrt{-\log\cos^2\theta}}
(\Gamma(I\otimes R_\theta)\circ j_{d,2d})(u)\,d\theta\\
&=\frac{1}{\sqrt{2\pi}}\int_{-\pi/2}^{\pi/2}\frac{(\cos\theta)^{k-1}|\sin\theta|}{\sqrt{-\log\cos^2\theta}}\,d\theta\cdot\nabla u =\frac{1}{\sqrt{k}}\nabla u=\nabla N^{-1/2}(u).
\end{align*}
The first equality (i.e., that the projection $\Pi$ commutes with the principal value integral) follows as in the proof of Lemma \ref{lem:L2-kernel}.
\end{proof}

We now pass from the $q$-Fock space $\mathcal{F}_q(\mathbb R^d_A)$ to the corresponding $q$-Araki--Woods algebra $\Gamma_q(\mathbb R^d_A)$. For this purpose, we first record two consequences of the second quantization construction from Subsection \ref{sec:prelim-qaw}. The proof follows the strategy of \cite[Lemma 3.1]{LP1999}.

\begin{lemma}\label{lem600}
\begin{enumerate}[(i)]
\item Let $\T:\real^l\to\real^h$ be either an orthogonal transformation when $h=l$, or the canonical projection $(x_1,\cdots,x_l)\in \mathbb R^l\mapsto (x_1,\cdots,x_h)\in \mathbb R^h$ when $1\leq h<l$, and denote its complexification from $\com^l$ to
$\com^h$ by the same symbol. Then, for $T=I\otimes\T:\real_A^{ld}\longrightarrow\real_A^{hd}$, we have
\begin{equation}\label{eq609}
[\Gamma_{A_l,A_h}(I\otimes \T)](X)=\Gamma(I\otimes \T)\cdot X\cdot \Gamma(I\otimes \T^*),\qquad X\in\Gamma_q(\real_A^{ld}).
\end{equation}
In particular, if $h=l$, then $\Gamma_{A_l}(I\otimes\T)$ is a $*$-automorphism of $\Gamma_q(\real_A^{ld})$. If $1\leq h<l$, then, under the canonical identification of $\Gamma_q(\real_A^{hd})$ with its copy inside $\Gamma_q(\real_A^{ld})$, $\Gamma_{A_l,A_h}(I\otimes \T)$ is the state-preserving conditional expectation.

\item
Let $S:\real^l\to\real^l$ be a contraction, and denote its complexification by the same symbol. Then there exist $m\geq l$, an orthogonal transformation $\T:\real^m\to\real^m$, and the canonical embedding $j:\real^l\to\real^m$ such that
\[S=j^*\T j.\]
Denoting the complexifications of $\T$ and $j$ by the same symbols, on
$\phi_{ld}^{-1}\bigl(\F_q^{\rm finite}(\real_A^{ld})\bigr)$ we have
\begin{equation}\label{eq611}
\Gamma_{A_l}(I\otimes S)=\Gamma_{A_m,A_l}(I\otimes j^*)\circ\Gamma_{A_m}(I\otimes\T)
\circ h_{ld,md}^{-1}.
\end{equation}
\end{enumerate}
\end{lemma}

\begin{proof}
We first prove (i). Since
\[(I\otimes\T)(U_t\otimes{\rm Id}_{\com^l})=(U_t\otimes{\rm Id}_{\com^h})(I\otimes\T),\qquad t\in\real,\]
the map $\Gamma_{A_l,A_h}(T)$ is well defined.

In either of the two cases in (i), we have $(I\otimes \T)(I\otimes \T)^*=I\otimes \T \T^*=I$ on $\real_A^{hd}$. Let $g\in\real_A^{ld}$. By definition, it is straightforward to see
\begin{align}
\label{eq607} & \Gamma(I\otimes \T)a_+(g)\Gamma(I\otimes \T^*)=a_+((I\otimes \T)g),\\
\label{eq608} & \Gamma(I\otimes \T)a_-(g)\Gamma(I\otimes \T^*)=a_-((I\otimes \T)g).
\end{align}

Let $g=g_1\otimes\cdots\otimes g_k$ be a simple tensor. Recall that
\[\phi_{ld}^{-1}(g)=\sum_{\substack{I=\{i_1<\cdots<i_r\},\\I^c=\{j_1<\cdots<j_s\},\\I\cup I^c=[k]}}q^{i(I)}a_+^I(g)a_-^{I^c}(g)\]
by the Wick formula \eqref{lem-Wick}. Then, by \eqref{eq607} and \eqref{eq608}, we obtain
\begin{align*}
\Gamma(I\otimes \T)\phi_{ld}^{-1}(g)\Gamma(I\otimes \T^*)
&=\sum_{\substack{I=\{i_1<\cdots<i_r\},\\I^c=\{j_1<\cdots<j_s\},\\I\cup I^c=[k]}}q^{i(I)}a_+^I((I\otimes \T)^{\otimes k}g)a_-^{I^c}((I\otimes \T)^{\otimes k}g) \\
&=\phi_{hd}^{-1}\bigl((I\otimes \T)^{\otimes k}g\bigr)=[\Gamma_{A_l,A_h}(I\otimes \T)]\left(\phi_{ld}^{-1}(g)\right).
\end{align*}
Since $\phi_{ld}^{-1}\bigl(\F_q^{\rm finite}(\real_A^{ld})\bigr)$ is $w^*$-dense in $\Gamma_q(\real_A^{ld})$, this proves \eqref{eq609}.

Suppose first that $h=l$ and $\T$ is orthogonal. In this case, by Proposition \ref{prop200} (a),
\begin{align*}
\Gamma(I\otimes \T)^*&=\Gamma(I\otimes \T^*)\\
\Gamma(I\otimes \T)^*\Gamma(I\otimes \T)&=\Gamma(I\otimes \T^*\T)=I,\\
\Gamma(I\otimes \T)\Gamma(I\otimes \T)^*&=\Gamma(I\otimes \T\T^*)=I.
\end{align*}
Thus $\Gamma(I\otimes \T)$ is unitary. It follows from \eqref{eq609} that $\Gamma_{A_l}(I\otimes \T)$ is a $*$-automorphism of $\Gamma_q(\real_A^{ld})$, with inverse $\Gamma_{A_l}(I\otimes \T^*)$.

Now suppose that $1\leq h<l$ and $\T$ is the canonical projection from $\real^l$ onto $\real^h$. Then $I\otimes\T^*:\real_A^{hd}\hookrightarrow \real_A^{ld}$ is the canonical embedding and $\T\T^*=I_{\real^h}$. Hence, by \eqref{eq231},
\[\Gamma_{A_l,A_h}(I\otimes\T)\circ\Gamma_{A_h,A_l}(I\otimes\T^*)=\Gamma_{A_h}(I\otimes\T\T^*)={\rm Id}_{\Gamma_q(\real_A^{hd})}.\]
Therefore, identifying $\Gamma_q(\real_A^{hd})$ with its canonical copy in $\Gamma_q(\real_A^{ld})$ through $\Gamma_{A_h,A_l}(I\otimes\T^*)$, the map
\[\mathsf E=\Gamma_{A_h,A_l}(I\otimes\T^*)\circ\Gamma_{A_l,A_h}(I\otimes\T)\]
is a normal unital completely positive projection onto that subalgebra. Since both $\Gamma_{A_h,A_l}(I\otimes\T^*)$ and $\Gamma_{A_l,A_h}(I\otimes\T)$ preserve the vacuum states, $\mathsf E$ also preserves the vacuum state. By Tomiyama's theorem, $\mathsf E$ is therefore the state-preserving conditional expectation onto the canonical copy of $\Gamma_q(\real_A^{hd})$.

For (ii), the factorization $S=j^*\T j$ follows from the Sz.-Nagy dilation theorem \cite{SN1959}. By Proposition \ref{prop200} (a), we have
\[\Gamma(I\otimes S)=\Gamma(I\otimes j^*)\Gamma(I\otimes\T)\Gamma(I\otimes j).\]
For a simple tensor $\displaystyle f=(f_1\otimes g_1)\otimes\cdots\otimes(f_k\otimes g_k)\in \F_q^{\rm finite}(\real_A^{ld})$, by \eqref{eq322}, 
\begin{align*}
(\phi_{md}\circ h_{ld,md}^{-1}\circ \phi_{ld}^{-1})(f)=(f_1\otimes j(g_1))\otimes\cdots\otimes (f_k\otimes j(g_k))=\Gamma(I\otimes j)(f).
\end{align*}
This means that, on $\phi_{ld}^{-1}\bigl(\F_q^{\rm finite}(\real_A^{ld})\bigr)$, we have
\begin{equation}\label{eq610}
\phi_{md}^{-1}\circ \Gamma(I\otimes j)\circ \phi_{ld}=h_{ld,md}^{-1}.
\end{equation}
Therefore, for any $X\in\phi_{ld}^{-1}\bigl(\F_q^{\rm finite}(\real_A^{ld})\bigr)$, by \eqref{eq230} and Proposition \ref{prop200}(a), we have
\begin{align*}
\Gamma_{A_l}(I\otimes S)(X)&=\phi_{ld}^{-1}\Gamma(I\otimes S)\phi_{ld}(X)\\
&=\phi_{ld}^{-1}\circ\Gamma(I\otimes j^*)\circ\Gamma(I\otimes\T)\circ\Gamma(I\otimes j)\circ\phi_{ld}(X)\\
&=\Gamma_{A_m,A_l}(I\otimes j^*)\circ\phi_{md}^{-1}\circ\Gamma(I\otimes\T)\circ\Gamma(I\otimes j)\circ\phi_{ld}(X)\\
&=\Gamma_{A_m,A_l}(I\otimes j^*)\circ\Gamma_{A_m}(I\otimes\T)\circ\left(\phi_{md}^{-1}\circ\Gamma(I\otimes j)\circ\phi_{ld}(X)\right)\\
&=\Gamma_{A_m,A_l}(I\otimes j^*)\circ\Gamma_{A_m}(I\otimes\T)\circ h_{ld,md}^{-1}(X).
\end{align*}
This proves \eqref{eq611}.
\end{proof}

 Let us define $\A_d=\phi_d^{-1}\bigl(\F_q^{\rm finite}(\real_A^d)\bigr)$ and $\A_d^0=\{X\in\A_d:\f(X)=0\}$. By Lemma \ref{lem:qaw-wick}, $\A_d$ is $w^*$-dense in $\Gamma_q(\real_A^d)$. We now lift the operator $N:\F_q^{\rm finite}(\real_A^d)\longrightarrow \F_q^{\rm finite}(\real_A^d)$ to $\wt{N}:\A_d\rightarrow \A_d$ through the GNS identification, and similarly lift $\nabla$, $\R$, and $\Gamma(I\otimes R_\theta)$ to the corresponding $\wt{\nabla}, \wt{\mathcal R}_d, \alpha_{\theta}$ as follows.

\begin{definition}
\begin{enumerate}
\item We define $\wt{N}=\phi_d^{-1}N\phi_d:\A_d\rightarrow \A_d$ and $\wt{\nabla}=\phi_{2d}^{-1}\nabla\phi_d:\A_d\rightarrow \A_{2d}$. The Riesz transform on $\Gamma_q(\real_A^d)$ is then defined by
\begin{equation}\label{eq:qaw-Riesz}
\wt{\R}_d=\wt{\nabla}\circ \wt{N}^{-1/2}=\phi_{2d}^{-1}\circ \R\circ \phi_d:\A_d^0 \longrightarrow \A_{2d}.
\end{equation}    
\item For any $\theta\in\real$, we define
\[\alpha_\theta=\Gamma_{A_2}(I\otimes R_\theta)=\phi_{2d}^{-1}\circ \Gamma(I\otimes R_{\theta})\circ \phi_{2d}:\Gamma_q(\real_A^{2d}) \longrightarrow \Gamma_q(\real_A^{2d}).\]
\end{enumerate}

\end{definition}

\begin{remark}\label{rmk600}
\begin{enumerate}
\item Since $R_\theta$ is orthogonal, Lemma \ref{lem600}(i) gives
\begin{equation}\label{eq605}
\alpha_\theta(X)=\Gamma(I\otimes R_\theta)X\Gamma(I\otimes R_{-\theta}),\qquad X\in\Gamma_q(\real_A^{2d}).
\end{equation}
Thus $(\alpha_\theta)_{\theta\in\real}$ is a one-parameter group of state-preserving $*$-automorphisms. Moreover, since $I\otimes R_\theta$ commutes with $U_t\otimes{\rm Id}_{\com^2}$, we have $\alpha_\theta\circ\sigma_t^\f=\sigma_t^\f\circ\alpha_\theta$ for all $t,\theta\in\real$. Thus, $\alpha_{\theta}$ extends to an isometric isomorphism $\alpha_{\theta}^{(p)}$ on $L_p(\Gamma_q(\mathbb R^{2d}_A))$ for all $1\leq p<\infty$.
\item By Proposition \ref{prop200} (c), $D=\d\Gamma(I\otimes P)$ satisfies $\Gamma(I\otimes R_\theta)=e^{i\theta D}$, so \eqref{eq605} is written as
\[\alpha_\theta(X)=e^{i\theta D}Xe^{-i\theta D}.\]
Furthermore, $\wt D=\phi_{2d}^{-1}\circ D\circ\phi_{2d}$ satisfies $\alpha_\theta=e^{i\theta\wt D}$.
\item By \eqref{eq600} and \eqref{eq612}, we obtain
\begin{align*}
\wt\nabla&=\phi_{2d}^{-1}\circ\nabla\circ\phi_d= \phi_{2d}^{-1}\circ iD\circ j_{d,2d}\circ\phi_d \\
&=\phi_{2d}^{-1}\circ iD\circ\phi_{2d}\circ h_{d,2d}^{-1} = i\wt D\circ h_{d,2d}^{-1}.
\end{align*}
Thus the gradient is the infinitesimal action of the rotation after the canonical embedding into the doubled $q$-Araki--Woods algebra.
\end{enumerate}    
\end{remark}

We also lift the operators $\Pi$ defined on $\mathcal{F}^{\rm finite}_q(\mathbb{R}^{2d}_A)$ to $\wt{\Pi}_d$ defined on $\A_{2d}$ by setting
\begin{equation}
\wt{\Pi}_d=\phi_{2d}^{-1}\circ\Pi\circ\phi_{2d}.
\end{equation}
In particular, the Fock-space kernel representation \eqref{KR:eq-3} is lifted to the following kernel representation of $\wt{\R}_d$.

\begin{lemma}\label{lem-ker-R}
For any $X\in\mathcal A_d^0$, we have
\begin{equation}\label{eq:ker-R}
\wt{\R}_d(X)=\frac{1}{\sqrt{2\pi}}\wt{\Pi}_d\circ{\rm p.v.}\int_{-\pi/2}^{\pi/2}\frac{{\rm sgn}(\theta)}{\sqrt{-\log\cos^2\theta}}(\alpha_\theta\circ h_{d,2d}^{-1})(X)\,d\theta.
\end{equation}

\end{lemma}

\begin{proof}
By \eqref{KR:eq-3}, we have
\begin{align*}
&\wt{\R}_d(X)= (\phi_{2d}^{-1}\circ \R\circ \phi_d)(X)\\
&=\frac{1}{\sqrt{2\pi}} \phi_{2d}^{-1} \circ  \Pi\circ \,{\rm p.v.} \left ( \int_{-\pi/2}^{\pi/2}\frac{{\rm sgn}(\theta)}{\sqrt{-\log\cos^2\theta}}\left (\Gamma(I\otimes R_\theta)\circ j_{d,2d}\circ \phi_d\right )(X)\,d\theta\right )\\
&=\frac{1}{\sqrt{2\pi}} \wt{\Pi}_d\circ \,{\rm p.v.} \left ( \int_{-\pi/2}^{\pi/2}\frac{{\rm sgn}(\theta)}{\sqrt{-\log\cos^2\theta}}\left (\phi_{2d}^{-1}\circ \Gamma(I\otimes R_\theta)\circ j_{d,2d}\circ \phi_d\right )(X)\,d\theta\right ).
\end{align*}
Since $\phi_{2d}\circ\alpha_\theta=\Gamma(I\otimes R_\theta)\circ\phi_{2d}$ and $\phi_{2d}^{-1}\circ j_{d,2d}\circ \phi_d=h_{d,2d}^{-1}$ by \eqref{eq600}, we obtain
\begin{align*}
\left (\Gamma(I\otimes R_\theta)\circ j_{d,2d}\circ \phi_d\right )(X)&=\left (\Gamma(I\otimes R_\theta)\circ\phi_{2d}\circ \phi_{2d}^{-1}\circ j_{d,2d}\circ \phi_d\right )(X)\\
&=(\phi_{2d}\circ \alpha_\theta)
\bigl(h_{d,2d}^{-1}(X)\bigr).    
\end{align*}
The desired identity follows.
\end{proof}

\subsection{$L_p$-estimates for the Riesz transform}

Let us write the kernel function of the Riesz transform $\wt{\mathcal{R}}_d$ as $\kappa(\theta)=\displaystyle \frac{{\rm sgn}(\theta)}{\sqrt{-\log\cos^2\theta}}$ for $-\displaystyle \frac{\pi}{2}<\theta<\frac{\pi}{2}$. For $X\in\A_{2d}$, define
\begin{equation}\label{eq613}
T_d(X)=\frac{1}{\sqrt{2\pi}}\,{\rm p.v.}\int_{-\pi/2}^{\pi/2}\alpha_\theta(X)\kappa(\theta)\,d\theta.
\end{equation}
Then Lemma \ref{lem-ker-R} gives the factorization
\begin{equation}\label{eq625}
\wt{\R}_d=\wt{\Pi}_d\circ T_d\circ h_{d,2d}^{-1}:\A_d^0\longrightarrow \A_{2d}.
\end{equation}

We first establish the $L_p$-boundedness of the singular integral $T_d$.

\begin{lemma}\label{lem:qaw-T-bound}
Let $1<p<\infty$. Then $T_d$ extends to a bounded operator
$T_d^{(p)}$ on $L_p(\Gamma_q(\real_A^{2d}))$, and there exists a
constant $C_p>0$, depending only on $p$, such that
\[\|T_d^{(p)}\|_{L_p(\Gamma_q(\real_A^{2d}))\rightarrow L_p(\Gamma_q(\real_A^{2d}))}
\leq C_p.\]
In particular, $C_p$ is independent of $d$, $q$, and $A$.
\end{lemma}

\begin{proof}
By Remark \ref{rmk600} (1), $\alpha_\theta^{(p)}$ is an isometric isomorphism on $L_p(\Gamma_q(\mathbb R^{2d}_A))$ for all  $\theta\in\real$. Thus, exactly as in the proof of Lemma \ref{lem:Tn-bound}, the Coifman--Weiss transference theorem and the boundedness of the vector-valued Hilbert transform on $L_p\bigl(\mathbb T;L_p(\Gamma_q(\real_A^{2d}))\bigr)$ imply
the desired conclusion.
\end{proof}

We next establish the $L_p$-boundedness of $\wt{\Pi}_d$. Recall that
$E=\begin{bmatrix}0&0\\0&1\end{bmatrix}$ and
$N_2=\d\Gamma(I\otimes E)$. For $t\geq0$, let us consider
\begin{equation}\label{eq614}
S_t=\Gamma_{A_2}(I\otimes e^{-tE})
=\phi_{2d}^{-1}\circ \Gamma(I\otimes e^{-tE})\circ \phi_{2d}:
\Gamma_q(\real_A^{2d})
\longrightarrow
\Gamma_q(\real_A^{2d}).
\end{equation}
Since $e^{-tE}$ is a real contraction and
\[(I\otimes e^{-tE})(U_s\otimes{\rm Id}_{\com^2})=(U_s\otimes{\rm Id}_{\com^2})(I\otimes e^{-tE}) \]
for all $s\in\real$, the map $S_t$ is normal unital completely positive and state-preserving. In addition, we have $S_0={\rm id}$ and 
\[S_t\circ S_{t'}=\Gamma_{A_2}(I\otimes e^{-tE}e^{-t'E})=S_{t+t'}\]
by \eqref{eq231}.  By Proposition \ref{prop200}(c), we have
\[\Gamma(I\otimes e^{-tE})=e^{-t\d\Gamma(I\otimes E)}=e^{-tN_2}.\]
This implies
\begin{equation}\label{eq615}
S_t=\phi_{2d}^{-1}\circ e^{-tN_2}\circ\phi_{2d}.
\end{equation}

Since $e^{-tN_2}$ is self-adjoint on $\F_q(\real_A^{2d})$, for any
$X,Y\in\Gamma_q(\real_A^{2d})$ we have
\begin{align*}
\f(X^*S_t(Y))=\left\langle \phi_{2d}(X),e^{-tN_2}\phi_{2d}(Y) \right\rangle_q=
\left\langle e^{-tN_2}\phi_{2d}(X),\phi_{2d}(Y)\right\rangle_q= \f(S_t(X)^*Y).
\end{align*}
Thus $(S_t)_{t\geq0}$ is symmetric with respect to the vacuum state.
Moreover, since $S_t\circ\sigma_s^\f= \sigma_s^\f\circ S_t$ for all $s\in\real$ and $\ t\geq 0$, the map $S_t$ admits a canonical contractive extension
\[S_t^{(p)}: L_p(\Gamma_q(\real_A^{2d})) \longrightarrow L_p(\Gamma_q(\real_A^{2d}))\]
for all $1\leq p<\infty$.

On $L_2(\Gamma_q(\real_A^{2d}))$, \eqref{eq615} and the strong continuity of $(e^{-tN_2})_{t\geq0}$ imply
\[\lim_{t\to0}\|S_t^{(2)}(X)-X\|_{L_2(\Gamma_q(\real_A^{2d}))}=0.\]
Since $S_t^{(p)}$ is contractive, interpolation and density imply that
\[\lim_{t\to0} \|S_t^{(p)}(X)-X\|_{L_p(\Gamma_q(\real_A^{2d}))}=0, \qquad X\in L_p(\Gamma_q(\real_A^{2d})),
\]
for all $1<p<\infty$. Hence $(S_t^{(p)})_{t\geq0}$ is a strongly continuous symmetric contraction semigroup on $L_p(\Gamma_q(\real_A^{2d}))$.

\begin{lemma}\label{lem:qaw-Pi-bound}
Let $1<p<\infty$. Then $\wt{\Pi}_d$ extends to a bounded operator on
$L_p(\Gamma_q(\real_A^{2d}))$, and there exists a constant $C_p>0$,
depending only on $p$, such that
\[\|\wt{\Pi}_d^{(p)}\|_{L_p(\Gamma_q(\real_A^{2d}))\to L_p(\Gamma_q(\real_A^{2d}))}\leq C_p.\]
\end{lemma}

\begin{proof}
{\color{purple}By Stein's theorem for symmetric contraction semigroups \cite{Ste1970,JLMX2006}}, there exist $\theta_p>0$ and $B_p>0$, depending only on $p$, such that $(S_t^{(p)})_{t\geq0}$ admits an analytic extension $S_z^{(p)}$ to $\Sigma_{\theta_p}=\{z\in\com:|\arg z|<\theta_p\}$ satisfying $\sup_{z\in\Sigma_{\theta_p}}\|S_z^{(p)}\|\leq B_p$.

On $\A_{2d}$, the semigroup $S_t$ in \eqref{eq615} extends analytically to
\begin{equation}\label{eq617}
S_z=\phi_{2d}^{-1}\circ e^{-zN_2}\circ\phi_{2d}, \qquad z\in\Sigma_{\theta_p}.
\end{equation}

Choose $a_p>0$ such that $\arctan\frac{\pi}{a_p}<\theta_p$. Then $a_p+is\in\Sigma_{\theta_p}$ for all $-\pi\leq s\leq\pi$ and, by \eqref{KR:eq-2} and \eqref{eq617}, we obtain
\begin{equation}\label{eq618}
\wt{\Pi}_d=\frac{e^{a_p}}{2\pi}\int_{-\pi}^{\pi}e^{is}S_{a_p+is}\,ds
\end{equation}
on $\A_{2d}$. Consequently, for their $L_p$-extensions, we obtain
\begin{align*}
\|(\wt{\Pi}_d)^{(p)}\|_{L_p(\Gamma_q(\real_A^{2d}))\to L_p(\Gamma_q(\real_A^{2d}))} \leq \frac{e^{a_p}}{2\pi} \int_{-\pi}^{\pi} \|S_{a_p+is}^{(p)}\|_{L_p(\Gamma_q(\real_A^{2d}))\to L_p(\Gamma_q(\real_A^{2d}))} \,ds \leq e^{a_p}B_p.
\end{align*}
Since $a_p$ and $B_p$ depend only on $p$, setting $C_p=e^{a_p}B_p$ completes the proof.
\end{proof}

We now pass to the $L_p$-level. Let $D_d$ and $D_{2d}$ denote the density operators associated with the vacuum states on
$\Gamma_q(\real_A^d)$ and $\Gamma_q(\real_A^{2d})$, respectively. Let $(h_{d,2d}^{-1})^{(p)}$ denote the canonical $L_p$-extension of $h_{d,2d}^{-1}$. Since
\[h_{d,2d}^{-1}=\Gamma_{A_1,A_2}(I\otimes j)=\phi_{2d}^{-1}\circ \Gamma(I\otimes j)\circ \phi_d\]
is a state-preserving $*$-monomorphism intertwining the corresponding modular automorphism groups, $(h_{d,2d}^{-1})^{(p)}$ is an isometric embedding. For $1<p<\infty$, let
\begin{equation}
\wt{\R}_d^{(p)}=(\wt{\Pi}_d)^{(p)}\circ T_d^{(p)} \circ (h_{d,2d}^{-1})^{(p)}.
\end{equation}
By \eqref{eq625}, for any $X\in\A_d^0$,
\[\wt{\R}_d^{(p)}\left(D_d^{1/2p}XD_d^{1/2p}\right)=D_{2d}^{1/2p}\wt{\R}_d(X)D_{2d}^{1/2p}.\]
Since $(h_{d,2d}^{-1})^{(p)}$ is an isometry, Lemma \ref{lem:qaw-T-bound} and Lemma \ref{lem:qaw-Pi-bound} imply
\begin{equation}\label{eq624}
\left\|\wt{\R}_d^{(p)}\left(D_d^{1/2p}XD_d^{1/2p}\right) \right\|_p\leq C_p\left\| D_d^{1/2p}XD_d^{1/2p} \right\|_p, \qquad X\in\A_d^0    
\end{equation}
and $C_p$ depends only on $p$. Thus $\wt{\R}_d^{(p)}$ extends uniquely to 
\[L_p^0(\Gamma_q(\mathbb R^d_A))=\overline{\left\{D_d^{1/2p}XD_d^{1/2p}:X\in\A_d^0\right\}}\subseteq L_p(\Gamma_q(\real_A^d)).\]

We are now ready to state the main result of this section.

\begin{thm}\label{thm-main2}
Let $1<p<\infty$. There exists a constant $C_p>0$, depending only on $p$, such that
\[C_p^{-1}\|X\|_{L_p(\Gamma_q(\real_A^d))} \leq \|\wt{\R}_d^{(p)}(X)\|_{L_p(\Gamma_q(\real_A^{2d}))} \leq C_p \|X\|_{L_p(\Gamma_q(\real_A^d))}\]
for every $X\in L_p^0(\Gamma_q(\real_A^d))$. In particular, $C_p$ is independent of $d$, $q$, and $A$.
\end{thm}

\begin{proof}
The upper estimate follows from \eqref{eq624}, so we focus on the reverse estimate. Let us compute the adjoint of $\wt{\R}_d=\wt{\Pi}_d\circ T_d\circ h_{d,2d}^{-1}$. First, let us compute $(h_{d,2d}^{-1})^*$. For $X\in\A_d$ and $Y\in\A_{2d}$,
\begin{align}
\notag \left\langle h_{d,2d}^{-1}(X),Y\right\rangle_2 &=\left\langle \phi_{2d}h_{d,2d}^{-1}(X),\phi_{2d}(Y)\right\rangle_q=\left\langle j_{d,2d}\phi_d(X),\phi_{2d}(Y) \right\rangle_q\\
\label{eq629}&=\left\langle \phi_d(X),j_{d,2d}^*\phi_{2d}(Y) \right\rangle_q=\left\langle X,\phi_d^{-1}j_{d,2d}^*\phi_{2d}(Y) \right\rangle_2
\end{align}
by \eqref{eq600} and \eqref{eq230}. Let us denote by
\[\mathcal E_d =\Gamma_{A_2,A_1}(I\otimes j^*)=\phi_d^{-1}\circ j_{d,2d}^*\circ \phi_{2d}: \Gamma_q(\real_A^{2d}) \longrightarrow \Gamma_q(\real_A^d).\]
Then, by Lemma \ref{lem600}(i), $\mathcal E_d$ is the state-preserving conditional expectation onto the canonical copy of $\Gamma_q(\real_A^d)$ in $\Gamma_q(\real_A^{2d})$ and satisfies $(h_{d,2d}^{-1})^*=\mathcal E_d$ on $L_2(\Gamma_q(\mathbb R^{2d}_A))$ by \eqref{eq629}.

Second, let us compute $T_d^*$. Since $\alpha_\theta^*=\alpha_{-\theta}$ on $L_2(\Gamma_q(\mathbb R^{2d}_A))$ and $\kappa(-\theta)=-\kappa(\theta)$, we have
\begin{align*}
T_d^*=\frac{1}{\sqrt{2\pi}}\,{\rm p.v.}\int_{-\pi/2}^{\pi/2}\alpha_{-\theta}\kappa(\theta)\,d\theta =-\frac{1}{\sqrt{2\pi}}\, {\rm p.v.}\int_{-\pi/2}^{\pi/2}\alpha_\theta\kappa(\theta)\,d\theta = -T_d.
\end{align*}

Lastly, $\wt{\Pi}_d^*=\wt{\Pi}_d$. Therefore, we obtain
\begin{equation}\label{eq626}
\wt{\R}_d^*=-\mathcal E_d\circ T_d\circ\wt{\Pi}_d
\end{equation}
on the dense subspace $\A_{2d}$ in $L_2(\Gamma_q(\mathbb R^{2d}_A))$.

Note that $T_d$ and $\wt{\Pi}_d$ admit bounded $L_p$-extensions by Lemma \ref{lem:qaw-T-bound} and Lemma \ref{lem:qaw-Pi-bound}. Also, since $\mathcal E_d$ is a state-preserving conditional expectation, we have $\|\mathcal E_d^{(p)}\|\leq 1$. Thus \eqref{eq626} induces a bounded operator
\begin{equation}\label{eq626-p}
(\wt{\R}_d^*)^{(p)}=-\mathcal E_d^{(p)}\circ T_d^{(p)} \circ (\wt{\Pi}_d)^{(p)}
\end{equation}
from $L_p(\Gamma_q(\real_A^{2d}))$ to $L_p(\Gamma_q(\real_A^d))$. By contractivity of $\mathcal E_d^{(p)}$, Lemma
\ref{lem:qaw-T-bound} and Lemma \ref{lem:qaw-Pi-bound}, we obtain
\begin{equation}\label{eq626-bound}
\|(\wt{\R}_d^*)^{(p)}\|_{L_p(\Gamma_q(\real_A^{2d}))\longrightarrow L_p(\Gamma_q(\real_A^{d}))} \leq C_p.
\end{equation}
Here, $C_p$ depends only on $p$.

On the other hand, by Lemma \ref{lem601}, for $X,Y\in\A_d^0$,
\begin{align*}
\left\langle \wt{\R}_d(X),\wt{\R}_d(Y) \right\rangle_2&=\left\langle \phi_{2d}\wt{\R}_d(X),\phi_{2d}\wt{\R}_d(Y) \right\rangle_q\\
&=\left\langle\R\phi_d(X),\R\phi_d(Y)\right\rangle_q=\left\langle\phi_d(X),\phi_d(Y)\right\rangle_q=\langle X,Y\rangle_2.
\end{align*}
Hence we obtain $\wt{\R}_d^*\wt{\R}_d=I$ on the dense subspace $\A_d^0$ in $L_2(\Gamma_q(\mathbb{R}^{d}_A))$.

By the construction of the $L_p$-extensions, for any $Y\in\A_{2d}$,
\[(\wt{\R}_d^*)^{(p)}\left(D_{2d}^{1/2p}YD_{2d}^{1/2p}\right)=D_d^{1/2p}\wt{\R}_d^*(Y)D_d^{1/2p}.\]
Consequently,
\[(\wt{\R}_d^*)^{(p)}\wt{\R}_d^{(p)}\left(D_d^{1/2p}XD_d^{1/2p}\right)=D_d^{1/2p}XD_d^{1/2p}\]
for any $X\in\A_d^0$. Therefore, by \eqref{eq626-bound},
\[\left\|D_d^{1/2p}XD_d^{1/2p}\right\|_p\leq
C_p\left\|\wt{\R}_d^{(p)}\left(D_d^{1/2p}XD_d^{1/2p}\right)\right\|_p.\]
By density, this extends to any $X\in L_p^0(\Gamma_q(\real_A^d))$, and hence
\[ C_p^{-1} \|X\|_{L_p(\Gamma_q(\real_A^d))} \leq \|\wt{\R}_d^{(p)}(X)\|_{L_p(\Gamma_q(\real_A^{2d}))}.\]
Together with the upper estimate, this completes the proof.
\end{proof}

We conclude with a remark concerning square-function estimates. We do not obtain square-function estimates analogous to those in the baby Fock setting. Although the gradient admits the realization $\wt\nabla=i\wt D\circ h_{d,2d}^{-1}$, we do not have a suitable decomposition of $\wt\nabla$ into coordinate-wise partial derivatives analogous to the operators $\partial_j^\e$ in the baby Fock model. Consequently, the Riesz transform estimate in Theorem \ref{thm-main2} does not directly yield a coordinate square-function estimate.

	\bigskip

\appendix

\section{Carr\'e du champ of the Ornstein--Uhlenbeck semigroup on biased hypercubes}\label{appendix-A}

We record the explicit formula for the carr\'e du champ associated with the Ornstein--Uhlenbeck semigroup $(P_t)_{t\geq0}$ on the biased hypercube $(\Omega_n,\nu)$.

\begin{proposition}\label{hz:propA1}
For any $f,g\in L_\infty(\Omega_n,\nu)$, we have
\begin{equation}\label{eq:Gamma-1}
\Gamma(f,g)= \frac12 \sum_{j=1}^n (1+\zeta_j^2)\overline{D_jf}\,D_jg.
\end{equation}
In particular, for any mean-zero $f\in L_\infty(\Omega_n,\nu)$,
\[\Gamma(N^{-1/2}f,N^{-1/2}f)=\ce_{\omega'}\left[\sum_{j=1}^n(1\otimes1-z_j\otimes z_j)(|R_jf|^2\otimes1)\right],\]
where $\ce_{\omega'}$ denotes the conditional expectation obtained by integrating with respect to the second variable on $(\Omega_n\times\Omega_n,\nu\times\nu)$.
\end{proposition}

\begin{proof}
By definition,
\[\Gamma(f,g)=\frac12\sum_{j=1}^n\left(\overline f\,D_j^*D_jg+\overline{D_j^*D_jf}\,g-D_j^*D_j(\overline fg)\right)=:\sum_{j=1}^n\Gamma_j(f,g).\]
It suffices to compute $\Gamma_j(\zeta_A,\zeta_B)$ for $A,B\subseteq[n]$.

Suppose first that $j\notin A\cap B$. We may assume that $j\notin A$. Since $\zeta_A$ is independent of the $j$-th
coordinate and $D_j^*D_j=I-E_j$, we have $D_j^*D_j(\zeta_A)=0$ and $D_j^*D_j(\zeta_A\zeta_B)=\zeta_A D_j^*D_j(\zeta_B)$. Therefore,
\[2\Gamma_j(\zeta_A,\zeta_B)=\zeta_A D_j^*D_j(\zeta_B)-D_j^*D_j(\zeta_A\zeta_B)=0.\]

Now suppose that $j\in A\cap B$. Since
$E_j(\zeta_j^2)=\int_{\{-1,1\}}\zeta_j^2\,d\nu_j=1$, we have
\[D_j^*D_j(\zeta_j^2)=(I-E_j)(\zeta_j^2)=\zeta_j^2-1.\]
Hence we obtain
\begin{align*}
D_j^*D_j(\zeta_A\zeta_B)=\zeta_{A\setminus\{j\}}\zeta_{B\setminus\{j\}}D_j^*D_j(\zeta_j^2)=\zeta_{A\setminus\{j\}}
\zeta_{B\setminus\{j\}}(\zeta_j^2-1).
\end{align*}
Since $D_j^*D_j(\zeta_A)=\zeta_A$ and $D_j^*D_j(\zeta_B)=\zeta_B$, we obtain
\begin{align*}
2\Gamma_j(\zeta_A,\zeta_B)&= 2\zeta_A\zeta_B-\zeta_{A\setminus\{j\}}\zeta_{B\setminus\{j\}}(\zeta_j^2-1)\\
&=(1+\zeta_j^2)\zeta_{A\setminus\{j\}}\zeta_{B\setminus\{j\}}=(1+\zeta_j^2)D_j(\zeta_A)D_j(\zeta_B).
\end{align*}
Thus, if $f=\sum_{A\subseteq[n]}f_A\zeta_A$ and $g=\sum_{B\subseteq[n]}g_B\zeta_B$, then
\begin{align}
\notag&\Gamma(f,g)=\sum_{A,B\subseteq[n]}\overline{f_A}g_B\Gamma(\zeta_A,\zeta_B)\\
\label{eqA1}&=\frac12\sum_{j=1}^n(1+\zeta_j^2)\sum_{A,B\subseteq[n]}\overline{f_A}g_B D_j(\zeta_A)D_j(\zeta_B)= \frac12 \sum_{j=1}^n (1+\zeta_j^2)\overline{D_jf}\,D_jg.
\end{align}

For the second assertion, let $u=N^{-1/2}f$. Then $|R_jf|^2=\frac{(\mu_j^{-2}+\mu_j^2)^2}{4}|D_ju|^2$. For each $\omega\in\Omega_n$, by the proof of Proposition \ref{prop500}, we have
\[\ce_{\omega'}\left[1-z_j(\omega)z_j(\omega')\right]=2\cdot \mathbb{P}_{\omega'}[z_j(\omega)\neq z_j(\omega')]
=\begin{cases}
\displaystyle \frac{2\mu_j^2}{\mu_j^{-2}+\mu_j^2},
& z_j(\omega)=1,\\[8pt]
\displaystyle
\frac{2\mu_j^{-2}}{\mu_j^{-2}+\mu_j^2},
& z_j(\omega)=-1.
\end{cases}\]
If $z_j(\omega)=1$, then $\zeta_j(\omega)=\mu_j^2$, and therefore
\begin{align*}
&\ce_{\omega'}\left[(1-z_j\otimes z_j)(|R_jf|^2\otimes1)\right](\omega)=\ce_{\omega'}\left[(1-z_j(\omega) z_j(\omega')\right] \cdot |R_jf(\omega)|^2\\
&=\frac{2\mu_j^2}{\mu_j^{-2}+\mu_j^2} \frac{(\mu_j^{-2}+\mu_j^2)^2}{4} |D_ju(\omega)|^2=\frac{1+\mu_j^4}{2}|D_ju(\omega)|^2=\frac{1+\zeta_j(\omega)^2}{2}|D_ju(\omega)|^2.
\end{align*}
Similarly, if $z_j(\omega)=-1$, then $\zeta_j(\omega)=-\mu_j^{-2}$ and
\[\ce_{\omega'}\left[(1-z_j\otimes z_j)(|R_jf|^2\otimes1)
\right](\omega)=\frac{1+\zeta_j(\omega)^2}{2}|D_ju(\omega)|^2.\]
Consequently, we have
\[\ce_{\omega'}\left[(1-z_j\otimes z_j)(|R_jf|^2\otimes1)\right]=\frac{1+\zeta_j^2}{2}|D_jN^{-1/2}f|^2.\]
Summing over $j$ and applying \eqref{eqA1} gives the desired identity.
\end{proof}

\bigskip

\section{Carr\'e du champ on the baby Fock model}\label{appendix-B}

We prove the identity \eqref{eq:baby-Gamma} in this section. For any $1\leq i\leq n$, let $N_i^\e=\beta_i^*\beta_i+\beta_{-i}^*\beta_{-i}\in B(H)$, and define $\wt N_i^\e=\phi_{n,\e}^{-1}N_i^\e\phi_{n,\e}\in B(\Gamma_{n,\e})$ and
\[\Gamma_i(X,Y)=\frac12 \left(X^*\wt N_i^\e(Y)+\wt N_i^\e(X)^*Y-\wt N_i^\e(X^*Y)\right).\]
Then $\displaystyle \Gamma_P(X,Y)=\sum_{i=1}^n\Gamma_i(X,Y)$ for all $X,Y\in\Gamma_{n,\e}$. We first record a commutation property that will be used below.

\begin{lemma}\label{lem-B1}
Fix $1\leq i\leq n$. There exists an involutive $*$-automorphism $\vartheta_i:\Gamma_{[n]\setminus\{i\},\e} \longrightarrow
\Gamma_{[n]\setminus\{i\},\e}$ satisfying $\vartheta_i(\gamma_j)=\e(i,j)\gamma_j$ for all $j\in[n]\setminus\{i\}$. Moreover, for any $b\in\Gamma_{[n]\setminus\{i\},\e}$ and $\#\in\{1,*\}$, we have $\beta_{\pm i}^{\#}b=
\vartheta_i(b)\beta_{\pm i}^{\#}$ and $b\beta_{\pm i}^{\#}=\beta_{\pm i}^{\#}\vartheta_i(b)$. Thus, we have
\begin{align}\label{eqB1}
\gamma_i^{\#}b=\vartheta_i(b)\gamma_i^{\#},\qquad b\gamma_i^{\#}=\gamma_i^{\#}\vartheta_i(b).
\end{align}
\end{lemma}

\begin{proof}
For fixed $i$, we define $\displaystyle \delta_j^{(i)}=
\begin{cases}
1,&j=i,\\
\e(i,j),&j\neq i,
\end{cases}$ for all $j\in[n]$, and a unitary $U_i\in B(H)$ by
\[U_i(x_S)=\left(\prod_{s\in S}\delta_{|s|}^{(i)}\right)x_S,\qquad S\subseteq[\pm n].\]
Consider the $*$-automorphism ${\rm Ad}(U_i)$ of $B(H)$. Then we can verify $U_i\beta_{\pm j}U_i^*=\delta_j^{(i)}\beta_{\pm j}$ and $U_i\beta_{\pm j}^*U_i^*=\delta_j^{(i)}\beta_{\pm j}^*$ for all $j\in[n]$. Hence, for any $j\in[n]\setminus\{i\}$,
\[ U_i\gamma_jU_i^*=\e(i,j)\gamma_j.\]
Thus ${\rm Ad}(U_i)$ leaves $\Gamma_{[n]\setminus\{i\},\e}$ invariant. We denote its restriction to this subalgebra by $\vartheta_i$. Since $U_i^2=I$, $\vartheta_i$ is an
involutive $*$-automorphism.

Moreover, for any $j \in [n]\setminus \left\{ i\right\}$, recall that
\[\beta_{\pm i}^{\#}\gamma_j=\e(i,j)\gamma_j\beta_{\pm i}^{\#} =\vartheta_i(\gamma_j)\beta_{\pm i}^{\#},\]
and the same identity holds with $\gamma_j^*$ in place of $\gamma_j$. Since $\Gamma_{[n]\setminus\{i\},\e}$ is finite-dimensional and generated by $\{\gamma_j:j \in [n]\setminus\left\{i\right\}\}$, it follows that
\[\beta_{\pm i}^{\#}b=\vartheta_i(b)\beta_{\pm i}^{\#}\]
for all $b\in\Gamma_{[n]\setminus\{i\},\e}$. Since $\vartheta_i^2={\rm id}$, replacing $b$ by $\vartheta_i(b)$ yields
\[b\beta_{\pm i}^{\#}=\beta_{\pm i}^{\#}\vartheta_i(b).\]
Finally, \eqref{eqB1} follows since $\gamma_i=\mu_i\beta_{-i}+\mu_i^{-1}\beta_i^*$ and $\gamma_i^*=\mu_i\beta_{-i}^*+\mu_i^{-1}\beta_i$.

\end{proof}

We next record the action of $\wt N_i^\e$ and of the partial derivatives on the coordinate decomposition from Proposition
\ref{prop:basis-gamma}.

\begin{lemma}
Let $1\leq i\leq n$ and $a,b,c,d\in\Gamma_{[n]\setminus\{i\},\e}$. Then we have
\begin{align}\label{eqB2}
\left \{\begin{array}{llll}
&\wt N_i^\e(a)=0,\\
&\wt N_i^\e(b\gamma_i)=b\gamma_i,\\
&\wt N_i^\e(c\gamma_i^*)=c\gamma_i^*,\\
&\wt N_i^\e(d\eta_i)=2d\eta_i,
\end{array} \right .
\end{align}
and
\begin{align}\label{eqB3}
\left\{\begin{array}{llll}
&\partial_i^\e(a)=0,\\
&\partial_i^\e(b\gamma_i)=\mu_i^{-1}\vartheta_i(b),\\
&\partial_i^\e(c\gamma_i^*)=0,\\
&\partial_i^\e(d\eta_i)=-\mu_i^{-1}\vartheta_i(d)\gamma_i^*,
\end{array} \right .  ,
\qquad \left\{\begin{array}{llll}
&\partial_{-i}^\e(a)=0,\\
&\partial_{-i}^\e(b\gamma_i)=0,\\
&\partial_{-i}^\e(c\gamma_i^*)=\mu_i\vartheta_i(c),\\
&\partial_{-i}^\e(d\eta_i)=\mu_i\vartheta_i(d)\gamma_i.
\end{array} \right .
\end{align}
\end{lemma}

\begin{proof}
Since $a{\bf 1}$ contains neither the $i$-th nor the $(-i)$-th spin coordinate, we have $\beta_i(a{\bf 1})=\beta_{-i}(a{\bf 1})=0$ and $N_i^\e(a{\bf 1})=0$. Thus, we obtain
\begin{align*}
&\wt N_i^\e(a)=0,\\
&\partial_i^\e(a)=\partial_{-i}^\e(a)=0.    
\end{align*}

Since $\vartheta_i(b){\bf 1}$ and $\vartheta_i(c){\bf 1}$ contain neither of the coordinates $\pm i$, by Lemma \ref{lem-B1}, we have
\begin{align*}
\phi_{n,\e}(b\gamma_i)&=\phi_{n,\e}(\gamma_i\vartheta_i(b))=\gamma_i\vartheta_i(b){\bf 1}=\mu_i^{-1}x_i\vartheta_i(b){\bf 1},\\
\phi_{n,\e}(c\gamma_i^*)&=\phi_{n,\e}(\gamma_i^*\vartheta_i(c))=\gamma_i^*\vartheta_i(c){\bf 1}=\mu_i x_{-i}\vartheta_i(c){\bf 1}.
\end{align*}
This implies $(\beta_i^*\beta_i+\beta_{-i}^*\beta_{-i})(\phi_{n,\e}(b\gamma_i))=\phi_{n,\e}(b\gamma_i)$ and $(\beta_i^*\beta_i+\beta_{-i}^*\beta_{-i})(\phi_{n,\e}(c\gamma_i^*))=\phi_{n,\e}(c\gamma_i^*)$, so we obtain
\begin{align*}
&\wt N_i^\e(b\gamma_i)=b\gamma_i,\\
&\wt N_i^\e(c\gamma_i^*)=c\gamma_i^*.
\end{align*}
Moreover, by definitions of $\partial_{i}^{\e}$ and $\partial_{-i}^{\e}$,
\begin{align*}
&\partial_i^\e(b\gamma_i)=(\phi_{n,\e}^{-1}\circ \beta_i) \left(\mu_i^{-1}x_i\vartheta_i(b){\bf 1}\right)=\phi_{n,\e}^{-1}(\mu_i^{-1}\vartheta_i(b){\bf 1})=\mu_i^{-1}\vartheta_i(b),\\
&\partial_{-i}^\e(b\gamma_i)=(\phi_{n,\e}^{-1}\circ \beta_{-i})(\mu_i^{-1}x_i\vartheta_i(b){\bf 1})=0.
\end{align*}
Similarly, we obtain 
\[\partial_i^\e(c\gamma_i^*)=0,\qquad \partial_{-i}^\e(c\gamma_i^*)=
\mu_i\vartheta_i(c).\] 

Lastly, $d\eta_i=\eta_i d$ by Proposition \ref{prop-indep} (iv). Since
\[\phi_{n,\e}(d\eta_i)=\eta_i d{\bf 1}=(\gamma_i^*\gamma_i-\mu_i^{-2}{\rm Id})d{\bf 1}=x_{-i}x_i d{\bf 1}\]
contains both coordinates $i$ and $-i$, we obtain $N_i^\e\phi_{n,\e}(d\eta_i)=2\phi_{n,\e}(d\eta_i)$, and hence
\[\wt N_i^\e(d\eta_i)=2d\eta_i.\]
Furthermore,
\begin{align*}
&\partial_i^\e(d\eta_i)=(\phi_{n,\e}^{-1}\circ \beta_i)(x_{-i}x_i d{\bf 1})=\phi_{n,\e}^{-1}(-x_{-i}d{\bf 1})\\
&=\phi_{n,\e}^{-1}(-\mu_i^{-1}\gamma_i^* d{\bf 1})=\phi_{n,\e}^{-1}(-\mu_i^{-1}\vartheta_i(d)\gamma_i^*{\bf 1})=-\mu_i^{-1}\vartheta_i(d)\gamma_i^*,
\end{align*}
where we used $\e(i,-i)=\e(i,i)=-1$ and Lemma \ref{lem-B1}. Similarly, we obtain
\begin{align*}
&\partial_{-i}^\e(d\eta_i) =(\phi_{n,\e}^{-1}\circ \beta_{-i})(x_{-i}x_i d{\bf 1})=\phi_{n,\e}^{-1}(x_i d{\bf 1}) \\
&= \phi_{n,\e}^{-1}(\mu_i \gamma_i d{\bf 1})=\phi_{n,\e}^{-1}( \mu_i\vartheta_i(d)\gamma_i{\bf 1})=\mu_i\vartheta_i(d)\gamma_i.
\end{align*}
\end{proof}

\begin{proof}[Proof of \eqref{eq:baby-Gamma}]
Let us fix $1\leq i\leq n$, and first prove
\begin{equation}\label{eqB4}
\Gamma_i(X,Y)=\partial_i^\e(X)^*\partial_i^\e(Y)+\partial_{-i}^\e(X)^*\partial_{-i}^\e(Y),
\qquad X,Y\in\Gamma_{n,\e}.
\end{equation}
By Proposition \ref{prop:basis-gamma}, any $X\in\Gamma_{n,\e}$ admits a unique decomposition
\[X=a+b\gamma_i+c\gamma_i^*+d\eta_i,\qquad a,b,c,d\in\Gamma_{[n]\setminus\{i\},\e}.\]
Since both sides of \eqref{eqB4} are conjugate-linear in $X$ and linear in $Y$, it is sufficient to verify \eqref{eqB4} on pairs of $a, b\gamma_i, c\gamma_i^*, d\eta_i$.

We shall use the identities
\begin{align}\label{eqB5}
\left\{\begin{array}{llll}
\gamma_i^*\gamma_i=\eta_i+\mu_i^{-2}{\rm Id}\\
\gamma_i\gamma_i^*=\mu_i^2{\rm Id}-\eta_i\\
\gamma_i^*\eta_i=-\mu_i^{-2}\gamma_i^*\\
\gamma_i\eta_i=\mu_i^2\gamma_i\\
\eta_i^2=(\mu_i^2-\mu_i^{-2})\eta_i+{\rm Id}
\end{array} \right . .
\end{align}

Let $a\in\Gamma_{[n]\setminus\{i\},\e}$. By \eqref{eqB2} and the decomposition in Proposition \ref{prop:basis-gamma}, we have $\wt N_i^\e(a^*Y)=a^*\wt N_i^\e(Y)$ for all $ Y\in\Gamma_{n,\e}$, and this implies
\begin{align*}
    \Gamma_i(a,Y)&=\frac{1}{2}\left (\wt{N}_i^\e(a)^*Y+a^*\wt{N}^\e_i(Y)-\wt{N}_i^\e(a^*Y)\right )=\frac{1}{2}\left (0+a^*\wt{N}_i^\e(Y)-a^*\wt{N}_i^\e(Y)\right )\\
    &=0=0+0=\partial_i^\e(a)^*\partial_i^\e(Y)+\partial_{-i}^\e(a)^*\partial_{-i}^\e(Y).
\end{align*}
Moreover, \eqref{eqB1} and \eqref{eqB2}, together with Proposition \ref{prop:basis-gamma}, show that $\wt N_i^\e$ is $*$-preserving. Hence
\[\Gamma_i(X,Y)^*=\Gamma_i(Y,X).\]
Thus $\Gamma_i(Y,a)=0=\partial_i^\e(Y)^*\partial_i^\e(a)+\partial_{-i}^\e(Y)^*\partial_{-i}^\e(a)$ as well. 

Let $b,b',c,c',d,d'\in \Gamma_{[n]\setminus\{i\},\e}$. Using \eqref{eqB1}, \eqref{eqB2}, and \eqref{eqB5}, we obtain
\begin{align}
\label{eqB6}
\notag \Gamma_i(b\gamma_i,b'\gamma_i)&=\mu_i^{-2}\vartheta_i(b)^*\vartheta_i(b')=\bigl(\mu_i^{-1}\vartheta_i(b)\bigr)^*\bigl(\mu_i^{-1}\vartheta_i(b')\bigr)+0\\
\notag &=\partial_i^\e(b\gamma_i)^*\partial_i^\e(b'\gamma_i)+\partial_{-i}^\e(b\gamma_i)^*
\partial_{-i}^\e(b'\gamma_i), \\[4pt]
\notag
\Gamma_i(c\gamma_i^*,c'\gamma_i^*)
&=\mu_i^2\vartheta_i(c)^*\vartheta_i(c')=0+\bigl(\mu_i\vartheta_i(c)\bigr)^*\bigl(\mu_i\vartheta_i(c')\bigr)\\
\notag
&=\partial_i^\e(c\gamma_i^*)^*\partial_i^\e(c'\gamma_i^*)+\partial_{-i}^\e(c\gamma_i^*)^*
\partial_{-i}^\e(c'\gamma_i^*), \\[4pt]
\notag \Gamma_i(b\gamma_i,c'\gamma_i^*)
&=\Gamma_i(c\gamma_i^*,b'\gamma_i)=0=0+0=\partial_i^\e(b\gamma_i)^*\partial_i^\e(c'\gamma_i^*)+\partial_{-i}^\e(b\gamma_i)^*\partial_{-i}^\e(c'\gamma_i^*).
\end{align}

Moreover, for the mixed terms involving $\eta_i$, we have
\begin{align*}
\Gamma_i(b\gamma_i,d'\eta_i)&=-\mu_i^{-2}\vartheta_i(b)^*\vartheta_i(d')\gamma_i^*=
\bigl(\mu_i^{-1}\vartheta_i(b)\bigr)^*\bigl(-\mu_i^{-1}\vartheta_i(d')\gamma_i^*\bigr)+0\\
&=\partial_i^\e(b\gamma_i)^*\partial_i^\e(d'\eta_i)+\partial_{-i}^\e(b\gamma_i)^*
\partial_{-i}^\e(d'\eta_i),\\
\Gamma_i(c\gamma_i^*,d'\eta_i)
&=\mu_i^2\vartheta_i(c)^*\vartheta_i(d')\gamma_i=0+\bigl(\mu_i\vartheta_i(c)\bigr)^*
\bigl(\mu_i\vartheta_i(d')\gamma_i\bigr)\\
&=\partial_i^\e(c\gamma_i^*)^*\partial_i^\e(d'\eta_i)+\partial_{-i}^\e(c\gamma_i^*)^* \partial_{-i}^\e(d'\eta_i).
\end{align*}
The reversed mixed terms follow from $\Gamma_i(X,Y)^*=\Gamma_i(Y,X)$.

It remains to consider the case $X=d\eta_i$ and $Y=d'\eta_i$. Since $\eta_i\cdot d^*d'=d^*d'\cdot \eta_i$, we have 
\[(d\eta_i)^*(d'\eta_i)=d^*d'\eta_i^2=d^*d'\left((\mu_i^2-\mu_i^{-2})\eta_i+{\rm Id}\right)\]
by \eqref{eqB5}. Therefore, by \eqref{eqB2},
\begin{align}
\label{eqB7}
\Gamma_i(d\eta_i,d'\eta_i)&=2d^*d'\eta_i^2-\frac12\wt N_i^\e(d^*d'\eta_i^2)=d^*d'\left(2\cdot {\rm Id}+(\mu_i^2-\mu_i^{-2})\eta_i\right).
\end{align}
On the other hand, by \eqref{eqB1}, \eqref{eqB3}, and \eqref{eqB5},
\begin{align*}
&\partial_i^\e(d\eta_i)^*\partial_i^\e(d'\eta_i)+\partial_{-i}^\e(d\eta_i)^*\partial_{-i}^\e(d'\eta_i)\\
&=\mu_i^{-2}\gamma_i\vartheta_i(d)^*\vartheta_i(d')\gamma_i^*+\mu_i^2\gamma_i^*\vartheta_i(d)^*\vartheta_i(d')\gamma_i\\
&=\mu_i^{-2}d^*d'\gamma_i\gamma_i^*+\mu_i^2d^*d'\gamma_i^*\gamma_i =d^*d'\left(2\cdot {\rm Id}+(\mu_i^2-\mu_i^{-2})\eta_i\right).
\end{align*}
Hence, $\Gamma_i(X,Y)=\partial_i^\e(X)^*\partial_i^\e(Y)+\partial_{-i}^\e(X)^*\partial_{-i}^\e(Y)$ holds for $X=d\eta_i$ and $Y=d'\eta_i$.

Lastly, by summing \eqref{eqB4} over all $1\leq i\leq n$, we can conclude that
\begin{align*}
\Gamma_P(X,Y)&=\sum_{i=1}^n\Gamma_i(X,Y)=
\sum_{i=1}^n\left(\partial_i^\e(X)^*\partial_i^\e(Y)+\partial_{-i}^\e(X)^*\partial_{-i}^\e(Y)\right) =\sum_{j\in[\pm n]}\partial_j^\e(X)^*\partial_j^\e(Y).
\end{align*}
\end{proof}

\bigskip

{\bf Statement on the Use of Artificial Intelligence.} The authors declare that no artificial intelligence (AI) or AI-assisted technologies were used in the development, derivation of proofs, mathematical verifications, or writing of this paper.

{\bf Acknowledgements.} The authors would like to thank Marius Junge, Javier Parcet, and Gilles Pisier for helpful discussions and comments concerning this problem, as well as for their encouragement.
H.H.Lee and Z.Xu were supported by the National Research Foundation of Korea (NRF) grant funded by the Korea government(MSIT) (No.RS-2022-NR069971).
Z.Xu and S.-G.Y. were supported by the National Research Foundation of Korea (NRF) grant funded by the Korea government(MSIT) (No.RS-2025-00561391).
S.-G.Y. was supported by the National Research Foundation of Korea (NRF) grant funded by the Korea government(MSIT) (No.RS-2024-00413957).

\end{document}